\documentclass[12pt]{amsart}

\usepackage[T1]{fontenc}

\usepackage{amssymb}
\usepackage{amsthm}

\usepackage{mathrsfs}

\usepackage[a4paper,%
centering]{geometry}

\usepackage{lscape}
\usepackage{rotating}
\usepackage{pdflscape}
\usepackage{rotating} 
\usepackage{afterpage} 

\usepackage{xcolor}
\usepackage{graphicx}

\usepackage{comment}

\usepackage{hyperref}

\usepackage{enumerate}

\usepackage{amsmath, amstext, amsopn, amsxtra, amsfonts}

\usepackage[shortlabels]{enumitem}

\usepackage[low-sup]{subdepth}

\usepackage{mathtools}

\usepackage{ifpdf}
\ifpdf
\usepackage[all,pdf,cmtip]{xy} 
\xyoption{rotate}
\else
\input xy
\xyoption{all}
\xyoption{v2}
\xyoption{cmtip}
\xyoption{rotate}
\fi

\DeclareMathOperator{\tr}{tr}

\DeclareMathOperator{\Ind}{Ind}
\DeclareMathOperator{\id}{id}

\DeclareMathOperator{\Ad}{Ad}

\DeclareMathOperator{\pr}{pr}
\DeclareMathOperator{\ev}{ev}

\DeclareMathOperator{\Mfld}{Mfld}

\DeclareMathOperator{\Act}{act}

\DeclareMathOperator{\String}{String}
\DeclareMathOperator{\Spin}{Spin}
\DeclareMathOperator{\Tracered}{Tr}

\theoremstyle{plain}
\newtheorem{theorem}{Theorem}[section]
\newtheorem{corollary}[theorem]{Corollary}
\newtheorem{lemma}[theorem]{Lemma}
\newtheorem{proposition}[theorem]{Proposition}

\theoremstyle{definition}
\newtheorem{definition}[theorem]{Definition}

\newtheorem{construction}[theorem]{Construction}

\theoremstyle{remark}
\newtheorem{example}[theorem]{Example}

\newtheorem{remark}[theorem]{Remark}

\numberwithin{equation}{section}

\newcommand{\cH}{{\mathcal H}}
\newcommand{\cG}{\mathcal{G}}

\newcommand{\cE}{{\mathcal E}}
\newcommand{\cF}{{\mathcal F}}

\newcommand{\ZZ}{{\mathbb Z}}

\renewcommand{\a}{\alpha}
\renewcommand{\b}{\beta}

\renewcommand{\d}{\delta}
\newcommand{\g}{\mathfrak g}
\newcommand{\fg}{\mathfrak g}

\newcommand{\mult}{\mathrm{mult}}

\DeclareMathOperator{\curv}{curv}

\newcommand\rightthreearrow{%
        \mathrel{\vcenter{\mathsurround0pt
                \ialign{##\crcr
                        \noalign{\nointerlineskip}$\rightarrow$\crcr
                        \noalign{\nointerlineskip}$\rightarrow$\crcr
                        \noalign{\nointerlineskip}$\rightarrow$\crcr
                }%
        }}%
}

\def\into {\hookrightarrow}

\newcounter{anote}

\newcommand{\intSone}{\int_{0}^{2\pi}\mathrm{d}\theta\,}
\newcommand{\Omegahat}[1]{\widehat{\Omega #1}}
\newcommand{\Thetahat}{\widehat{\Theta}}
\newcommand{\phihat}{\widehat{\phi}}
\newcommand{\lMC}[1]{#1^{-1}d#1}
\newcommand{\rMC}[1]{d#1\,#1^{-1}}

\newcommand{\phifun}[1]{#1^{-1}\partial #1}
\newcommand{\phihatfun}[1]{\partial #1\,#1^{-1}}
\newcommand{\Adhat}{\widehat{\Ad}}
\DeclareMathOperator{\inv}{inv}

\makeatletter
\providecommand{\leftsquigarrow}{%
  \mathrel{\mathpalette\reflect@squig\relax}%
}
\newcommand{\reflect@squig}[2]{%
  \reflectbox{$\m@th#1\rightsquigarrow$}%
}
\makeatother

\allowdisplaybreaks

\begin{document}



\title[Rigid 2-gerbes I]{Rigid models for 2-gerbes I: Chern--Simons geometry}
   
  \author[D.\ M.\ Roberts]{David Michael Roberts}
  \address[David Michael Roberts]
{Sirus-beta Labs\quad ||\quad 
 School of Computer and Mathematical Sciences\\
The University of Adelaide\\
  Adelaide, SA 5005 \\
  Australia
  }
  \email{droberts.65537@gmail.com}
  
  \author[R.\ F.\ Vozzo]{Raymond F.\ Vozzo}
\address[Raymond Vozzo]
{School of Computer and Mathematical Sciences\\
The University of Adelaide\\
  Adelaide, SA 5005 \\
  Australia}
\email{raymond.vozzo@adelaide.edu.au}

\thanks{The authors acknowledge the support of the Australian
Research Council, grants DP120100106, DP130102578 and DP180100383. 
Diagrams were set with \Xy-pic\ \xyversion.}


\begin{abstract}
Motivated by the problem of constructing explicit geometric string structures, we give a rigid model for bundle 2-gerbes, and define connective structures thereon. 
This model is designed to make explicit calculations easier in applications to physics. 
To compare to the existing definition, we give a functorial construction of a bundle 2-gerbe as in the literature from our rigid model, including with connections.
As an example we prove that the Chern--Simons bundle 2-gerbe from the literature, with its connective structure, can be rigidified---it arises, up to \emph{isomorphism} in the strongest possible sense, from a rigid bundle 2-gerbe with connective structure via this construction.
Further, our rigid version of 2-gerbe trivialisation (with connections) gives rise to trivialisations (with connections) of bundle 2-gerbes in the usual sense, and as such can be used to describe geometric string structures.

The preprint of this article is available as \href{https://arxiv.org/abs/2209.05521}{arXiv:2209.05521}.
\end{abstract}

\maketitle

\tableofcontents


\section*{Introduction}

Our goal with this paper---the first in a series of three---is to provide a new calculational framework for higher geometry and higher gauge theory. 
In particular we introduce here the novel concept of {\em rigid bundle $2$-gerbe} and illustrate by examples the way in which
rigidity simplifies calculations.
This will permit the level of explicit detail appearing in, for example, the classic paper \cite{BPST}, and provide a general framework that will recover results such as the explicit coordinate-level construction in \cite[\S 5]{RSW_22}.
In the forthcoming \cite{rb2g_III} we will use the framework of the present paper to give explicit global formulas for geometric string structures \cite{Wal} on infinite families of reductive homogeneous spaces.%
\footnote{To the best of our knowledge, the only known explicit formula in the literature for geometric string structures is given in \cite[\S 5.3]{RSW_22} for $S^4$. 
Note also that in \cite[\S 6b]{Redden}, one piece of the data arising from a geometric nontrivial string structure for $S^3$ is calculated.}

One model for understanding string structures on a $G$-bundle $Q$ and their associated geometry uses the objects known as \emph{2-gerbes}, originally introduced in a general context by \cite{Breen94} over an arbitrary Grothendieck site (as opposed to the site of a fixed smooth manifold), and then applied by \cite{BM} to a more geometric setting. 
In the differential-geometric category, \emph{bundle 2-gerbes} were initially sketched in \cite{Carey-Murray-Wang_97}, then expounded more fully in \cite{StevPhD,Ste}.
The details of general bundle 2-gerbes are not needed here (and a peek at Appendix \ref{app:weak_2-gerbes} will show how complex they are), but it suffices for now to know two things about them.
First, just as $U(1)$-bundles on a manifold $X$ are classified by $H^2(X,\ZZ)$ and bundle gerbes are classified by $H^3(X,\ZZ)$ \cite{Mur}, bundle 2-gerbes are classified by $H^4(X,\ZZ)$ \cite{Ste}.
Secondly, while the existence of a $\Spin^c$-structure is obstructed by a certain bundle gerbe classified by the characteristic class $W_3$---a chosen trivialisation of which \emph{is} a $\Spin^c$-structure---the existence of a string structure is obstructed by a certain bundle 2-gerbe classified by the characteristic class $\frac{1}{2} p_1$; a chosen trivialisation of this bundle 2-gerbe \emph{is} a string structure \cite{JohnsonPhD,CJMSW}.
This bundle 2-gerbe is called the \emph{Chern--Simons bundle 2-gerbe} in \emph{loc.\ cit.}, due to its holonomy over closed 3-manifolds giving the action for the Chern--Simons sigma-model attached to the original principal bundle.

It is important to note that one can either trivialise a bundle 2-gerbe purely (differential) topologically (see e.g.\ \cite{BM,Stolz-Teichner}) or trivialise it with geometric structure encoded by certain differential form data; the latter is what we are interested in here, and such \emph{geometric string structures} were introduced in \cite{Wal} (see also \cite{FiorenzaSatiSchreiber_14}).
This general approach has the advantage that it utilises the geometry of bundle (1-)gerbes, which is well understood and has been studied from multiple points of view. 
Chern--Simons bundle 2-gerbes have also appeared in many applications in physics and geometry (e.g.\ \cite{CJMSW,CareyWang_08,Sati_10,Wal,FiorenzaSatiSchreiber_14,Mickler_18}); results such as the Green--Schwarz anomaly cancellation in heterotic string theory can be encoded using a Chern--Simons bundle 2-gerbe \cite{Sati-Schreiber-Stasheff_12,Fiorenza-Sati-Schreiber_15}. 
The recent \emph{Hypothesis H} \cite{Fiorenza-Sati-Schreiber_20,Fiorenza-Sati-Schreiber_21b} gives rise to a variant of a Chern--Simons 2-gerbe encoding the full $C$-field data on M5-branes \cite{Fiorenza-Sati-Schreiber_21a}.

The purpose of this paper is to introduce new methods for studying the geometry and topology of bundle 2-gerbes. 
These methods are amenable to performing concrete global and local calculations, and thus widely applicable to problems in differential geometry and mathematical physics. 
Our definition of \emph{rigid} bundle 2-gerbes simplifies significantly the data involved, and we explain the differential geometry of these objects, as well as their trivialisations, together with \emph{their} geometry. 
It is this latter construction that encodes geometric string structures.
Much as bundle gerbes are a kind of categorification of abelian principal bundles, trivialisations of a Chern--Simons 2-gerbe correspond to categorified non-abelian bundles, that is \emph{non-abelian principal 2-bundles} \cite{Bartels} (see also \cite{Nikolaus-Waldorf_13}).
Therefore, methods for understanding the geometry of trivialisation of Chern--Simons 2-gerbes and for being able to perform related computations will also carry over to the geometry of principal 2-bundles, which at present has many proposed formalisms: \cite{B_M}, \cite{ACJ}, \cite{BS04}, \cite{Schreiber-Waldorf_13}, \cite{Waldorf_18}, \cite{Kim-Saemann_20,RSW_22}, etc. 
With concrete examples to hand, we will be able to give empirical answers to questions like: Is the ``fake flatness'' condition of \cite{BS04} something that arises in practice? Or can we get natural non-fake-flat examples? 
Can we clarify the role of the ``adjustment'' of \cite{Kim-Saemann_20}, which has been seen in at least one example, namely in \cite{RSW_22}? 
In the best case scenario, definitions should be driven by a natural class of examples; this was the case for bundle gerbe connections/curvings from \cite{Mur}, defined so as to make the topological term in the Wess--Zumino--Witten action to be the ``holonomy'' of some kind of geometric object, as considered in \cite{Carey-Murray_86} (Murray, personal communication).

\medskip

The rigid model that we introduce here is much simpler than the definition of bundle 2-gerbe in \cite{StevPhD}. 
For instance, we note that to prove the existence of appropriate connective structures (called \emph{rigid connective structures}) in Proposition \ref{prop:conn_exist} is far more straightforward than the corresponding proof for weak bundle 2-gerbes in \cite[Proposition 10.1]{StevPhD}. Similarly, the use of plain 2-gerbes (as by e.g.\ \cite{BM}) often involves a step of 2-stackification, which is rather nontrivial and leads to inexplicit results.
This is a theme throughout this paper: the relative ease with which one is able to establish the theory is in stark contrast to existing methods.

We also posit that rigid and explicit constructions are useful for applications to physics, where weak and inexplicit constructions do not give \emph{coordinate-level} descriptions of \emph{global} objects. 
Many existing constructions in higher gauge theory rely on universal properties in an $\infty$-topos, which only specifies data up to a contractible space of choices.
On the other hand, the concrete calculations in for example \cite{FiorenzaSchreiberStasheff_12}, leading to constructions such as the `brane bouquet' \cite{Fiorenza-Sati-Schreiber_15b}, are linearised and purely local in nature.

For an example of how the philosophy of seeking concrete constructions can be fruitfully applied, see the recent paper \cite{RSW_22} of Rist, Saemann and Wolf, developed in parallel with this one, and building partly on early work for this paper presented by the first author \cite{Roberts_14}.
Another result following from our work (to appear in \cite{rb2g_III}) is that we can enhance a theorem of Back \cite{Back_81} stating $p_1$ of $T(G/H)$ is torsion under certain mild conditions on $G$, $H$ and the Dynkin index $\Ind^D_{H\into G}$---proved using Chern--Weil representatives---to include the full 2-gerbe connective structure.
Our methods are constructive in that we explicitly build a trivialisation of the 2-gerbe classified by $\Ind^D_{H\into G}\cdot \frac12 p_1$.

Independently from and in parallel with the present work (which was long in gestation!), Kottke and Melrose \cite{KM19}\footnote{Please note that the original published version of this paper contains an error in the custom version of \v{C}ech cohomology introduced to classify bigerbes, fixed in v2 on the arXiv.} introduced a theory of ``bigerbes''.
The definition of bigerbe is similar to, but not identical with, our structure here, and at present only the topological version has been covered; connections were left for future work. 
The stated motivation in \cite{KM19} is analytic in nature, with an eye toward fusion structures and a Dirac operator on loop space, and does not consider geometry in the style of what is known for bundle gerbes.
It turns out that every bigerbe gives rise to a canonical rigid bundle 2-gerbe, and also a trivialised bigerbe gives rise to a canonical rigidly trivialised rigid bundle 2-gerbe; more details are given in Appendix~\ref{app:bigerbes}.

We also point out that our construction in Appendix~\ref{app:weak_2-gerbes} specifies a functor from rigid bundle 2-gerbes to bundle 2-gerbes as defined in the literature (eg \cite{Ste,Wal}). 
Having such a functor is stronger than relying on a common classification by degree-4 integral cohomology as mentioned earlier in the case of 2-gerbes, which is the approach in \cite{KM19}.
In \cite{rb2g_II} we will show that every 4-class is represented by a rigid bundle 2-gerbe, via constructing a universal smooth rigid bundle 2-gerbe.

\smallskip

\noindent \textbf{Outline of paper.}\quad  In Section 1 we review the necessary background on bundle gerbes. 
We also provide in this section new formulae for a connection and curving on the basic gerbe that are similar to, but simpler than, those given in \cite{MS03}.

In Section 2 we introduce the main objects that we study in this paper, rigid bundle 2-gerbes. 
We give the definition of these and describe the difference between our rigid model and the existing models in the literature via the notion of associated weak bundle 2-gerbes. 
There are quite a few technical details here that are relegated to Appendix~\ref{app:weak_2-gerbes} for readability purposes, but the basic ideas relating rigid 2-gerbes and weak 2-gerbes are described in this section at a high level. 
We also explain the rigid version of the concepts of triviality and of connective structure, both of which are necessary ingredients for the main examples and applications that motivate the constructions in this paper.

Section 3 contains the main results of the paper in the sense that we give a rigid model for the Chern--Simons bundle 2-gerbe and show that this is indeed a rigidification of the existing models in the literature. 
This section includes detailed, explicit computations of the geometry of this rigid 2-gerbe, which highlights the utility of our approach here. 

In addition to Appendix~\ref{app:weak_2-gerbes} on associated weak bundle 2-gerbes, containing the technical details of the correspondence between our model and the accepted model, we also include the short Appendix~\ref{app:bigerbes}, which explains the relationship between the bigerbes of \cite{KM19} and our model for 2-gerbes.\footnote{We thank the anonymous referee for encouraging us to pin this down.}
Finally, we include a table of data in Appendix \ref{app:data}, listing all of the specific differential forms and functions used throughout the paper for ease of reference.

\noindent \textbf{Convention.}\quad In this paper we use the conventions of Kobayahi and Nomizu for evaluating differential forms, which involve factors of $\frac12$ when evaluating 2-forms of the types $\alpha\wedge \beta$ \cite[page 35]{Kobayashi-Nomizu} and $d\alpha$ \cite[Proposition 3.11]{Kobayashi-Nomizu}, where $\alpha,\beta$ are 1-forms, on a given pair of tangent vectors $X,Y$.\footnote{That is:
\begin{align*}
    \alpha\wedge\beta(X,Y) &= \tfrac12(\alpha(X)\beta(Y) - \alpha(Y)\beta(X)),\\ 
    d\alpha(X,Y) & = \tfrac12(X\cdot \alpha(Y) - Y\cdot \alpha(X) - \alpha([\tilde X,\tilde Y]))
\end{align*}
for $\tilde V$ denotes any vector field extending a tangent vector $V$.}

\smallskip

\noindent \textbf{Acknowledgements.}\quad  Christian Saemann asked the question---about string structures on $S^5$---that sparked the line of inquiry leading ultimately to this paper, during a visit to Adelaide. 
CS also financially supported the first author during a visit to Heriot-Watt University in 2014 to speak \cite{Roberts_14} on early work related to this project (expanded in a forthcoming paper \cite{rb2g_III} in the series) at the Workshop on Higher Gauge Theory and Higher Quantization.
This work was also presented in 2015 at the Eduard \v{C}ech Institute for Algebra, Geometry and Physics, while visiting Urs Schreiber, who also provided support at the time.
General thanks to Michael Murray for helpful discussions on many occasions, including some of the prehistory of bundle gerbes, and support all-round; Ben Webster for discussion of Lie algebras; Dominik Rist, Christian Saemann and Martin Wolf for discussions around the definition of the central extension of the loop group and the connection thereon (based on the first arXiv version of \cite{RSW_22}); John Baez for gracious discussion around the sign errors in \cite{BSCS} and issuing a correction to that paper; and Kevin van Helden for sharing his lengthy private calculations (for the paper \cite{vanHelden_21}) around sign conventions for 2-term $L_\infty$-algebras and their weak maps, helping to confirm the correct sign choice for Definition~\ref{def:string_xmod}.
Lastly we thank the referee for their helpful suggestions, including establishing the construction of a (trivialised) rigid bundle 2-gerbe from a (trivialised) bigerbe.


\section{Background on bundle gerbes}

\subsection{Bundle gerbes: definitions}

In this section we will briefly review the various facts about bundle gerbes; more details can be found in \cite{Mur,MurSte,Murray_09}. The reader familiar with bundle gerbes can safely skip this subsection.

Let $X$ be a manifold and $Y \xrightarrow{\pi} X$ a surjective submersion. 
We will work in this paper with smooth Fr\'echet manifolds\footnote{Manifolds locally modelled on complete metrisable locally convex topological vector spaces.} in general, and differential forms are defined as in \cite[Appendix E]{Schmeding_book}. 
Examples include mapping spaces $C^\infty(K,N)$ for $K$ a compact smooth manifold (possibly with boundary) and $N$ smooth and finite-dimensional.
We denote by $Y^{[k]}$ the $k$-fold fibre product of $Y$ with itself, that is $Y^{[k]} = Y \times_X Y \times_X \cdots \times_X Y$.
This is a simplicial space whose face maps are given by the projections $d_i \colon Y^{[k]} \to Y^{[k-1]}$ which omit the $(i+1)^{th}$ factor.\footnote{In the case of $Y^{[2]}\to Y$, we may also write $\pr_i$, $i=1,2$ for projection \emph{onto} the $i^{th}$ factor.} 
We will assume familiarity with standard simplicial object definitions without further comment.
A \emph{bundle gerbe} $(E, Y)$ (or simply $E$ when $Y$ is understood) over $X$ is defined by a principal $U(1)$-bundle $E \to Y^{[2]}$ together with a bundle gerbe multiplication given by an isomorphism of bundles $d_0^*E \otimes d_2^*E \xrightarrow{\simeq} d_1^*E$ over $Y^{[3]}$, which is associative over $Y^{[4]}$. 
On fibres the multiplication looks like $E_{(y_2, y_3)} \otimes E_{(y_1, y_2)} \xrightarrow{\simeq}  E_{(y_1, y_3)}$ for $(y_1, y_2, y_3) \in Y^{[3]}$.

Another useful way of describing the bundle gerbe multiplication is to say that the $U(1)$-bundle $\delta E := d_0^*E \otimes d_1^*(E^*) \otimes d_2^*E \to Y^{[3]}$ has a section $m$. 
The bundle $\delta^2 E \to Y^{[4]}$ is canonically trivial and the associativity of the multiplication is equivalent to the induced section $\delta(m)$  of $\delta^2 E \to Y^{[4]}$ being equal to the canonical section.

The \emph{dual} of $(E, Y)$ is the bundle gerbe $(E^*, Y)$, where by
$E^*$ we mean the $U(1)$-bundle which is $E$ with the action of $U(1)$
given by $p \cdot z = p \, \bar z = p \, z^{-1}$ for $(p, z) \in E \times U(1)$; that is, as a space $E^*
= E$, but with the conjugate $U(1)$-action.

A bundle gerbe $(E, Y)$ over $X$ defines a class in $H^3(X, \ZZ)$, called the \emph{Dixmier--Douady class} of $E$.
Conversely, any class $c\in H^3(X,\ZZ)$ is the Dixmier--Douady class of a bundle gerbe over $X$.

We say that a bundle gerbe $(E, Y)$ is \emph{trivial} if there exists a $U(1)$-bundle $L \to Y$ such that $P$ is isomorphic to $ \d L := d_0^*L \otimes d_1^*(L^*)$ with the canonical multiplication given by the canonical section of $\d^2 L$. 
A choice of $L$ and an isomorphism $E \simeq \d L$ is called a \emph{trivialisation}. We call a trivialisation $L \to Y$ a \emph{strong trivialisation} if $L$ is the trivial bundle. We will see this case coming up frequently in later sections. 
The Dixmier--Douady class of $E$ is precisely the obstruction to $E$ being trivial.

An \emph{isomorphism} between two bundle gerbes, $(E, Y)\xrightarrow{\simeq} (F, Z)$,
over $X$ is a pair of maps $(\varphi,  f)$ where $f \colon Y \to Z$ is
an isomorphism that covers the identity on $X$, and $\varphi\colon E\simeq (f^{[2]})^* F$ is an isomorphism of $U(1)$-bundles respecting the bundle gerbe product.  
Isomorphism is in fact far too strong to be the right notion of equivalence for bundle gerbes, since there are many non-isomorphic bundle gerbes with the same Dixmier--Douady class. The correct notion of equivalence is stable isomorphism \cite{MurSte}, which has the property that two bundle gerbes are stably isomorphic if and only if they have the same Dixmier--Douady class. However, part of the advantage of our approach here is that we only ever \emph{need} to consider isomorphisms in the strongest sense as above.

Various constructions are possible with bundle gerbes. For example, we can pull back bundle gerbes in the following way: given a map $f \colon N \to X$ and a bundle gerbe $(E, Y)$ over $X$, we can
pull back the surjective submersion $Y\to X$ to a surjective submersion
$f^*Y \to N$ and the bundle gerbe $(E,Y)$ to a bundle gerbe $f^*(E,
Y) := \left( \left(f^{[2]}\right)^* E, f^*Y \right)$ over $N$, where
$f^{[2]} \colon f^*\big(Y^{[2]} \big) \to Y^{[2]}$ is the map induced by $\tilde{f} \colon f^*Y \to Y$. 

\begin{example}\label{ex:lifting}
Let $K$ be a Lie group and suppose $ U(1) \to \widehat{K} \to K$ is a central extension. 

If $P \to X$ is a $K$-bundle then there is a bundle gerbe $E$ over $X$ associated to the bundle $P$ and the extension $\widehat{K}$ called the \emph{lifting bundle gerbe}, whose Dixmier--Douady class is precisely the obstruction to the existence of a lift of $P$ to a $\widehat{K}$ bundle.
To construct this bundle gerbe, consider the simplicial manifold $NK_\bullet$, given by $NK_p = K^p$ whose face maps are given by 
\[
d_i(k_0, \ldots, k_p) = \begin{cases}
(k_1, \ldots, k_p), 				& i = 0,\\
(k_0, \ldots, k_{i-1} k_{i}, \ldots, k_p),	& 1 \leq i \leq p-1,\\
(k_0, \ldots, k_{p-1}),				&i = p.
\end{cases}
\]
As explained in \cite{BM} the central extension can be viewed as a $U(1)$-bundle $\widehat{K} \to K$ together with a section $s$ of $\delta \widehat{K} := d_0^*\widehat{K} \otimes d_1^* \widehat{K}^* \otimes d_2^* \widehat{K}$, such that the induced section $\d (s)$ of $\delta^2 \widehat{K}$ is equal to the canonical section. The surjective submersion of $E$ is given by $P \to X$. Notice that we have $P^{[k]} \simeq P \times K^{k-1}$ and so there is a map of simplicial manifolds $P^{[\bullet]} \to NK_\bullet$. We use this map to pull back the $U(1)$- bundle $\widehat{K}$ and the section of the pullback of $\d \widehat{K}$ over $P^{[3]}$ defines the bundle gerbe multiplication.

\end{example}

Bundle gerbes have rich geometry associated with them. A \emph{bundle gerbe connection} on $(E,Y)$ is a connection $\nabla$ on $E$ that respects the bundle gerbe product in the sense that if $m \colon Y^{[3]} \to \d E$ is the section defining the multiplication then $m^* \d \nabla = 0$, where $\d \nabla$ is the induced connection on $\d E$. This condition ensures that if $F_\nabla$ is the curvature of $\nabla$ then $ \d F_{\nabla} = 0$. 

The following Lemma is key to the geometry of bundle gerbes:

\begin{lemma}[\cite{Mur}]\label{lemma:key}
\label{lemma:fund_lemma}
Let $\pi\colon Y\to X$ be a surjective submersion. Then for all $p\geq 0$ the complex
\[
  0 \to \Omega^p(X)\xrightarrow{\pi^*} \Omega^p(Y) \xrightarrow{\delta} \Omega^p(Y^{[2]}) \xrightarrow{\delta} \Omega^p(Y^{[3]}) \xrightarrow{\delta} \cdots
\]
is exact, where $\delta\colon \Omega^p(Y^{[n]}) \to \Omega^p(Y^{[n+1]})$ is the alternating sum of pullbacks: $\xi\mapsto \sum_{i=0}^n (-1)^i d_i^*\xi$.
\end{lemma}

Lemma \ref{lemma:key} implies that there exists a 2-form $\beta \in \Omega^2(Y)$ such that $\d \beta = F_{\nabla}$. 
A choice of such a 2-form is called a \emph{curving} for $E$. 
Since $\d (d \beta) = d(\d \b) = dF_{\nabla} = 0$, Lemma \ref{lemma:key} again implies that $d \b$ descends to a closed 3-form $\omega$ on $X$. 
This 3-form is called the \emph{3-curvature} of $E$. 
The class of $\omega$ in de Rham cohomology represents the image in real cohomology of the Dixmier--Douady class of $E$.

When paired with the condition in the definition of a bundle gerbe connection, Lemma \ref{lemma:key} implies that given two bundle gerbe connections $\nabla$ and $\nabla'$ on a given bundle gerbe $(E,Y)$, there is a 1-form $a$ on $Y$ such that $\nabla = \nabla' + \pi^*\delta(a)$. 
Such a 1-form is not unique: any two choices differ by a 1-form pulled up to $Y$ from the base.

In all of the constructions above, the bundle gerbes can be enriched with connections and curvings. For example, if $E$ has bundle gerbe connection $\nabla_E$ then the dual $E^*$ has an induced connection $\nabla_{E^*}$ given by $-\nabla_E$. Similarly, if $f \colon N \to X$ and $E$ has a connection $\nabla_E$ and curving $\beta$, with 3-curvature $\omega$, then the pullback has an induced connection and curving, denoted $f^*\nabla$ and $f^*\beta $, respectively, and the 3-curvature of this connection and curving is $f^* \omega$.

We can also talk of trivialisations of bundle gerbes \emph{with connection}. In fact, the case we will use most often is that of a \emph{strong trivialisation with connection}: if $E$ has connection $\nabla_E$ and curving $\beta$, a strong trivialisation with connection is a strong trivialisation $L$ of $E$ together with a 1-form $a$ on $Y$ such that $\delta a = \ell^*\nabla_E$, and $\beta=da$. Here $\ell$ is the section of $E$ induced by the isomorphism $E \simeq \d L\simeq Y^{[2]}\times U(1)$.
Another viewpoint is that a strong trivialisation with connection is an isomorphism of bundle gerbes with connection (namely an isomorphism of bundle gerbes in the strongest sense, preserving connection and curving), to the trivial bundle gerbe $(Y^{[2]}\times U(1),Y)$ equipped with the trivial multiplication, bundle gerbe connection $\delta(a)+\theta^{-1}d\theta$, and curving $da$.

\subsection{Example: a new connection and curving for the basic gerbe}\label{sec:new_example}

We take this opportunity to work through an example calculation of a bundle gerbe connection and curving, which appears to be new, but which is particularly suited to the construction, in section \ref{subsec:conn_structure} below, of a connective structure on a rigid bundle \emph{2-gerbe}.

\begin{definition}\label{def:Killing_normalised}
On a semisimple Lie algebra $\g$ we define the normalised Cartan--Killing form $\langle-,-\rangle$ following \cite{BSCS}, namely the `basic inner product' of \cite[\S 4.4]{PS} on $\g$ divided by $4\pi$.
The basic inner product itself is defined to be such that a coroot in $\g$ corresponding to\footnote{Here a coroot is defined to be the image of a root under the isomorphism $\g^*\simeq \g$ induced by the Cartan--Killing form, following \cite{PS}, as opposed to a rescaled root.} a long root (in the dual of some Cartan subalgebra) has length squared equal to 2.
\end{definition}

\begin{example}\label{eg:normalised_Killing_forms}
 Here are the normalised Cartan--Killing forms for the classical matrix Lie algebras, adapted from \cite[p~583]{Wang-Ziller_85}:
\begin{itemize}
\item $\mathfrak{su}(n)$: $\langle X,Y\rangle = -\tr(XY)/4\pi$.
\item $\mathfrak{so}(n)$, $n\geq 4$: $\langle X,Y\rangle = -\tr(XY)/8\pi$.\footnote{The case of $\mathfrak{so(4)}$ is not included in the citation, not being simple, but one can calculate the length of the coroots using the Killing form and then normalise appropriately. 
Notice that the basic inner product on $\mathfrak{so}(5)$ restricts to the basic inner product on $\mathfrak{so}(4)$.}
\item $\mathfrak{sp}(n)$: $\langle X,Y\rangle = -\Tracered(XY)/4\pi = -\Re \tr_{\mathbb{H}}(XY)/2\pi$. 
\end{itemize}
For definiteness, all the Lie algebras here are the compact real forms, using antisymmetric matrices for the orthogonal algebras, the anti-Hermitian trace-free complex matrices for the unitary algebras, and the anti-Hermitian quaternionic matrices for the symplectic algebra. 
The trace $\tr_{\mathbb{H}}(\cdot)$ here is the ordinary matrix trace, and the \emph{reduced trace} $\Tracered(\cdot)$ for $n\times n$ quaternionic matrices is the composite of the embedding into $2n\times 2n$ complex matrices (thinking of $\mathbb{H}\simeq \mathbb{C}\oplus j\mathbb{C}$), and the ordinary trace of complex matrices.
\end{example}

Recall (see for example \cite{MS03}) that the \emph{basic gerbe} on a compact Lie group $G$ can be modelled as the lifting bundle gerbe associated to the path fibration of $G$ and the level 1 central extension of the loop group $\Omega G$. 

In more detail, define $PG = \{\gamma\in C^\infty([0,2\pi],G)\mid \gamma(0)=1\}$, the path fibration of $G$, and $\Omega G = \ker (\ev_{2\pi})$. Then, as per Example \ref{ex:lifting} we pull back the central extension $\widehat{\Omega G}$ of $\Omega G $ over $PG^{[2]}$. In fact, we will use the isomorphism of simplicial manifolds $PG^{[\bullet]} \simeq PG \times \Omega G^{\bullet - 1}$ mentioned above, so that the pullback is $\id\times \pi\colon PG\times \Omegahat{G} \to PG \times \Omega G$.
Note that we then have that $d_0\colon PG\times \Omega G \to PG$ sends $(p,\gamma)\mapsto p\gamma$, and corresponds to the second projection for $PG^{[2]}$. Likewise, $d_1$ is projection on the first factor.

The central extension $\pi\colon\Omegahat{G}\to \Omega G$, meanwhile, is equipped with a connection $\mu$ and is constructed from a pair of differential forms\footnote{In \cite{MS03} the 1-form $\nu$ is instead denoted $\alpha$.} $(R,\nu) \in \Omega^2(\Omega G) \times \Omega^1(\Omega G \times \Omega G)$ satisfying the equations \cite{MS03}
\[
dR		= 0, \qquad
\d R		= d \nu, \qquad
\d \nu		= 0.
\]
The general procedure for constructing a central extension from such a pair is described in \emph{loc.\ cit.} but the important points here are that the 2-form $R$ is the curvature of the connection $\mu$ and the 1-form $\nu$ satisfies $\nu = s^* \d \mu$, where, as in Example \ref{ex:lifting}, $s\colon  \Omega G \times \Omega G \to \d \widehat{\Omega G}$ is the section defining the multiplication on the central extension and $\d \mu$ is the induced connection on $\d \widehat{\Omega G}$.

\label{form:nu}
In the case of the loop group \label{form:central extension curvature} $R$ is the left-invariant 2-form given by $R_\gamma = \intSone \langle\Theta_\gamma,\partial \Theta_\gamma \rangle$
and $\nu$ is given by $\nu_{(\gamma,\eta)}= 2\intSone\langle\Theta_\gamma,\phihat_\eta\rangle$, for $\gamma, \eta \in \Omega G$ and $\Theta$ is the usual left-invariant Maurer--Cartan form. 
Here and throughout we define $\partial := \tfrac{\partial}{\partial \theta}$, applied pointwise to function-valued differential forms.
The function $\phihat$ is the restriction of the function $PG\to \Omega \g$ given by $p\mapsto \phihatfun{p}$, denoted by the same symbol by a slight abuse of notation. 
It is closely related to the \emph{Higgs field} $\phi_p = p^{-1} \partial p$ from \cite{MS03}, satisfying $\phihat_p = \Ad_p \phi_p$.

In order to give a bundle gerbe connection on $(PG \times\widehat{\Omega G}, PG)$ an obvious suggestion would be to use the pullback of the connection $\mu$ on $\widehat{\Omega G}$ to $PG\times \widehat{\Omega G} \to PG\times \Omega G$. 
However, $\mu$ is not a \emph{bundle gerbe} connection, precisely because\footnote{Recall that we are using principal $U(1)$-bundles, not line bundles, for our bundle gerbes, so that connections are 1-forms, and we can manipulate them as forms. Everything in this paper could be done with line bundles at the cost of reworking statements about connections to their equivalents in that setting.} $s^*\d \mu = \nu \neq 0$ (on $PG^{[3]}\simeq PG\times \Omega G^2$). 
Notice though, that as $\d \nu = 0$ (on $PG^{[4]}\simeq PG\times \Omega G^3$) there is a 1-form $\epsilon$ on $PG\times \Omega G$ such that $\d \epsilon = \nu$. 
Then the connection $\nabla = \pr_2^*\mu - \pi^*\epsilon$ satisfies $s^*\d \nabla = s^* (\d \mu -  \pi^*\d \epsilon) = \nu - \nu = 0$ and so $\nabla$ \emph{is} a bundle gerbe connection.

In \cite{MS03} it is explained how to produce such a 1-form $\epsilon$ for a lifting bundle gerbe in general, given a connection on the principal bundle to which it is associated, and indeed a connection and curving for the basic gerbe is given. 
The 3-curvature is then the standard invariant 3-form $\omega$ on $G$ representing a generator of $H^3(G,2\pi\ZZ)$ inside $H^3_{dR}(G)$. Or, equivalently, $\frac{1}{2\pi}\omega$ is a generator of $H^3(G,\ZZ)\subset H^3_{dR}(G)$.

Here, we shall define a \emph{different} connection and curving on the basic gerbe to what is possible using the general technique of \cite{MS03}, but, as we shall see, this connection and curving will have the same 3-form curvature at the one from \emph{loc.\ cit}.

\label{form:central extension connection}
The connection $\mu$ is given by the descent of the 1-form $\pr_2^*\theta^{-1}d\theta + \pr_1^*\int_{[0,2\pi]} \ev^* R$ along a certain quotient map $P\Omega G \ltimes U(1) \twoheadrightarrow \Omegahat{G}$, where $\ev\colon P\Omega G \times [0,2\pi] \to \Omega G$ is $(f,s) \mapsto f(s)$. 
Here $\int_{[0,2\pi]}$ is integration over the fibre of differential forms along the projection $P\Omega G \times [0,2\pi] \to P\Omega G$. 
For completeness, we recall the definition, chosen to be consistent with the constructions in \cite{MS03}.

\begin{definition}\label{def:fibre_integration}
Let $\xi$ be a $k$-form on $X\times [0,2\pi]$. The \emph{integration over the fibre} of $\xi$ along $X \times [0,2\pi] \to X$ is defined to be the $(k-1)$-form given at a point $x\in X$ by\footnote{Note that the factor of $k$ here is due to how the contraction $\iota_{T_t}\xi$ is evaluated in Kobayashi--Nomizu conventions.}
\[
  \left(\int_{[0,2\pi]} \xi\right)_{\!x}\!(X_1,\ldots,X_{k-1}) := \int_0^{2\pi}k\cdot \xi_{(x,t)}((0,T_t),(X_1,0),\ldots,(X_{k-1},0))\,\mathrm{d}t, 
\]
for all tangent vectors $X_1,\ldots,X_{k-1}\in T_xX$, where also $T_t:=\frac{\partial}{\partial t}$ is the standard basis tangent vector at $t\in [0,2\pi]$.
\end{definition}

We now give a bundle gerbe connection on the basic gerbe, which is different from that in \cite{MS03}. \label{form:basic gerbe connection} \label{form:basic gerbe connection fix-up form}

\begin{lemma}\label{lemma:connection}
The 1-form $\nabla := \pr_2^*\mu - \pi^*\epsilon$ 
where at $(p,\gamma)\in PG\times \Omega G$, we have
\begin{equation}\label{eq:eps_def}
  \epsilon_{(p,\gamma)} := 2\intSone \langle\Theta_p,\phihat_\gamma\rangle
\end{equation}
is a bundle gerbe connection on the basic gerbe.  
\end{lemma}

\begin{proof}
According to the discussion above, we need simply check that $\d \epsilon = \nu$. We have
\begin{align*}
\delta \epsilon_{(p,\gamma,\eta)} & = 2\intSone\langle
  \Theta_{p\gamma},\phihat_\eta\rangle 
  - \langle\Theta_p,\phihat_{\gamma\eta} \rangle
  + \langle\Theta_p,\phihat_\gamma
\rangle\\
&=2\intSone\langle 
  \Ad_{\gamma^{-1}}\Theta_p,\phihat_\eta \rangle
  + \langle \Theta_\gamma,\phihat_\eta \rangle
  - \langle \Theta_p,\phihat_\gamma  \rangle
  - \langle \Theta_p,\Ad_\gamma\phihat_\eta \rangle
  + \langle \Theta_p,\phihat_\gamma \rangle
\\
&=2\intSone\langle \Theta_\gamma,\phihat_\eta\rangle \\
& = \nu.
\end{align*}
Thus $\nabla$ is a bundle gerbe connection for the basic gerbe.
\end{proof}

\label{form: basic gerbe curving}
\begin{lemma}\label{lemma:curving}
A curving for the connection in Lemma~\ref{lemma:connection} is given by 
\[
  B_p := \intSone \langle \Theta_p, \partial \Theta_p \rangle .
\]
\end{lemma}

\begin{proof}
We need to check that $\delta B_{(p,\gamma)} = B_{p\gamma} - B_p= R_{\gamma} - d\epsilon_{(p,\gamma)} $. 

In the calculation below we make use of the identities
\begin{align*}
\partial(\Ad_{\gamma^{-1}}\Theta_p) 
	&= \partial(\gamma^{-1}\Theta_p \gamma)\\
	& = \Ad_{\gamma^{-1}}\partial\Theta_p + \gamma^{-1}(\Theta_p\partial\gamma \gamma^{-1})\gamma- \gamma^{-1}(\partial\gamma \gamma^{-1}\Theta_p)\gamma\\
	& = \Ad_{\gamma^{-1}}(\partial\Theta_p + [\Theta_p,\phihat_\gamma]),\\
\Ad_\gamma\partial\Theta_\gamma 
	&= d\phihat_\gamma,\\ 
\intertext{and} 
\Ad_\gamma\partial\Theta_\gamma -[\Thetahat_\gamma,\phihat_\gamma] &= \partial\Thetahat_\gamma,
\end{align*} 
where $\Thetahat_g = \rMC{g} = \Ad_{g} \Theta_g\in \Omega^1(G,\g)$ denotes the right-invariant Maurer--Cartan form. 
The last identity follows from:
\begin{align*}
\partial(\Ad_\gamma\Theta_\gamma) & = \partial\gamma \gamma^{-1}\gamma\Theta_\gamma\gamma^{-1} + \Ad_\gamma\partial\Theta_\gamma - \gamma\Theta_\gamma\gamma^{-1}\partial\gamma\gamma^{-1}\\
&=\Ad_\gamma\partial\Theta_\gamma + [\phihat_\gamma,\Thetahat_\gamma]\\
& = \Ad_\gamma\partial\Theta_\gamma - [\Thetahat_\gamma,\phihat_\gamma].
\end{align*}

Now we calculate $\d B$ at the point $(p,\gamma)\in PG\times \Omega G$:
\begin{align*}
&B_{p\gamma} - B_p \\
& = \intSone 
\langle \Theta_{p\gamma}, \partial \Theta_{p\gamma}\rangle 
- \langle\Theta_p, \partial \Theta_p \rangle \\
& = \begin{multlined}[t]\intSone \langle \Ad_{\gamma^{-1}}\Theta_p, \partial(\Ad_{\gamma^{-1}}\Theta_p)\rangle 
+ \langle\Ad_{\gamma^{-1}}\Theta_p, \partial\Theta_\gamma\rangle \\
+ \langle\Theta_\gamma,\partial(\Ad_{\gamma^{-1}}\Theta_p)\rangle
+ \langle\Theta_\gamma,\partial \Theta_{\gamma}\rangle 
- \langle\Theta_p, \partial \Theta_p \rangle 
\end{multlined}\\
& = \begin{multlined}[t]
R_\gamma + \intSone \langle \Ad_{\gamma^{-1}}\Theta_p, \partial(\Ad_{\gamma^{-1}}\Theta_p)\rangle 
+ \langle\Ad_{\gamma^{-1}}\Theta_p, \partial\Theta_\gamma\rangle \\
+ \langle\Theta_\gamma,\partial(\Ad_{\gamma^{-1}}\Theta_p)\rangle
- \langle\Theta_p, \partial \Theta_p\rangle
\end{multlined} \\
& =  \begin{multlined}[t]R_\gamma + \intSone 
\langle  \Theta_p,\partial\Theta_p + [\Theta_p,\phihat_\gamma]\rangle 
+\langle\Ad_{\gamma^{-1}}\Theta_p,\partial\Theta_\gamma \rangle \\
+\langle\Theta_\gamma,\Ad_{\gamma^{-1}}(\partial\Theta_p + [\Theta_p,\phihat_\gamma])\rangle 
- \langle\Theta_p,\partial\Theta_p\rangle
\end{multlined} \\
 & = R_\gamma + \intSone 
 \langle \Theta_p,[\Theta_p,\phihat_\gamma] + \Ad_{\gamma^{-1}}(\Theta_p),\partial\Theta_\gamma\rangle
+\langle\Thetahat_\gamma,\partial\Theta_p \rangle
+ \langle\Thetahat_\gamma,[\Theta_p,\phihat_\gamma]\rangle \\
 & = R_\gamma + \intSone 
 \langle [\Theta_p,\Theta_p],\phihat_\gamma\rangle
 + \langle\Theta_p ,\Ad_{\gamma}\partial\Theta_\gamma \rangle
 + \langle\Thetahat_\gamma,[\Theta_p,\phihat_\gamma] \rangle
+ \langle\Thetahat_\gamma, \partial\Theta_p\rangle \\
 & = \begin{multlined}[t]
 R_\gamma + \intSone 
  \left(-2\langle d\Theta_p,\phihat_\gamma \rangle
  + 2\langle\Theta_p,d\phihat_\gamma \rangle \right)\\
   +\intSone
  \langle -\Theta_p,\Ad_\gamma\partial\Theta_\gamma\rangle
 + \langle\Theta_p,[\Thetahat_\gamma,\phihat_\gamma]\rangle
  +\langle \Thetahat_\gamma,\partial\Theta_p\rangle 
 \end{multlined} \\
 & = R_\gamma - d\epsilon_{(p,\gamma)} +\intSone
 \langle \Ad_\gamma\partial\Theta_\gamma-[\Thetahat_\gamma,\phihat_\gamma],\Theta_p\rangle 
  + \langle\Thetahat_\gamma\partial\Theta_p\rangle \\
 &= R_\gamma - d\epsilon_{(p,\gamma)} 
 +\intSone
 \langle \partial\Thetahat_\gamma,\Theta_p\rangle
  + \langle\Thetahat_\gamma,\partial\Theta_p\rangle \\
&= R_\gamma - d\epsilon_{(p,\gamma)} + \Big[\langle\Thetahat_\gamma,\Theta_p\rangle\Big]_0^{2\pi}\\
& = R_\gamma - d\epsilon_{(p,\gamma)}, 
\end{align*}
as needed.
\end{proof}

\begin{lemma}\label{lemma:basic_gerbe_curvature}
The 3-form curvature of $B$ is equal to the standard 3-form on $G$, namely for $\omega = \frac16\langle [\Theta_G,\Theta_G],\Theta_G\rangle $ we have $\ev_{2\pi}^*\omega = dB$.
\end{lemma}

\begin{proof}
We calculate $d B$:
\begin{align*}
dB_p & = \intSone 
\langle d\Theta_p, \partial \Theta_p\rangle
 -\langle \Theta_p, d\partial\Theta_p \rangle \\
& = \intSone \langle d\Theta_p, \partial \Theta_p\rangle 
- \langle\Theta_p, \partial d\Theta_p\rangle \\
& = \intSone \langle -\frac12[\Theta_p,\Theta_p], \partial \Theta_p\rangle 
+\frac12 \langle\Theta_p, \partial [\Theta_p,\Theta_p]\rangle \\
& = \intSone -\frac12\langle [\Theta_p,\Theta_p], \partial \Theta_p\rangle 
+\frac12 \langle\Theta_p, [\partial \Theta_p,\Theta_p]\rangle 
+ \frac12 \langle\Theta_p, [\Theta_p,\partial\Theta_p]\rangle \\
& = \frac12\intSone \langle [\Theta_p,\Theta_p], \partial \Theta_p\rangle.
\end{align*}
We now calculate this last term, in reverse, using the fact $\Theta$ is $P\mathfrak{g}$-valued so that evaluation at $0$ vanishes:
\begin{align*}
\langle [\Theta_{p(2\pi)},\Theta_{p(2\pi)}],\Theta_{p(2\pi)}\rangle  & = \intSone \partial \langle [\Theta_p,\Theta_p], \Theta_p\rangle \\
& = \intSone\langle [\partial\Theta_p,\Theta_p], \Theta_p\rangle 
+ \langle[\Theta_p,\partial\Theta_p], \Theta_p\rangle 
+ \langle[\Theta_p,\Theta_p], \partial\Theta_p\rangle \\
& = 3\intSone \langle [\Theta_p,\Theta_p], \partial \Theta_p\rangle.
\end{align*}
Thus we have $dB_p = \frac16\langle [\Theta_{p(2\pi)},\Theta_{p(2\pi)}],\Theta_{p(2\pi)}\rangle $ for all $p\in PG$, which gives $dB = \ev_{2\pi}^*\omega$, as needed.
\end{proof}

Now recall that we have an extra factor of $\frac{1}{4\pi}$ built into the normalised Cartan--Killing form $\langle -,-\rangle $ from Definition \ref{def:Killing_normalised}, so in terms of the basic inner product $\langle-,-\rangle_\text{basic} = 4\pi\langle -,-\rangle$ of \cite{PS}, this gives $\omega = \frac{1}{24\pi}\langle[\Theta^G,\Theta^G],\Theta^G\rangle_\text{basic}$.
Further, we know that the de Rham curvature of a bundle gerbe has periods living in $2\pi\mathbb{Z}$ \cite{Mur}, so to get an \emph{integral} form, we need to further divide by a factor of $2\pi$, so that as an integral 3-form, the 3-curvature we find is $\frac{1}{48\pi^2}\langle[\Theta^G,\Theta^G],\Theta^G\rangle_\text{basic}$, which is (the image in de Rham cohomology of) the generator of  $H^3(G,\mathbb{Z})$.

The previous Lemmata together prove:
\begin{proposition}
The pair $(\nabla,B)$ form a bundle gerbe connection and curving on the basic gerbe with identical 3-form curvature as the one from \cite{MS03}.
\end{proposition}

\begin{remark}\label{rem:MS_epsilon_correction}
The difference between the connection and curving given here, and the one from \cite{MS03} is that there is a 1-form $\varepsilon_{MS}=2\intSone \tfrac{\theta}{2\pi} \langle\ev^*\Thetahat_G,\phi\rangle$ on $PG$ such that the connection from \emph{loc.\ cit.} is $\nabla+\pi^*\delta(\varepsilon_{MS})$, and the curving is $B+d\varepsilon_{MS}$. Everything in this article works using the connection and curving $(\nabla+\pi^*\delta\varepsilon_{MS},B+d\varepsilon_{MS})$, but the expressions are longer and more complicated.
One conceptual way to view this is that the Murray--Stevenson construction is identical to the one here, up to tensoring with a trivial, flat bundle gerbe with connection $\theta^{-1}d\theta + \delta(\varepsilon_{MS})$ and curving $d\varepsilon_{MS}$.
\end{remark}


\section{Rigid models}

In \cite{Ste} two notions of bundle 2-gerbe are given, namely bundle 2-gerbes, and \emph{weak} (or `stable') bundle 2-gerbes. 
The weaker notion is now the default in the literature, and it is for this version one recovers the classification results and so on. A bundle 2-gerbe  \`a la Stevenson gives rise to a weak bundle 2-gerbe by replacing bundle gerbe morphisms by their associated stable isomorphisms.
However for the purposes of our computations, an even more rigid model is needed, even more so than Stevenson's stricter notion of bundle 2-gerbes.

\subsection{Rigid bundle 2-gerbes}

\begin{definition}
A \emph{bisemisimplicial manifold} $X_{\bullet,\bullet}$ consists of a family of manifolds $X_{n,m}$, $n,m\geq 0$ together with maps
\[
  d_i^h\colon X_{n,m}\to X_{n-1,m},\quad i=0,\ldots,n,\quad d_i^v\colon X_{n,m}\to X_{n,m-1},\quad i=0,\ldots,m,
\]
satisfying the following equations for all $0\leq i< j$:
\begin{align*}
d_i^h d_j^h = d_{j-1}^h d_i^h\\
d_i^v d_j^v = d_{j-1}^v d_i^v
\end{align*}

\end{definition}

Recall that for a bisemisimplicial manifold $Z$ we have two differentials acting on geometric objects, namely $\delta_h$ and $\delta_v$ via the usual alternating sum (or alternating tensor product, as appropriate). 
For the following examples, fix a bisemisimplicial manifold $X_{\bullet,\bullet}$.

\begin{example}
Given a differential form $\omega \in \Omega^k(X_{n,m})$, we have $\delta_h\omega = \sum (-1)^i (d_i^h)^*\omega \in \Omega^k(X_{n+1,m})$ and $\delta_v\omega = \sum (-1)^i (d_i^v)^*\omega \in \Omega^k(X_{n,m+1})$
\end{example}

\begin{example}
Given a principal $U(1)$-bundle $E\to X_{n,m}$, we have $\delta_h E = (d_0^h)^*E\otimes (d_1^h)^*E^* \otimes\ldots$, a bundle on $X_{n+1,m}$, and similarly for $\delta_v E \to X_{n,m+1}$.
Given a section $s$ of $E$, there are induced sections $\delta_h(s)$ and $\delta_v(s)$ of $\delta_h E$ and $\delta_v E$ respectively.
\end{example}

Notice that we have equality $\delta_h\delta_v\omega = \delta_v\delta_h\omega$, and a canonical isomorphism $\delta_h\delta_v E \simeq \delta_v\delta_h E$ induced by the universal property of the pullbacks.

\begin{definition}\label{def:vstrictb2g}
Let $X$ be a manifold. 
A \emph{rigid bundle 2-gerbe} on $X$ consists of the following data
\begin{enumerate}[(1)]

  \item A surjective submersion $Y \to X$.
  
  \item A 3-truncated semisimplicial space $Z_\bullet$ (that is $Z_n$ for $n\leq 3$) and surjective submersion of (truncated) semisimplicial spaces $Z_\bullet \to Y^{[\bullet+1]}$, i.e.\ a (truncated) semisimplicial map that is a surjective submersion $Z_n \to Y^{[n+1]}$ in each degree.

  \item A $U(1)$-bundle $E\to Z_1^{[2]}$ such that $E \rightrightarrows Z_1$ is a bundle gerbe on $Y^{[2]}$, with bundle gerbe multiplication $m\in \Gamma(\delta_vE)$.

  \item\label{b2g:mult}
   A section $M\in \Gamma(\delta_hE)$, called the \emph{bundle 2-gerbe multiplication};
   
   \item The section $M$ satisfies\footnote{The equalities here of spaces of sections are in fact canonical isomorphisms arising from the isomorphisms of bundles.} 
   \begin{enumerate}
      \item $\delta_v(M) = \delta_h(m)\in \Gamma(\delta_v\delta_hE) = \Gamma(\delta_h\delta_vE)$, and
      \item $\delta_h(M) = 1\in \Gamma(\delta_h^2E) = \Gamma(Z_3^{[2]}\times U(1)) = C^\infty(Z_3^{[2]}, U(1))$.
    \end{enumerate} 
\end{enumerate}

\end{definition}

The picture shown in Figure~\ref{fig:verystrictb2g} illustrates Definition~\ref{def:vstrictb2g}; the greyed-out parts are only relevant for Definition~\ref{def:vsb2g_triv} and can be ignored for now.

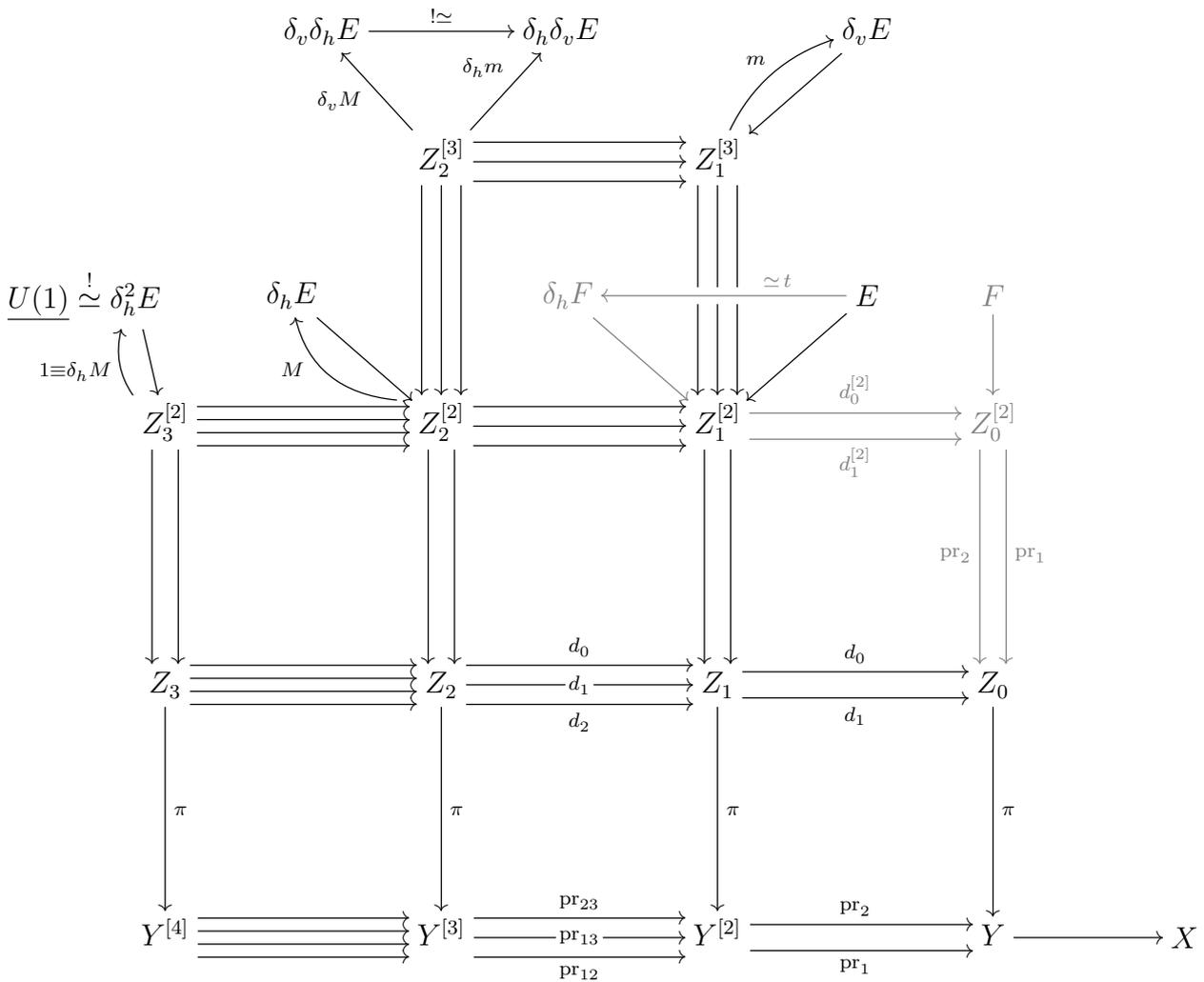
\begin{figure}[b]
\[
 \centerline{
  \xymatrix@=16ex{
    &
    \save 
      []+<-4em,10ex>  
      *+{\delta_v\delta_h E}="dvdh" 
      \ar@{<-}[]_-(.4){\delta_vM} 
    \restore
    \save 
      []+<4em,10ex> 
      *+{ \delta_h\delta_v E}="dhdv" 
      \ar@{<-}[]_-(.4){\delta_hm}
    \restore
    \ar@{->}"dvdh";"dhdv"^{!\simeq} 
    Z_2^{[3]}   
    \ar@<1.5ex>[r] \ar@<-1.5ex>[r] \ar[r] \ar@<1.5ex>[d] \ar@<-1.5ex>[d] \ar[d] 
    &
    Z_1^{[3]} 
    \save 
      []+<5em,10ex> 
      *+{\delta_v E} 
      \ar[] \ar@{<-}@/_1pc/[]_{m} 
    \restore 
    & 
    \\
    Z_3^{[2]} 
    \ar@<1ex>[d] \ar@<-1ex>[d] \ar@<1.5ex>[r] \ar@<-1.5ex>[r] \ar@<-0.5ex>[r]\ar@<0.5ex>[r]
    \save 
      []+<-1em,10ex> 
      *+[l]{\underline{U(1)}\stackrel{!}{\simeq}\delta_h^2 E} 
      \ar[] \ar@{<-}@/_1pc/[]_{1\equiv \delta_hM} 
    \restore 
    &
    Z_2^{[2]} 
    \ar@<1ex>[d] \ar@<-1ex>[d] \ar@<1.5ex>[r] \ar@<-1.5ex>[r] \ar[r] 
    \save
      []+<-5em,10ex> 
      *+{\delta_h E}  
      \ar[]+<-1em,2ex> \ar@{<-}@/_1pc/[]+<-1.5em,2ex>_(0.4){M} 
    \restore 
    &  
    \save 
      []+<5em,10ex> 
      *+{E}="2gerbe"  
      \ar[]+<1em,2ex>
    \restore 
    \save 
      []+<-5em,10ex> 
      *+{{\color{gray}\delta_h F}}="deltatrivbg"  
      \ar@*{[gray]}[]+<-1em,2ex>
    \restore 
    \ar@*{[gray]}@{->}"2gerbe";"deltatrivbg"_(.3){\color{gray}\simeq\,t} 
    \ar@<1.5ex>[u];[]|!{"2gerbe";"deltatrivbg"}{\hole}
    \ar@<-1.5ex>[u];[]|!{"2gerbe";"deltatrivbg"}{\hole} 
    \ar[u];[]|!{"2gerbe";"deltatrivbg"}{\hole}
    Z_1^{[2]}
    \ar@<1ex>[d] \ar@<-1ex>[d] \ar@*{[gray]}@<1ex>[r]^-{\color{gray}d_0^{[2]}} \ar@*{[gray]}@<-1ex>[r]_-{\color{gray}d_1^{[2]}}
    & 
    {\color{gray}Z_0^{[2]} }
    \ar@*{[gray]}@<-1ex>@<1ex>[d]^{\color{gray}\pr_1} \ar@*{[gray]}@<-1ex>@<-1ex>[d]_{\color{gray}\pr_2}
    \save 
      []+<0em,10ex> 
      *+{\color{gray}F}  
      \ar@*{[gray]}[] 
    \restore  
    \\
    Z_3 
    \ar[d]^{\pi} \ar@<1.5ex>[r] \ar@<-1.5ex>[r] \ar@<-0.5ex>[r]\ar@<0.5ex>[r]
    &
    Z_2 
    \ar@<1.5ex>[r]^{d_0} \ar[r]|{\,d_1\,} \ar@<-1.5ex>[r]_{d_2} \ar[d]^{\pi}
    & 
    Z_1 
    \ar@<-1ex>@<1ex>[r]^{d_0} \ar@<-1ex>@<-1ex>[r]_{d_1} \ar[d]^{\pi}
    &
    {Z_0} 
    \ar[d]^{\pi} 
    \\
    Y^{[4]}
    \ar@<1.5ex>[r] \ar@<-1.5ex>[r] \ar@<-0.5ex>[r]\ar@<0.5ex>[r] 
    &
    Y^{[3]}
    \ar@<1.5ex>[r]^{\pr_{23}} \ar[r]|{\,\pr_{13}\,} \ar@<-1.5ex>[r]_{\pr_{12}}
    & 
    Y^{[2]}
    \ar@<1ex>[r]^{\pr_2} \ar@<-1ex>[r]_{\pr_1}
    & 
    Y
    & 
    \save 
      []+<-2em,0pt> 
      *+{X} 
      \ar@{<-}[l]
    \restore
  }
 }
\]
\vspace{1cm}
\caption{A rigid bundle 2-gerbe on $X$. {\color{gray}(A trivialisation)}}\label{fig:verystrictb2g}
\end{figure}

The section $M$ is a \emph{strong} trivialisation of the induced bundle gerbe on $Y^{[3]}$, as opposed to an ordinary trivialisation, as would be expected from \cite[\S 7.1]{StevPhD}, even accounting for the different presentation of the bundle gerbe on $Y^{[3]}$.

\begin{proposition}\label{prop:rb2g_mult_associativity_square}
The condition (5)(b) in Definition \ref{def:vstrictb2g} is equivalent to the isomorphism $M'\colon d_0^* E\otimes d_2^*E \to d_1^* E$ of $U(1)$-bundles on $Z_2^{[2]}$ (corresponding to the section $M$) satisfying:
\begin{enumerate}
\setcounter{enumi}{4}
\item \begin{itemize}
\item[(b')] The diagram
\[
  \xymatrix{
    E_{23}\otimes E_{12} \otimes E_{01} \ar[r] \ar[d] & E_{13} \otimes E_{01} \ar[d] \\
    E_{23} \otimes E_{02}\ar[r] & E_{03}
  }
\]
of $U(1)$-bundles on $Z_3^{[3]}$ commutes, where $E_{ab}$ is the pullback along the map $Z_3^{[2]} \to Z_1^{[2]}$ induced by the map $(a,b)\into (0,1,2,3)$ in the simplex category $\Delta$.
 \end{itemize}
\end{enumerate}
\end{proposition}

\begin{proof}
The proof follows the same reasoning as, for example, \cite[Theorem 5.2]{BM}, following an idea of Grothendieck.
\end{proof}

\begin{remark}\label{rem:compare_to_MRSV}
The definition of bundle 2-gerbe given as \cite[Definition 6.5]{MRSV} is both more and less general than Definition \ref{def:vstrictb2g}: it restricts to a specific choice of semisimplicial surjective submersion $Z_\bullet \to Y^{[\bullet+1]}$ (one induced from the surjective submersion $Z_1\to Y^{[2]}$ in a specific way), but also allows for a trivialisation of the bundle gerbe on $Y^{[3]}$ that is not strong.
However, it is somewhat remarkable that for the main example in the literature, the Chern--Simons 2-gerbe, the rigid definition here is sufficient.
\end{remark}

\begin{remark}
A structure similar to that in Definition~\ref{def:vstrictb2g} is considered by \cite{KM19}, built using a bisimplicial space arising from a suitable square of maps of manifolds (we outline the details in Appendix~\ref{app:bigerbes}). 
In the notation of the current paper, the semisimplicial manifold $Z_\bullet$ would be the \v{C}ech nerve of a `horizontal' submersion parallel to $Y\to X$ in a certain commutative square. 
The key example considered by \cite{KM19} in their framework under the name `Brylinski--McLaughlin bigerbe' fits naturally in our framework, and corresponds to the `Chern--Simons bundle 2-gerbe' in the terminology of \cite{CJMSW}.
However, the point of view in \cite{KM19} is more geared towards \emph{fusion} products on loop spaces, whereas in this article we are concerned with more \emph{pointwise} products on loop groups and explicit computations with connective structures.\footnote{We recall that fusion products amount to composition in the \v{C}ech groupoid $PG^{[2]} \rightrightarrows PG$, namely $(\gamma_1,\gamma_2)(\gamma_2,\gamma_3) = (\gamma_1,\gamma_3)$. 
Often in the literature, especially when dealing with merely continuous paths, the identification is made that $PG^{[2]}$, the space of paths from the identity element that meet at their other endpoints, is `the same' as (a version of) $\Omega G$ the space of loops at the identity.
One makes a loop by following the first path, then the reverse of the other path, and the composition above `fuses' loops along half their length, and removes the detour back to the basepoint along $\gamma_2$.
However, in this paper, we replace this \v{C}ech groupoid by the isomorphic \emph{action} groupoid $PG\times \Omega G \rightrightarrows PG$, and here the composition is implemented by $(\gamma, \eta_1)(\gamma\eta_1,\eta_2) = (\gamma,\eta_1\eta_2)$, where of course $\eta_1\eta_2(t) = \eta_1(t)\eta_2(t)$.}
For the sake of completeness, we give the construction of a rigid bundle 2-gerbe (resp. rigidly trivialised rigid bundle 2-gerbe) from a bigerbe (resp.\ a trivialised bigerbe) in Appendix~\ref{app:bigerbes}.
\end{remark}

\begin{example}\label{example:rb2g_constructions}
There are several ways to get new rigid bundle 2-gerbes from old:
\begin{enumerate}[1.]
\item
  Given a rigid bundle 2-gerbe on $X_1$, and a map $f\colon X_2\to X_1$, we can pull everything along the map $f$ to get a rigid bundle 2-gerbe on $X_2$.
\item
  Given a rigid bundle 2-gerbe on $X$, with a surjective submersion $Y\to X$ and a semisimplicial surjective submersion $Z_\bullet \to Y^{[\bullet+1]}$, a second surjective submersion $Y'\to X$ and a map $Y'\to Y$ over $X$, we can pull back the semisimplicial submersion to form $Z_\bullet\times_{Y^{[\bullet+1]}} Y'^{[\bullet+1]} \to Y^{[\bullet+1]}$, then pull back the $U(1)$-bundle to $Z_1^{[2]}\times_{Y^{[2]}} Y'^{[2]}$. This gives a new bundle gerbe on $Y^{[2]}$, and the data of the section $M$ also pulls back to give a new rigid bundle 2-gerbe.
\item
  Given a rigid bundle 2-gerbe on $X$, with a surjective submersion $Y\to X$ and a semisimplicial surjective submersion $Z_\bullet \to Y^{[\bullet+1]}$, given a map of semisimplicial manifolds $Z'_\bullet \to Z_\bullet$ over $Y^{[\bullet+1]}$, there is a morphism of bundle gerbes involving $Z'_1\to Z_1$ over $Y^{[2]}$, and a rigid 2-gerbe using $Z'_\bullet$.
\item
  Given two rigid bundle 2-gerbes on $X$ involving the same surjective submersion $Y\to X$ and semisimplicial surjective submersion $Z_\bullet \to Y^{[\bullet+1]}$, with $U(1)$-bundles $E\to Z_1^{[2]}$ and $E'\to Z_1^{[2]}$, there is a \emph{product} rigid bundle 2-gerbe with $U(1)$-bundle $E\otimes E'$.
\item
  Given a rigid bundle 2-gerbe on $X$, with $U(1)$-bundle $E$, there is another with $U(1)$-bundle $E^*$, the \emph{dual}.
\end{enumerate}
All of these give examples of morphisms of rigid bundle 2-gerbes in a strict sense, but we will not be needing any general theory about these here, nor any notion of weak morphism analogous to stable isomorphisms.
\end{example}

We can, however, give a somewhat more concrete special example than these generic constructions.
Fix a manifold $X$ with an open cover by two open sets $U_0,U_1$. Two examples to keep in mind are $X=S^4$, or $X = M_0 \# M_1$ where $M_1$ and $M_2$ are compact 4-manifolds, and where $U_i \simeq M_i \setminus D^4$. 
In these latter examples one has $U_0\cap U_1 \simeq S^3 \times (-\varepsilon,\varepsilon)$, which allows for an even more concrete discussion.

\begin{example}
Given $X,U_0,U_1$ as above, define $U_{01} = U_0\cap U_1$.
If we have a bundle gerbe $(E,Z)$ on $U_{01}$ then we can build a rigid bundle 2-gerbe on $X$ in the following way.
Define $Y = U_0\amalg U_1$, and then $Y^{[2]} = U_{01} \amalg U_{10} \amalg U_0 \amalg U_1=U_{01} \amalg U_{10} \amalg Y$, where $U_{10} = U_{01}$, but the two projections $U_{10} \to Y$ are swapped from the projections $U_{01} \to Y$, namely $d_0\colon U_{01}\to U_1 \into Y$, and $d_1\colon U_{01}\to U_0 \into Y$.
Similarly,
\[
  Y^{[3]} = U_{001} \amalg U_{010}\amalg U_{100} \amalg 
  U_{011} \amalg U_{101}\amalg U_{110} \amalg  U_0 \amalg U_1
\]
where each $U_{ijk}\simeq U_{01}$, and the face maps restrict on each $U_{ijk}$ to omit the relevant index, for example $d_0\colon U_{001}\to U_{01}\into Y$, $d_1\colon U_{001}\to U_{01}\into Y$ and $d_2\colon U_{001}\to U_{00}=U_0\into Y$, where we adopt the convention that any $U_{i\cdots i} = U_i$, if only one index is present.
The reader can now safely calculate $Y^{[4]}$ on their own.

To define the semisimplicial space $Z_\bullet$, we take $Z_0=Y$, then $Z_1 = Z\amalg Z \amalg Y$, projecting down to $Y^{[2]}$ as $Z\amalg Z \to U_{01}\amalg U_{10}$ on the first two disjoint summands, and then the identity map on the rest. The two face maps $Z_1 \to Z_0$ are given by composition of the projection $Z_1\to Y^{[2]}$ and the face maps $Y^{[2]} \to Y$.
Then $Z_2$ is defined as 
\[
  Z_2 = Z \amalg Z \amalg Z \amalg Z \amalg Z \amalg Z \amalg U_0 \amalg U_1
\]
again with the summand-wise submersions to suitable copies of $U_{01}$ or $U_{10}$ (as appropriate), or $\id$ on the last two. The face maps $d_i\colon Z_2\to Z_1$ are defined so as to make the projections to $Y^{[n]}$ a semisimplicial map.
Again, a similar construction will work for $Z_3$, if the reader is diligent with the combinatorics.

Now we have a bundle gerbe on $Y^{[2]}$ that on $U_{01}$ is the given bundle gerbe $(E,Z)$, on $U_{10}$ is the dual $(E^*,Z)$, and on $U_i$ is the trivial gerbe associated to the identity submersion\footnote{Recall that for a manifold $N$ such a trivial trivial gerbe is $(N\times U(1),N)$, taking the trivial bundle on $N\times_N N = N$ and the identity map on the trivial bundle as the bundle gerbe multiplication.}. Denote by $\mathcal{E}$ the resulting $U(1)$-bundle on $Z_1^{[2]}$. We need to supply a bundle 2-gerbe multiplication, namely a section $M$ of $\delta_h(\mathcal{E})$.
But $\delta_h(\mathcal{E})$ restricted to one of the copies of $Z \subset Z_2$ is $E\otimes E^*$ and on either $U_i$ is the trivial bundle, so we take the canonical section in all cases.

Lastly, the section $M$ needs to satisfy $\delta(M) = 1$, relative to the canonical trivialisation of $\delta_h\delta_h(\mathcal{E})$. But this canonical trivialisation is the tensor product of copies of the canonical section of $E\otimes E^*$, which is exactly how $M$ is built.
\end{example}

\begin{remark}
In the case that $U_{01}\simeq S^3\times  (-\varepsilon,\varepsilon)$ (e.g. $X$ a connected sum of 4-manifolds, or $X=S^4$), then we know the classification of bundle gerbes tells us that stable equivalence classes of them on $U_{01}$ are in bijection with $H^3(S^3,\ZZ) = \ZZ$.
The above method then permits the construction of infinitely many rigid bundle 2-gerbes on $S^4$ and on a connected sum of 4-manifolds more generally. 
The question of equivalence (or not) of the members of this family of rigid bundle 2-gerbes in some weak sense will be left for future considerations. 
At a cohomological level, and given the classification of bundle 2-gerbes (\`a la Stevenson), the Mayer--Vietoris sequence tells us the construction should recover a rigid model for every integral 4-class on $S^4$, at least.
All the bundle gerbes on $S^3\simeq SU(2)$ can be constructed very explicitly, for instance using the powers of the basic gerbe from Section \ref{sec:new_example}, namely replace the universal central extension of $\Omega SU(2)$ by ${\widehat{\Omega SU(2)}{}^{\otimes k}} \to \Omega SU(2)$ and make the obvious minimal changes to the rest of the structure.
This ``clutching'' process on $S^4$ then would seem to give a rigid model for each bundle 2-gerbe on $S^4$, up to stable isomorphism.
Once the classifying theory is developed in \cite{rb2g_II} it will be possible to return to this question.
\end{remark}

\subsection{Trivialisations of rigid bundle 2-gerbes}

We are of course particularly interested in what it means for a rigid bundle 2-gerbe to be `trivial'. 
The triviality of the Chern--Simons 2-gerbe (described in Section~\ref{sec:CS2-gerbe}) of the frame bundle of a Spin manifold is what ensures the existence of a string structure. 
This idea goes back to the work in the 1980s of Killingback \cite{Killingback_87} and Witten \cite{Witten_85}, and developed more precisely in \cite{McL_92} and the geometric picture here has its roots in the use of 2-gerbes in \cite{BM}. Waldorf \cite{Wal} showed that the definition of \cite{Stolz-Teichner} could be reinterpreted in terms of trivialisations of bundle 2-gerbes, following \cite{CJMSW}. For more discussion, see the introduction of \cite{Wal}.

\begin{definition}\label{def:vsb2g_triv}
In the notation of the definition above, a \emph{trivialisation} is the data of 
\begin{enumerate}[(1)]
\item A $U(1)$-bundle $F \to Z_0^{[2]}$ such that $F\rightrightarrows Z_0$ is a bundle gerbe;
\item An isomorphism $t\colon E \xrightarrow{\simeq} \delta_h(F)$, compatible with the bundle gerbe multiplications of $(E,Z_1)$ and $(\delta_h(F),Z_1)$;
\item The section $\delta_h(t)\circ M$ of $\delta_h^2(F)\to Z_2^{[2]}$ is equal to the canonical section $c_{\delta^2F}$.
\end{enumerate}
\end{definition}

One can see the isomorphism $t$ as arising from a section $\tau$ of $E\otimes \delta_h(F)^*$ that is compatible with the bundle gerbe structure, and hence is a strong trivialisation. 
The last condition on then tells us that up to the canonical isomorphism $\delta_h(E)\simeq \delta_h(E\otimes \delta_h(F)^*)$, the strong trivialisations $M$ and $\delta_h(\tau)$ agree.

A third perspective is suggested by the following result.

\begin{proposition}\label{prop:module_version_of_rigid_triv}
The data  (2) and the condition (3) in Definition~\ref{def:vsb2g_triv} are equivalent to:
\begin{enumerate}[(1')]
 \setcounter{enumi}{1}
 \item An isomorphism $t'\colon E\otimes d_1^*F \to d_0^*F$ of $U(1)$-bundles on $Z_1^{[2]}$, compatible with the bundle gerbe multiplications (as they both are data for bundle gerbes on $Y^{[2]}$);

 \item The diagram
 \[
  \xymatrix{
    d_0^*E\otimes d_2^* E \otimes d_2^*d_1^*F \ar[r]^-{\id\otimes d_2^*t'} \ar[d]_{M'\otimes\id} & d_0^*E \otimes d_2^* d_0^* F \ar[d]^{d_0^*t'} \\
    d_1^* E\otimes d_2^* d_1^* F \ar[r]_-{d_1^*t'} & d_0^* F
  }
 \]
 of $U(1)$-bundle isomorphisms over $Z_2^{[2]}$ commutes (where we have used some simplicial identities implicitly, and $M'$ is the isomorphism corresponding to the section $M$).
\end{enumerate}
\end{proposition}

\begin{proof}
The argument follows the same line of reasoning as in the proof of Proposition~\ref{prop:rb2g_mult_associativity_square}.
\end{proof}

\begin{example}\label{example:triv_rb2g_constructions}
As in Example \ref{example:rb2g_constructions}, constructions are possible with trivialisations:
\begin{enumerate}[1.]
\item
  Given a rigid bundle 2-gerbe on $X_1$ equipped with a trivialisation, and a map $f\colon X_2\to X_1$, we can pull everything along the map $f$ to get a trivialised rigid bundle 2-gerbe on $X_2$.
\item
  The analogous result of the constructions in \ref{example:rb2g_constructions} 2.\ and 3.\ hold, assuming given a trivialisation.
\item
  Given two rigid bundle 2-gerbes on $X$ involving the same surjective submersion and semisimplicial surjective submersion, given trivialisations of both of them, the product has a trivialisation.
\item
  Given a rigid bundle 2-gerbe with a trivialisation, the dual inherits a trivialisation.

\end{enumerate}
\end{example}

This definition becomes more powerful when we consider connective structures, to be defined in Subsection~\ref{subsec:conn_structure} below, as this allows for the calculation of \emph{geometric} string structures \cite[Definition 2.8]{Wal}.
Such a calculation will be the topic of forthcoming work.

\begin{remark}
It should be noted that the definition of string structure in \cite{Wal} is not pefectly analogous to the one given here, in that there is an extra dualisation at play. The bundle gerbe $(F,Z_0)$ here corresponds to the \emph{dual} of the trivialising bundle gerbe $\mathcal{S}$ in \cite[Definition 3.5]{Wal}. As a result, and comparison between the DD classes of these bundle gerbes needs to take into account an extra minus sign in cohomology.
\end{remark}

\subsection{Associated weak bundle 2-gerbes and rigidification}\label{subsec:associated_b2gs}

A rigid bundle 2-gerbe $\cG$ gives rise to a bundle 2-gerbe $w(\cG)$ in the ordinary, weak sense of \cite{Ste} (`stable' bundle 2-gerbes in \cite[Definition 7.4]{StevPhD}) or \cite{Wal}. 
This definition is recalled in Appendix~\ref{app:weak_2-gerbes} below (Definition~\ref{def:bundle2gerbe}), and a conceptual description is enough for the purposes of understanding the construction sketch we now give.

Recall the intuitive picture of a bundle 2-gerbe (dating back to \cite{Carey-Murray-Wang_97}), which is a surjective submerion $Y\to X$, a bundle gerbe on $Y^{[2]}$, together with extra higher product data over $Y^{[3]}$ and coherence over $Y^{[4]}$. 
The devil in the detail is, of course, pinning down the definition of, and then constructing, the latter two items; the bundle gerbe on $Y^{[2]}$ is already given as part of the data of a rigid bundle 2-gerbe. 
The following definition sketch is very high-level, and can be merely skimmed by those not particularly concerned with how our definition relates to the existing definition in the literature. 
Readers with a keen interest are pointed to Appendix~\ref{app:weak_2-gerbes} for more details.

\begin{definition}\label{def:weakification_sketch}
Given the bundle gerbe $\cE := (E,Z_1)$ on $Y^{[2]}$, and the strong trivialisation $M$ of $(\delta_h(E),Z_2)$, we get an isomorphism of bundle gerbes $(d_0^*E\otimes d_2^*E,Z_2)\simeq (d_1^*E,Z_2)$, which allows us to construct a span of bundle gerbe morphisms $d_0^*\cE\otimes d_2^*\cE \leftarrow (d_1^*E,Z_2) \to d_1^*\cE$ over $Y^{[3]}$. 
Now a bundle gerbe morphism gives rise to a stable isomorphism \cite[Note 3.2.1]{StevPhD}, and in fact there is a pseudofunctor from the category of bundle gerbes and morphisms on a given manifold, to the bigroupoid of bundle gerbes with 1-arrows stable isomorphisms (denoted here $\rightsquigarrow$). 
We can thus map this span into the latter bigroupoid and replace the backwards pointing arrow with its canonically-defined adjoint inverse, giving a stable isomorphism $d_0^*\cE\otimes d_2^*\cE \rightsquigarrow d_1^*\cE$, which is the needed higher product data, the \emph{bundle 2-gerbe product}.

We then need to define an associator for the bundle 2-gerbe product. This requires considering an induced square diagram (in the 1-category of bundle gerbes above) whose boundary consists of pullbacks of the span version of the product, and which is filled with a further two spans. 
All the nodes in this diagram are bundle gerbes on $Y^{[4]}$, and the diagram consists of four commuting squares of the form
\[
  \xymatrix@=1em{
    \bullet & \ar[l] \bullet \ar[r] & \bullet\\
    \ar[u] \bullet \ar[d] & \ar[l] \ar[u] \bullet \ar[r] \ar[d] & \ar[u] \bullet \ar[d] \\
    \bullet & \ar[l] \bullet \ar[r] & \bullet
  }
\] 
Note that these will not paste into a single commuting diagram. 
At the centre of the square is a bundle gerbe with submersion $Z_3\to Y^{[4]}$, and associativity of $M$ is used to ensure that in fact the various constructions of this bundle gerbe from the given data are compatible. 
We use the concept of \emph{mates} in the bigroupoid of bundle gerbes and stable isomorphisms to turn around the various wrong-way edges to give four pasteable squares that may now only commute up to a uniquely-specified isomorphism. 
The composite of these four 2-arrows is the associator for the bundle 2-gerbe.

Then the associator needs to be proved coherent, and there is a cubical diagram built out of spans of morphsms of bundle gerbe morphisms over $Y^{[5]}$, which when mapped into the bigroupoid of bundle gerbes and stable isomorphisms gives the required 2-commuting cube.
 
We call $w(\cG)$ the bundle 2-gerbe \emph{associated to $\cG$}.
\end{definition}

The difference between our definition and the fully weak definition is that in Definition~\ref{def:vsb2g_triv} (2) we are not taking the obvious, default option for $Z_\bullet$, which is constructed by iterated fibre products of pullbacks of $Z_1\to Y^{[2]}$, and in Definition~\ref{def:vsb2g_triv} (4) we use a strong trivialisation instead of a stable trivialisation, and then the coherence condition on the trivialisation is much simpler.
We discuss the construction of $w(\cG)$ in more detail in Construction~\ref{constr:associated_b2g}.

One might wonder the extent to which this process is reversible, namely when a rigid bundle 2-gerbe can be assigned to a given (weak) bundle 2-gerbe.

\begin{definition}\label{def:rigidification}
A \emph{rigidification} of a bundle 2-gerbe $\mathcal{H}$ is a rigid bundle 2-gerbe $\cG$ such that the associated bundle 2-gerbe $w(\cG)$ is isomorphic to $\mathcal{H}$. 
\end{definition}

Isomorphism of bundle 2-gerbes here means isomorphism over the base space $X$ (cf Example \ref{example:rb2g_constructions} 2.), and consists of an isomorphism of the domains of the submersions $Y_i\to X$, $i=1,2$; an isomorphism (relative to this) of the bundle gerbes on $Y_i^{[2]}$; such that the isomorphism is compatible on the nose with the stable isomorphism that is the bundle 2-gerbe multiplication, and also such that the associators are then identified. 
Note that this is much stronger than simply requiring that $w(\cG)$ be stably isomorphic to $\cH$ \emph{as bundle 2-gerbes}, and any two rigidifications are isomorphic in the strongest sense over the base manifold.
We shall see an example in Section \ref{sec:CS2-gerbe} below.

One\footnote{for instance, the referee} might reasonably ask what happens if a rigid bundle 2-gerbe $\mathcal{G}$ is such that the bundle 2-gerbe $w(\mathcal{G})$ has a stable trivialisation. 
In this case the integral 4-class classifying $w(\mathcal{G})$ vanishes, and the natural question is whether there is a \emph{rigid} trivialisation of the original rigid bundle 2-gerbe, or perhaps a related one.
We will return to this question in \cite{rb2g_II}, but the approximate answer is that there is another rigid bundle 2-gerbe $\mathcal{G}'$ with a rigid trivialisation, and a morphism $\mathcal{G}'\to \mathcal{G}$ over the base.

Recall that the ultimate aim behind the current paper is to set up a framework in which it is easier to describe trivialisations of bundle 2-gerbes, for instance string structures, and in particular, geometric string structures.
If we have a rigidification of a bundle 2-gerbe, and a trivialisation of the rigid 2-gerbe in the sense of Definition~\ref{def:vsb2g_triv}, we need to have a trivialisation of the original bundle 2-gerbe in the weak sense of \cite[Definition~11.1]{Ste}, recalled in Definition~\ref{def:bundle2gerbetriv} below. 
The full details of this are given in Construction~\ref{constr:associated_trivialisation}, but the rough idea goes as follows: the data of the rigid trivialisation of a rigid bundle 2-gerbe with bundle gerbe $\cE = (E,Z_1)$ on $Y^{[2]}$, where the rigid trivialisation involves the bundle gerbe $\cF = (F,Z_0)$ on $Y$, gives rise to a bundle gerbe morphism $\cE \to d_0^*\cF \otimes d_1^*\cF$ over $Y^{[2]}$. 
This morphism then gives a stable isomorphism, as required for the definition of trivialisation of a bundle 2-gerbe, and one checks (as in Construction~\ref{constr:associated_trivialisation}), that this satisfies the compatibility requirements according to Definition~\ref{def:bundle2gerbetriv}.

Finally, rigidification is a very strong notion, and bundle 2-gerbes are weak objects, with a weak notion of equivalence. It is not true that every bundle 2-gerbe has a rigidification, but a natural question to ask is whether every bundle 2-gerbe is at least \emph{stably isomorphic} to a rigidified bundle 2-gerbe. 
This turns out to be the case, as there is a model for the universal bundle 2-gerbe, using diffeological spaces, that has a rigidification.
We will return to this construction in a sequel to this paper.

\subsection{Connective structures}\label{subsec:conn_structure}

The data of a rigid bundle 2-gerbe is not perhaps, on its own, that exciting, as there are other models for 2-gerbes of varying level of weakness, not least the one of \cite{KM19}, developed in parallel and independently of our work. 
However, it is the potential for truly \emph{geometric} information that makes the use of bundle gerbes so interesting, and the same is true for our model for 2-gerbes. 
The construction of bundle 2-gerbes introduced in \cite{StevPhD} is geometric, but the technical details are quite daunting. 
Our notion of connective structure is even more constrained than that in \cite{StevPhD}, and the proof of its existence, independent of \emph{loc.\ cit.}, is much shorter, even when giving complete details.

\begin{definition}\label{def:conn_structure}
A \emph{connective structure} on a rigid bundle 2-gerbe consists of the differential form data\footnote{Of course, connections on the $U(1)$-bundle $E$ are only an affine subspace of $\Omega^1(E)$, but the conditions here cut down the space a lot more.} (using the notation as before) $(\nabla_E,\beta,\lambda) \in \Omega^1(E) \oplus \Omega^2(Z_1) \oplus  \Omega^3(Y)$ such that
\begin{enumerate}
\item $\nabla_E$ is a bundle gerbe connection: $0=\delta_v(\nabla_E) \in \Omega^1(\delta_vE)$;
\item $\beta$ is a curving for $\nabla_E$, namely $\delta_v(\beta) = \curv(\nabla_E)$;
\item $\lambda$ is a 2-curving for $\beta$, namely $\pi^*\delta_h(\lambda) = d\beta$.
\end{enumerate}
We say a connective structure is \emph{rigid} if there is a specified 1-form $\alpha$ on $Z_2$ such that
\begin{enumerate}\setcounter{enumi}{3}
\item $\delta_v(\alpha) = M^*\delta_h(\nabla_E)$;
\item $d\alpha = \delta_h\beta$;
\item $\delta_h(\alpha)=0$.
\end{enumerate}
\end{definition}

The intuition behind this definition is that the data of the 2-gerbe multiplication $M$ is a strong trivialisation \emph{with connection} (items (4) and (5)), which is then highly coherent: its image under $\delta_h$ is identically trivial.
The intention here is that the rigid version is the `real' definition, not merely a special case.

It is not a priori obvious that one can find a connective structure, and even less obvious that one can find a rigid connective structure.
We rely heavily on the fundamental Lemma \ref{lemma:fund_lemma} and the fact we have a nice bisimplicial manifold built out of a level-wise surjective submersion of semisimplicial manifolds.

\begin{proposition}\label{prop:conn_exist}
Every rigid bundle 2-gerbe admits a rigid connective structure.
\end{proposition}

The proof below should be compared with the proof of Proposition 10.1 in \cite{StevPhD}, which is much more involved.

\begin{proof}
To begin, we know that the bundle gerbe $(E,Z_1)$ on $Y^{[2]}$ admits a connection $\nabla_E$ and curving $\widetilde\beta_{1,1}$ \cite{Mur}. 
This curving might not be the one we want, since its curvature 3-form $H$ will probably not satisfy $\delta_h (H) = 0$. 
As a result, we cannot yet guarantee the existence of the 2-curving 3-form. 
We shall chase through a sequence of forms that will give in the end a correction term.

\begin{enumerate}

\item The bundle 2-gerbe multiplication is a section $M$ of $\delta_h E \to Z_2^{[2]}$, so we can pull the connection $\delta_h(\nabla_E)$ on $\delta_hE$ back along $M$ to get $M^*\delta_h(\nabla_E) \in \Omega^1(Z_2^{[2]})$. Then by the compatibility of the the bundle 2-gerbe multiplication (\ref{def:vstrictb2g}\eqref{b2g:mult}) we have
\begin{align*}	
\delta_v(M^*\delta_h(\nabla_E)) & = (\delta_v M)^*\delta_v(\delta_h(\nabla_E))\\
						 & = (\delta_h\sigma)^*\delta_v(\delta_h(\nabla_E))\\
						 & = (\delta_h\sigma)^*\delta_h(\delta_v(\nabla_E))\\
						 & = \delta_h(\sigma^*\delta_v(\nabla_E)) = 0.
\end{align*}
Thus by the the fundamental sequence there (non-uniquely) exists $\tilde\alpha_{2,1}\in \Omega^1(Z_2)$ with $\delta_v(\tilde\alpha_{2,1}) = M^*\delta_h(\nabla_E)$.
This 1-form should be considered as a connection on the trivial bundle on $Z_2$ that is, together with $M$, the trivialisation of the bundle gerbe $(\delta_hE,Z_2)$.

\item Now consider $\delta_v(\delta_h(\tilde\alpha_{2,1}))$:
\begin{align*}
\delta_v(\delta_h(\tilde\alpha_{2,1})) & = \delta_h(\delta_v(\tilde\alpha_{2,1})) \\
& = \delta_h(M^*\delta_h(\nabla_E))\\
& = (\delta_hM)^*\delta_h(\delta_h(\nabla_E)) = 0.
\end{align*}
Then by the fundamental sequence again there exists a (unique) 1-form $\alpha_{3,0} \in \Omega^1(Y^{[4]})$ with $\pi^*(-\alpha_{3,0}) = \delta_h(\tilde\alpha_{2,1})$. Note that the choice of the sign here is to agree with the form $\alpha$ defined by equation (10.12) in \cite{StevPhD}.

\item Now $\pi^*\delta_h(\alpha_{3,0}) = \delta_h(\pi^*\alpha_{3,0}) = \delta_h(-\delta_h(\tilde\alpha_{2,1})) = 0$. Since $\pi^*$ is injective, we thus have $\delta_h(\alpha_{3,0}) = 0$.\footnote{This is the last spelled-out step of the proof of \cite[Proposition 10.1]{StevPhD}.} By the fundamental sequence there exists $\alpha_{2,0}\in \Omega^1(Y^{[3]})$ such that $\delta_h(\alpha_{2,0}) = \alpha_{3,0}$.

\item Define $\alpha_{2,1} = \tilde\alpha_{2,1} + \pi^*\alpha_{2,0}$. 
This new 1-form also satisfies $\delta_v(\alpha_{2,1}) = M^*\delta_h(\nabla_E)$, but now also $\delta_h(\alpha_{2,1}) = 0$.
Note also that the data of the curving $\widetilde\beta_{1,1}$ is not involved in the definition, and the only ambiguity is that $\alpha_{2,0}$ is only unique up to the addition of $\delta_h$ of a 1-form on $Y^{[2]}$.
With this new choice $\alpha_{2,1}$, we would have $\alpha_{3,0}=0$ and $\alpha_{2,0}=0$.

\item Now notice that the bundle gerbe $(\delta_hE,Z_2)$ has two curvings for $\delta(\nabla_E)$: $d\alpha_{2,1}$ and $\delta_h(\widetilde\beta_{1,1})$, so that by general properties of bundle gerbes there exists a (unique) 2-form $\beta_{2,0}\in \Omega^2(Y^{[3]})$ such that $\pi^*\beta_{2,0} = \delta_h(\widetilde\beta_{1,1}) - d\alpha_{2,1}$. This satisfies
\begin{align*}
\pi^*\delta_h(\beta_{2,0}) & =\delta_h(\pi^*\beta_{2,0}) \\
& = \delta_h(\delta_h(\widetilde\beta_{1,1} - d\alpha_{2,1}))\\
&= -\delta_h(\alpha_{2,1}) = 0,
\end{align*}
and so $\delta_h(\beta_{2,0})=0$.
By the fundamental sequence there exists $\beta_{1,0}\in \Omega^2(Y^{[2]})$ such that $\delta_h(\beta_{1,0}) = \beta_{2,0}$.

\item Now define $\beta_{1,1} = \widetilde\beta_{1,1} - \pi^*\beta_{1,0}$. This is also a curving for the bundle gerbe connection $\nabla_E$. This will be the curving we seek (i.e.\ it \emph{does} satisfy $\delta_h(d{\beta_{1,1}}) = 0$).
Now define $\lambda_{1,0}\in \Omega^3(Y^{[2]})$ so that $\pi^*\lambda_{1,0} = d\beta_{1,1}$.
This is the 3-form curvature of the bundle gerbe on $Y^{[2]}$.

\item We record the following calculation relating the bundle 2-gerbe curving to the other data so far:
\begin{align}\label{eq:strict_b2g_curving}
\delta_h(\beta_{1,1}) 
  & = \delta_h\widetilde\beta_{1,1} - \pi^*\delta_h\beta_{1,0} \nonumber\\
  & = \pi^*\beta_{2,0} + d\alpha_{2,1} - \pi^*\beta_{2,0}\\
  & = d\alpha_{2,1}. \nonumber
\end{align}
As noted earlier, $\alpha_{2,1}$ is independent of the starting choice $\widetilde\beta_{1,1}$, so that given a fixed connection $\nabla_E$ and a choice of 1-form $\alpha_{2,1}$ as above, the process always gives curvings $\beta_{1,1}$ satisfying \eqref{eq:strict_b2g_curving}.

\item We can now examine $\delta_h(\lambda_{1,0})$:
\begin{align*}
\pi^*\delta_h(\lambda_{1,0}) 
  & = \delta_h(d\beta_{1,1})\\
  & = d(\delta_h(\beta_{1,1}))\\
  & = d^2\alpha_{2,1} = 0.
\end{align*}
This is what we were after, since it follows that $\delta_h(\lambda_{1,0})=0$. Using the fundamental sequence, there exists $\lambda_{0,0}\in \Omega^3(Y)$ such that $\delta_h(\lambda_{0,0}) = \lambda_{1,0}$.
\end{enumerate}

The triple consisting of $(\nabla_E, \beta_{1,1}, \lambda_{0,0})$ is then a connective structure, and $\alpha = \alpha_{(2,1)}$ makes it a rigid connective structure.
\end{proof}

Given a connective structure, there is a unique $\curv(\nabla_E,\beta,\lambda)\in \Omega^4(X)$ satisfying $\pi^*\curv(\nabla_E,\beta,\lambda) = d\lambda$.
We can use injectivity of $\pi^*$ to conclude that $\curv(\nabla_E,\beta,\lambda)$ is closed, and hence defines a class in $H^4_{dR}(X)$. 

\begin{definition}
The \emph{4-curvature} of a rigid bundle 2-gerbe on $X$ with (rigid) connective structure $(\nabla_E,\beta,\lambda)$ is the closed 4-form $\curv(\nabla_E,\beta,\lambda)$.  
\end{definition}

\begin{example}\label{example:conn_str_constructions}
In parallel with Example \ref{example:rb2g_constructions}, we can build new (rigid) connective structures from old:
\begin{enumerate}
\item
  Given a rigid bundle 2-gerbe on $X_1$, and a (rigid) connective structure and a map $f\colon X_2\to X_1$, we can pulled back rigid bundle 2-gerbe on $X_2$ inherits a (rigid) connective structure, by pulling back all the forms involved.
\item
  Similarly for Example \ref{example:rb2g_constructions} (2).
\item
  Given two rigid bundle 2-gerbes on $X$ involving the same surjective submersion $Y\to X$ and semisimplicial surjective submersion $Z_\bullet \to Y^{[\bullet+1]}$, with $U(1)$-bundles $E\to Z_1^{[2]}$ and $E'\to Z_1^{[2]}$, with connective structures $(\nabla,\beta,\lambda)$ and $(\nabla',\beta',\lambda')$, the rigid bundle 2-gerbe with $U(1)$-bundle $E\otimes E'$ gains a rigid connective structure $(\nabla\otimes \nabla',\beta+\beta',\lambda+\lambda')$. 
  If the connective structures are rigid, with 1-forms $\alpha$ and $\alpha'$ on $Z_2$, then the analogous construction holds, with the corresponding 1-form $\alpha+\alpha'$.
\item
  Given a rigid bundle 2-gerbe on $X$, with $U(1)$-bundle $E$, and (rigid) connective structure, the rigid bundle 2-gerbe with $U(1)$-bundle $E^*$, has a (rigid) connective structure given by the negative of the one on the original 2-gerbe.
\end{enumerate}
\end{example}

Notice that any pair of 2-curvings for a fixed curving differ by a 3-form on $X$, by the fundamental lemma.
This implies the the 4-curvatures of connective structures $(\nabla_E,\beta,\lambda)$ and $(\nabla_E,\beta,\lambda')$ differ by an exact 4-form.
More generally,

\begin{proposition}\label{prop:deRhamClass}
Given a pair of rigid connective structures, $(\nabla_E,\beta,\lambda,\alpha)$ and $(\nabla'_E,\beta',\lambda',\alpha')$, on a rigid bundle 2-gerbe, the 4-curvatures differ by an exact 4-form.
Hence the de~Rham class associated to a rigid bundle 2-gerbe is independent of the choice of rigid connective structure.
\end{proposition}

\begin{proof} 
We will work through changing the connective structure one piece at a time. 

\begin{enumerate}
\item As observed above, given the data $(\nabla_E,\beta,\alpha)$ of a connection and curving fitting into \emph{some} rigid connective structure (hence taking $\nabla_E=\nabla_{E'}$, $\beta=\beta'$ and $\alpha=\alpha'$), different choices of 2-curving $\lambda$ only change the 4-curvature by an exact form.

\item Now fix the bundle gerbe connection $\nabla_E=\nabla_{E'}$, and the 1-form $\alpha=\alpha'$.
Suppose we are given two bundle 2-gerbe curvings (in the rigid sense), satisfying the condition $\delta_h\beta_{1,1}= d\alpha=\delta_h\beta'_{1,1}$.

\item This implies that $\delta_h(\beta_{1,1}-\beta'_{1,1}) = 0$, and further, the 2-form $b\in \Omega^2(Y^{[2]})$ with $\pi^*b = \beta_{1,1}-\beta'_{1,1}$ satisfies $\delta_h(b)=0$, so that $\beta_{1,1}-\beta'_{1,1} = \pi^*\delta_h(\beta_{0,0})$ for some $\beta_{0,0}\in \Omega^2(Y)$.

\item The two 3-form curvatures $\lambda_{1,0}$ and $\lambda'_{1,0}$ now satisfy $\lambda_{1,0} - \lambda'_{1,0} = \delta_h(d\beta_{0,0})$.

\item Since the 2-curvings satisfy $\delta_h(\lambda_{0,0})=\lambda_{1,0}$ (and same with primed terms), we have $\delta_h(\lambda_{0,0} - \lambda'_{0,0}) = \delta_h(d\beta_{0,0})$.

\item Thus $\lambda_{0,0} - \lambda'_{0,0} = d\beta_{0,0} + \pi^*\omega$, where $\omega\in \Omega^3(X)$. 
This allows us to conclude that the 4-form curvatures differ by an exact 4-form.

\item Now take $\nabla_E = \nabla_{E'}$, but consider two compatible choices of 1-form $\alpha$ and $\alpha' = \alpha+\pi^*a$, for $a\in \Omega^1(Y^{[3]})$ (this follows from the requirement that $\delta_h(\nabla_E) = \pi^*\delta_v(\alpha)=\pi^*\delta_v(\alpha')$).
Since we have $\delta_h(\alpha) = 0$, we have $\pi^*\delta_h (a) =\delta_h(\pi^*a) = 0$, and hence $\delta_h a=0$.
This gives $a=\delta_h(\alpha_{1,0})$ for some $\alpha_{1,0}\in \Omega^1(Y^{[2]})$.

\item Assuming we have a curving $\beta$ for $\nabla_E$ that is compatible with $\alpha$, so that $d\alpha = \delta_h(\beta)$, then $\beta + \pi^*d\alpha_{1,0}$ is another curving, this time compatible with $\alpha'$. 
But the 3-form curvatures on $Y^{[2]}$ of these two curvings are identical, so there is no change in the 4-form curvature down on $X$.

\item Now, given two different bundle gerbe connections, the curving $\beta$ for $\nabla_E$ is almost a curving $\beta'$ for $\nabla_{E'}$: one needs to add an exact 2-form that is the derivative of some 1-form $\kappa\in \Omega^1(Z_1)$ such that $\delta_v(\kappa)$ is the difference between the two connections. 
The 3-form curvatures of these are the same. 
So then picking an arbitrary (bundle 2-gerbe) curving for the second connection, we then land back in the previous case, and we are done.\qedhere
\end{enumerate}
\end{proof}

This theorem of course doesn't tell us what is the image of the map 
\begin{equation}\label{eq:char_class}
  \{\text{rigid 2-gerbes with rigid connective structure on } X\} \to H^4_{dR}(X).
\end{equation}
However, in a future paper \cite{rb2g_II}, we will construct a rigid model for the universal bundle 2-gerbe in diffeological spaces over a specific model of $B^3U(1)$ (which is also topologicaly a $K(\ZZ,4)$ space), so that every bundle 2-gerbe (in the weakest sense) will be stably isomorphic to a bundle 2-gerbe with a rigidification, at least as a diffeological rigid bundle 2-gerbe.
Unlike \cite[Theorem 3.15]{KM19}, our model will not use path fibrations, and will in fact build the universal bundle 2-gerbe using similar ideas to those behind the Chern--Simons 2-gerbe.
We are also ensuring that our universal 2-gerbe has the correct topological behaviour under the functor $\mathbf{Diffeol} \to \mathbf{Top}$ assigning a diffeological space the underlying set of points with the D-topology.

It follows from the discussion in Example \ref{example:conn_str_constructions} and Proposition \ref{prop:deRhamClass} that the de~Rham class $[\curv(\nabla,\beta,\lambda)]$ is a characteristic class for a rigid bundle 2-gerbe, as it pulls back appropriately along maps between manifolds.

\begin{remark}
Just as a rigid bundle 2-gerbe gives rise to a weak bundle 2-gerbe, a connective structure on a rigid bundle 2-gerbe gives rise to a connection, curving and 2-curving on a weak (aka `stable') bundle 2-gerbe in the sense of \cite[Definition 7.4]{Ste}. 
It even gives a connection, curving and 2-curving in the sense of \cite[Definition 4.1]{Wal}, which has additional constraints.\footnote{The existence proof for 2-gerbe connections etc in \cite{StevPhD} seems to establish existence of connections etc.\ in the stronger sense of Waldorf's later paper just cited, though the properties are not explicitly claimed to hold by Stevenson.}
A proof of this is given in Appendix \ref{app:weak_2-gerbes} (see Proposition~\ref{prop:associated_Waldorf_connective_structure}). 
The resulting 4-form curvatures are the same in both cases.
\end{remark}

Just as rigid bundle 2-gerbes have connective structure, we can discuss connective structure on a trivialisation.

\begin{definition}\label{def:vsb2g_triv_conn_str}
Given a trivialisation $(F,t)$ of, and a connective structure $(\nabla_E,\beta,\lambda)$ on, a rigid bundle 2-gerbe, a \emph{trivialisation connective structure} consists of the differential form data $(\nabla_F,a,b) \in \Omega^1(F)\oplus\Omega^1(Z_1)\oplus \Omega^2(Z_0)$ satisfying:
\begin{enumerate}[(1)]
\item $\nabla_F$ is a bundle gerbe connection: $0=\delta_v(\nabla_F) \in \Omega^1(\delta_vF)$;
\item $b$ is a curving for $\nabla_F$: $\pi^*\delta_v(b) = d\nabla_F$;
\item $\pi^*\delta_v(a) = \nabla_E - t^*\delta_h(\nabla_F)$;
\item $da = \beta - \delta_h(b) \in \Omega^2(Z_1)$.
\end{enumerate}
If we are given a 1-form $\alpha$ making $(\nabla_E,\beta,\lambda)$ a \emph{rigid} connective structure, the same definition of trivialisation connective structure holds verbatim.
\end{definition}

The third and fourth items make the strong trivialisation of the bundle gerbe $(E\otimes \delta_h(F)^*,Z_1)$ a strong trivialisation \emph{with connection}. 
The 1-form $a$ should be thought of as giving a connection on the trivial bundle on $Z_1$, such that its curvature is the difference between the curving $\beta$ and the induced curving $\delta_h(b)$.

Versions of the constructions in Examples \ref{example:conn_str_constructions} and \ref{example:triv_rb2g_constructions} hold in the case of a trivialisation connective struture.

It follows that given a trivialisation connective structure, the curvature $h\in \Omega^3(Y)$ of the curving $b$ is another 2-curving for the 1-curving $\beta$. 
Thus there is a unique 3-form $R\in \Omega^3(X)$ such that $\pi^*R = \lambda - h$. 
Further, since $dh=0$, we have $\pi^*dR = d\lambda = \pi^*\curv(\nabla_E,\beta,\lambda)$, and the injectivity of $\pi^*$ means that the curvature 4-form of the rigid connective structure is exact, hence trivial in de Rham cohomology, and we have an explicit primitive $R$ for $\curv(\nabla_E,\beta,\lambda)$.
Such primitives will appear in the forthcoming work \cite{rb2g_III} dealing with a wide class of examples of geometric string structures using our framework here.

\begin{remark}\label{rem:Redden_3-form}
This 3-form appears in work of Redden on the refined Stolz--H\"ohn conjecture \cite[\S 3]{Redden}, before arising in a more general setting in Waldorf's treatment of geometric string structures \cite[Lemma 4.6]{Wal}. 
Redden explicitly calculates deformation families of such 3-forms on $S^3$, there denoted $H_{\mathscr{S},g_{\alpha_1,\alpha_2}}$, associated to 2-parameter families of left-invariant metrics and different choices of string structures $\mathscr{S}$ of $S^3$, to the point of being able to give plots of the integral of said 3-form.
One potential application of \cite{rb2g_III} is to replicate Redden's explicit calculations in additional cases.
\end{remark}

\begin{proposition}
Given a trivialisation and a rigid connective structure $(\nabla_E,\beta,\lambda)$ on a rigid bundle 2-gerbe, a trivialisation connective structure exists.
\end{proposition}

\begin{proof}
Take for $\nabla_F$ any bundle gerbe connection on $F\rightrightarrows Z_0$, and $\tilde b_{0,1}$ any curving for it. 
We need to `fix' $\tilde b_{0,1}$ to a curving satisfying the third point in Definition~\ref{def:vsb2g_triv_conn_str}.

\begin{enumerate}
\item Using the isomorphism $t$ of line bundles, we get two connections $t^*\delta_h(\nabla_F)$ and $\nabla_E$ on the bundle $E$, so there is a unique $a_{1,2} \in \Omega^1(Z_1^{[2]})$ such that $\pi^*a_{1,2} = \nabla_E - t^*\delta_h(\nabla_F)$. Or, to put it another way, $\nabla_E = t^*\delta_h(\nabla_F) + \pi^*a_{1,2}$.

\item Since $t$ is also an isomorphism of bundle gerbes, we have $\delta_v(\pi^*a_{1,2}) = \delta_v(\nabla_E) - t^*\delta_h(\delta_v(\nabla_F)) = 0$. Using the fundamental sequence, there exists $\tilde a_{1,1} \in \Omega^1(Z_1)$ such that $\delta_v(\tilde a_{1,1}) = a_{1,2}$. 
Thus we have $\pi^*\delta_v(\tilde a_{1,1}) = \nabla_E - t^*\delta_h(\nabla_F)$, but this doesn't necessarily satisfy (4) with the given curving $\tilde b_{0,1}$.

\item We can now compare $\alpha_{2,1}$ and $\delta_h(\tilde a_{1,1})$:
\begin{align*}
\delta_v(\delta_h (\tilde a_{1,1}) - \alpha_{2,1}) 
& = \delta_h a_{1,2} - \delta_v\alpha_{2,1}\\
& = M^*\pi^*\delta_h a_{1,2} - M^* \delta_h(\nabla_E)\\
& = M^*(\delta_h\pi^*a_{1,2} - \delta_h\nabla_E)\\
& = M^*\delta_h(\nabla_E - t^*\delta_h(\nabla_F)-\nabla_E)\\
& = M^*\delta_h(t)^*\delta_h(\delta_h(\nabla_F))\\
& = c_{\delta^2F}^*\delta_h(\delta_h(\nabla_F)).
\end{align*}
This last 1-form in fact vanishes, as $c_{\delta^2F}$ is flat with respect to $\delta_h(\delta_h(\nabla_F))$. 
This follows from the fact the section of $E\otimes E^*$ is always flat with respect to any connection induced from $E$, since that is how the section of $\delta_h^2F$ is built.
By using the fundamental sequence there exists a (unique) $a_{2,0} \in \Omega^1(Y^{[3]})$ such that $\pi^*a_{2,0} = \delta_h(\tilde a_{1,1}) - \alpha_{2,1}$.

\item Since we have a rigid connective structure, with $\delta_h(\alpha_{2,1})=0$, we find that $\pi^*\delta_h(a_{2,0}) = 0$, and since $\pi^*$ is injectve, $\delta_h(a_{2,0})=0$ so there exists $a_{1,0} \in \Omega^1(Y^{[2]})$ such that $\delta_h(a_{1,0}) = a_{2,0}$.

\item Now define $a_{1,1} := \tilde a_{1,1} - \pi^* a_{1,0}$, which satisfies $\pi^*\delta_v(a_{1,1}) = \nabla_E - t^*\delta_h(\nabla_F)$, as $\delta_v\pi^*a_{1,0}=0$.
This will fix up the connection on the trivialisation of $(E\otimes \delta_h(F)^*,Z_1)$ to make it a strong trivialisation with connection, once we have the correct curving on $(F,Z_0)$.

\item Moreover, 
\begin{align}\label{eq:triv_connet_exist_prf_6}
\delta_h(a_{1,1}) & = \delta_h(\tilde a_{1,1}) - \pi^*\delta_h(a_{1,0})  \\
                  & = \alpha_{2,1} + \pi^*(a_{2,0} - \delta_h(a_{1,0}))  \nonumber\\
                  & = \alpha_{2,1} + \pi^*(a_{2,0} - a_{2,0} )  \nonumber\\
                  & = \alpha_{2,1},  \nonumber
\end{align}
and it follows from the rigidity of the original connective structure that $\delta_h(\beta_{1,1}) = d\alpha_{2,1} = \delta_h(da_{1,1})$.

\item Now we have $\beta_{1,1} - \delta_h(\tilde b_{0,1})$ and $da_{1,1}$ are both curvings for the bundle gerbe connection $\nabla_E - t^*\delta_h(\nabla_F)$. Thus there is a unique $b_{1,0} \in \Omega^2(Y^{[2]})$ such that $\pi^*b_{1,0} = \big(\beta_{1,1} - \delta_h(\tilde b_{0,1})\big) - da_{1,1}$.

\item But now $\pi^*\delta_h b_{1,0} = \delta_h(\beta_{1,1}) - \delta_h(da_{1,1}) = 0$ and hence $\delta_h(b_{1,0})=0$. So there exists $b_{0,0} \in \Omega^2(Y)$ such that $\delta_h(b_{0,0}) = b_{1,0}$. So now define $b_{0,1} := \tilde b_{0,1} + \pi^* b_{0,0} \in \Omega^2(Z_0)$. This is our `fixed' curving for $\nabla_F$.

\item Further, this satisfies $\beta_{1,1} - \delta_h(b_{0,1}) = da_{1,1}$, as needed.
\end{enumerate}

Then $(\nabla_F,a_{1,1},b_{0,1})$ is a trivialisation connective structure for the given data.
\end{proof}


We are not going to develop a complete theory of rigid bundle 2-gerbes and their connective structures, although one could in principle do this. 
It would require discussion of what the analogue of stable isomorphism would look like, the tensor product of a general pair of rigid 2-gerbes with connective structure, descent etc. 
An in-depth treatment of \emph{bigerbes} (without connective structures) is given in \cite{KM19}, but as our motivation is primarily towards finding a convenient and efficient means to construct bundle 2-gerbes, and in particular geometric string structures, we will not be pursuing the analogous study here.


\section{The Chern--Simons bundle 2-gerbe}\label{sec:CS2-gerbe}

A key example in the literature of a bundle 2-gerbe, namely that of a Chern--Simons 2-gerbe, is constructed in \cite{StevPhD}, and in \cite{JohnsonPhD,CJMSW} the authors describe a connective structure.
This bundle 2-gerbe takes as input a principal $G$-bundle $Q\to X$, and has as its characteristic class ($\pm2\pi$ times) the first fractional Pontryagin class $\frac{p_1}{2}(Q)\in H^4(X,\ZZ)$. 
The 4-form curvature of the bundle 2-gerbe---ultimately constructed from a $G$-connection $A$ on $Q$---is, likewise, ($\pm 2\pi$ times) the Chern--Weil representative $\frac{1}{8\pi^2}\tr(F_A^2)\in \Omega^4(X)$ for this cohomology class.

The main result in this section is that the Chern--Simons bundle 2-gerbe in the literature (eg \cite{CJMSW}) is rigidified by the rigid bundle 2-gerbe constructed here.

\subsection{The rigid Chern--Simons bundle 2-gerbe}

Here is a key construction that we will be relying on for many calcuations in the rest of the paper. 
Fix a compact, simple, simply-connected Lie group $G$.

\begin{definition}[{\cite[proof of Proposition 3.1]{BSCS}}]\label{def:string_xmod}
The crossed module $\Omegahat{G} \to PG$ has action $\Adhat\colon PG\times \Omegahat{G} \to \Omegahat{G}$ defined  as the descent of the map
\begin{align}\label{eq:lifted_Ad}
  \widetilde{\Ad}\colon & PG \times P\Omega G \ltimes U(1) \to P\Omega G \ltimes U(1) \nonumber \\
  &(p;f,z) \mapsto \left((t\mapsto \Ad_pf(t)), z\cdot \exp\left(2i\int_0^{2\pi} \intSone \langle f(t)^{-1}\partial_t f(t), \phi_p \rangle\, \mathrm{d}t\right)\right)
\end{align}
along the defining quotient map $P\Omega G \ltimes U(1) \twoheadrightarrow \Omegahat{G}$.
\end{definition}

\begin{remark}\label{rem:BCSS_error}
This is different to the definition implicit in \cite[Proposition 3.1]{BSCS}, because there is a sign mistake on the left side of equation (5) in Definition 2.7 (relative to  \cite[Definition 4.3.4]{HDA6}, which agrees with Definition 5.2 and Remark 5.3 in \cite{Lada-Markl_95}\footnote{The authors thank Kevin van Helden for sharing his detailed unpublished calculations for \cite{vanHelden_21} that confirm that \cite[Definition 4.34]{HDA6} agrees with the definition of weak $L_\infty$-algebra map suggested in \cite{Lada-Markl_95}.}), and this forces a cascade of sign flips leading back to the construction in \cite[Proposition 3.1]{BSCS} that gives this map $\Adhat$. 
We still get the same comparison result between the Lie 2-algebras if we change the sign of one of the brackets of the Lie 2-algebra associated to $\String_G$ and the coherence data for the Lie 2-algebra maps $\phi$ and $\lambda$ in \cite[Theorem 5.1]{BSCS}.
\end{remark}

\begin{lemma}\label{lemma:central_ext_PG_OmegaG}
The lift of the adjoint action as in Definition \ref{def:string_xmod} makes $PG\ltimes \widehat{\Omega G} \to PG \ltimes \Omega G$ a central extension of Lie groups, where the semidirect product on $PG\ltimes \widehat{\Omega G} $ is defined using $\Adhat$.
\end{lemma}

Now fix a principal $G$-bundle $Q\to X$. In (i)--(viii) below we refer to the items in Definition \ref{def:vstrictb2g}.

\begin{figure}[!t]
  \[
  \centerline{
    \xymatrix@=11ex{
      &
      \save 
        []+<-4em,10ex>  
        *+{\delta_v\delta_h \widehat{\Omega G}}="dvdh" 
        \ar@{<-}[]_-(.4){\delta_vM} 
      \restore
      \save 
        []+<4em,10ex> 
        *+{ \delta_h\delta_v \widehat{\Omega G}}="dhdv" 
        \ar@{<-}[]_-(.4){\delta_hm}
      \restore
      \ar@{->}"dvdh";"dhdv"^{!\simeq} 
      Q\times (PG\ltimes \Omega G^2)^2  
      \ar@<1.5ex>[r]^{\Act_{12}} \ar[r]|{\,(*)\,d_1\,} \ar@<-1.5ex>[r]_{\pr_{12}}  \ar@<1.5ex>[d] \ar@<-1.5ex>[d] \ar[d] 
      &
      Q\times (PG\ltimes \Omega G^2) 
      \save 
        []+<5em,10ex> 
        *+{\delta_v \widehat{\Omega G}} 
        \ar[] \ar@{<-}@/_1pc/[]_{m} 
      \restore 
      & 
      \\
      Q\times (PG\ltimes \Omega G)^3 
      \ar@<1ex>[d] \ar@<-1ex>[d] \ar@<1.5ex>[r] \ar@<-1.5ex>[r] \ar@<-0.5ex>[r]\ar@<0.5ex>[r]
      \save 
        []+<-1em,10ex> 
        *+[l]{\underline{U(1)}\stackrel{!}{\simeq}\delta_h^2 \widehat{\Omega G}} 
        \ar[] \ar@{<-}@/_1pc/[]_{1\equiv\delta_hM} 
      \restore 
      &
      Q\times (PG\ltimes \Omega G)^2 
      \ar@<1ex>[d]^{\id\times \pr_1^2} \ar@<-1ex>[d]_{m_{23}^2} \ar@<1.5ex>[r]^{\Act_{12}} \ar@<-1.5ex>[r]_{\pr_{12}} \ar[r]|{\,\id\times\mult\,} 
      \save
        []+<-5em,10ex> 
        *+{\delta_h\widehat{\Omega G}}  
        \ar[]+<-2em,2ex> \ar@{<-}@/_1pc/[]+<-2.7em,2ex>_(0.4){M} 
      \restore 
      &  
      \save 
        []+<5em,10ex> 
        *+{\widehat{\Omega G}}="2gerbe"  
        \ar[]+<2em,2ex>
      \restore 
      \ar@<1.5ex>[u];[]^(0.3){\pr_{123}}
      \ar@<-1.5ex>[u];[]_(0.65){m_{23}}
      \ar[u];[]|[right]{m_{34}}
      Q\times (PG\ltimes \Omega G)
      \ar@<1ex>[d]^{\id\times \pr_1} \ar@<-1ex>[d]_{m_{23}} 
      & 
      %
      \\
      Q\times PG^3 
      \ar[d]^{\pi} \ar@<1.5ex>[r] \ar@<-1.5ex>[r] \ar@<-0.5ex>[r]\ar@<0.5ex>[r]
      &
      Q\times PG^2 
      \ar@<1.5ex>[r]^{\Act_{12}}\ar[r]|{\,m_{23}\,} \ar@<-1.5ex>[r]_{\pr_{12}} \ar[d]^{\pi}
      & 
      Q\times PG
      \ar@<1ex>[r]^{\Act_{12}} \ar@<-1ex>[r]_{\pr_1} 
      \ar[d]^{\pi = \id\times \ev_{2\pi}}
      &
      Q 
      \ar[d]^{\id} 
      \\
      Q\times G^3 
      \ar@<1.5ex>[r]^{\Act_{12}}  \ar@<0.5ex>[r]|(0.4){\,m_{23}\,}  \ar@<-0.5ex>[r]|(0.6){\,m_{34}\,}  \ar@<-1.5ex>[r]_{\pr_{123}}
      &
      Q\times G^2  
      \ar@<1.5ex>[r]^{\Act_{12}} \ar[r]|{\,\id\times m\,} \ar@<-1.5ex>[r]_{\pr_{12}}
      & 
      Q\times G 
      \ar@<1ex>[r]^{\Act} \ar@<-1ex>[r]_{\pr_1}
      & 
      Q
      & 
      \save 
        []+<-2em,0pt> 
        *+{X} 
        \ar@{<-}[l]
      \restore
    }
  }
  \]
  \vspace{1cm}
  \caption{The rigid Chern--Simons bundle 2-gerbe of $Q\to X$. Extra factors in the $U(1)$-bundles have been suppressed.
  The map $\Act_{12}$ is the action of the second component on the first; $m_{ij}$ is multiplication of the $i$, $j$ components; $\pr_{ij\ldots}$ is projection on the $ij\ldots$ components. Isomorphisms marked $!$ are canonical. The map marked $(*)$ is explained in Definition~\ref{def:CS_b2g}.}
  \label{fig:CSb2g}
  \end{figure}

\begin{definition}\label{def:CS_b2g}
The \emph{rigid Chern--Simons bundle 2-gerbe of $Q$}, $CS(Q)$, is depicted in Figure~\ref{fig:CSb2g}, where:
\begin{enumerate}[(i)]
  \item The surjective submersion $Y\to X$ in (1) is the bundle projection $Q\to X$.
  \item The simplicial space $Q\times G^\bullet$ is isomorphic to the \v{C}ech nerve $Q^{[\bullet+1]}$ over $X$, and is the nerve of the action groupoid associated to $G$ acting on $Q$.
  \item The (truncated) semisimplicial space $Z_\bullet$ in (2) is $Q\times PG^\bullet$.
  \item The component $\pi\colon Q\times PG^n\to Q\times G^n$ of the simplicial map $Q\times PG^\bullet\to Q\times G^\bullet$ is $\id\times \ev^n_{2\pi}$, which is a surjective submersion, as required for (2).
  \item The (vertical) \v{C}ech nerve associated to the principal $\Omega G^n$-bundle $Q\times PG^n\to Q\times G^n$ is isomorphic to $Q\times (PG\ltimes \Omega G^\bullet)^n$ where the semidirect product arises from the action of $PG$ on $\widehat{\Omega G}$ as in Definition~\ref{def:string_xmod}.
  \item The map marked (*) is 
  \[
    d_1(p_1,\gamma_1,\eta_1;p_2,\gamma_2,\eta_2) = (p_1p_2,\Ad_{p_2^{-1}}(\gamma_1)\gamma_2,\gamma_2^{-1}\Ad_{p_2^{-1}}(\eta_1)\gamma_2\eta_2).
  \]
  \item The $U(1)$-bundle $E$ from (3) is $\pr_3^*\widehat{\Omega G} = Q\times PG \ltimes \widehat{\Omega G} \to Q\times PG \ltimes \Omega G$.
  In the figure we omit writing $\pr_3^*$ for brevity. 
  The details of the semidirect product involving $\widehat{\Omega G}$ appear below.
  \item The bundle gerbe multiplication $m$ from (3) is inherited from the basic gerbe on $G$. 
  Note that the bundle gerbe $(\pr_3^*\widehat{\Omega G},Q\times PG)$ is simply the pullback of the basic gerbe along the projection $Q\times G\to G$.
  \item The section $M\colon Q\times (PG\ltimes \Omega G)^2 \to Q\times \delta_h(PG\ltimes \Omegahat{G})$ is defined to be the one guaranteed to exist by discussion in Example \ref{ex:lifting}.

\end{enumerate}

\end{definition}

\begin{remark}
In \cite{StevPhD,JohnsonPhD,CJMSW} a different model for the basic gerbe is used to construct the Chern--Simons 2-gerbe, namely the \emph{tautological bundle gerbe} from \cite{Mur}, associated to the closed integral 3-form on the 2-connected manifold underlying $G$. 
This is \emph{stably} isomorphic to the model for the basic gerbe given here for formal reasons, but it is also actually \emph{isomorphic} to it.
The model of the basic gerbe we used here is inherited from \cite{MS03}.
\end{remark}

To explain the 2-gerbe multiplication $M$ in more detail, recall that the multiplication map $\mult\colon (PG\ltimes \Omega G)^2 \to PG\ltimes \Omega G$ is given by $\mult(p,\gamma,q,\eta)=(pq,\Ad_{q^{-1}}(\gamma)\eta)$.
Then $M$ arises from the isomorphism that on fibres looks like
  \begin{align*}
\delta_h(PG\times \widehat{\Omega G})_{(p,\gamma,q,\eta)} 
& \simeq \widehat{\Omega G}_\gamma\otimes \mult^*(PG\times \widehat{\Omega G})^*_{(p,\gamma,q,\eta)}\otimes \widehat{\Omega G}_\eta\\
&\simeq  \widehat{\Omega G}_\gamma\otimes \widehat{\Omega G}^*_{\Ad_{q^{-1}}(\gamma)\eta}\otimes \widehat{\Omega G}_\eta\\
&\stackrel{(1)}{\simeq} 
\widehat{\Omega G}_\gamma\otimes \widehat{\Omega G}^*_{\Ad_{q^{-1}}(\gamma)}\otimes \widehat{\Omega G}^*_\eta\otimes \widehat{\Omega G}_\eta\\
&\stackrel{(2)}{\simeq} 
\widehat{\Omega G}_\gamma\otimes \widehat{\Omega G}^*_{\Ad_{q^{-1}}(\gamma)}\\
&\stackrel{(3)}{\simeq} 
\widehat{\Omega G}_\gamma\otimes \widehat{\Omega G}^*_\gamma\\
&\stackrel{(4)}{\simeq}  U(1).
\end{align*}
In this chain of isomorphisms (1) is the isomorphism coming from the multiplicative structure on the $U(1)$-bundle, (2) and (4) come from the canonical trivialisation of $\widehat{\Omega G}^*\otimes \widehat{\Omega G} \to \Omega G$, and (3) comes from the lift of the adjoint action in Definition \ref{def:string_xmod} being a map of $U(1)$-bundles.
The section $M$ itself is then the pullback of this section along the projection to $(PG\ltimes \Omega G)^2$

\begin{proposition}
The section $M$ in Definition \ref{def:CS_b2g} (ix) satisfies the condition (4) in Definition \ref{def:vstrictb2g}, the definition of rigid bundle 2-gerbe.
\end{proposition}

\begin{proof}
We note that (4)(b) comes for free from the fact we have a central extension, which means that $\delta_h(M)=1$ \cite[\S 3]{MS03}, which is a manifestation of the multiplication in $PG\ltimes \widehat{\Omega G}$ being associative.
It thus remains to prove (a), namely $\delta_v\delta_h(M)=\delta_h\delta_v(m)$. 

The two trivialisations in question correspond fibrewise to the composite isomorphisms
  \begin{align*}
&\delta_v\delta_h(PG\times \widehat{\Omega G})_{(p,\gamma_1,\gamma_2,q,\eta_1,\eta_2)}\\ 
& \simeq \delta_h(PG\times \widehat{\Omega G})_{(p\gamma_1,\gamma_2,q\eta_1,\eta_2)}\otimes \delta_h(PG\times \widehat{\Omega G})^*_{(p,\gamma_1\gamma_2,q,\eta_1\eta_2)} \otimes \delta_h(PG\times \widehat{\Omega G})_{(p,\gamma_1,q,\eta_1)}  \\
&\simeq U(1)\otimes U(1)^*\otimes U(1)
 \simeq U(1),
\end{align*}
where the second isomorphism uses the appropriate instances of $M$, and
\begin{align*}
&\delta_h\delta_v(PG\times \widehat{\Omega G})_{(p,\gamma_1,\gamma_2,q,\eta_1,\eta_2)}\\
& \simeq \delta_v(PG\times \widehat{\Omega G})_{(q,\eta_1,\eta_2)} \otimes \delta_v(PG\times \widehat{\Omega G})^*_{(p_1p_2,\Ad_{p_2^{-1}}(\gamma_1)\gamma_2,\gamma_2^{-1}\Ad_{p_2^{-1}}(\eta_1)\gamma_2\eta_2)} \otimes \delta_v(PG\times \widehat{\Omega G})_{(p,\gamma_1,\gamma_2)} \\
& \simeq U(1)\otimes U(1)^*\otimes U(1) \simeq U(1),
\end{align*}
using the bundle gerbe multiplication section three times.
That these two isomorphisms are equal follows from \cite[Lemma 6.4]{MRSV}, and the fact they are built up using the structure of the crossed module in Definition \ref{def:string_xmod} (and so are `\textsf{xm}-morphisms' to the trivial bundle, defined in \emph{loc.\ cit.}).
\end{proof}

We end this subsection with the relationship between our rigid model for the Chern--Simons 2-gerbe and the Chern--Simons bundle 2-gerbe appearing in the literature.

\begin{theorem}\label{prop:rigidification}
The rigid bundle 2-gerbe $CS(Q)$ on $X$ is a rigidification of the Chern--Simons bundle 2-gerbe from \cite{CJMSW}.
\end{theorem}

We defer the proof of this Theorem to Subsection~\ref{subsec:CS2-gerbe_rigidified}.

\begin{remark}
As a final observation, in the ``double nerve of a commuting square'' picture of bigerbes from \cite{KM19}, the Chern--Simons aka Brylinski--McLaughlin bigerbe of a $G$-bundle $Q\to X$ is described using the double nerve of the square
\[
  \xymatrix{
    Q^I \ar[r] \ar[d] & Q^2 \ar[d] \\
    X^I \ar[r] & X^2
  }
\]
where $Q^I$ and $X^I$ denote the free path spaces. 
Kottke and Melrose remark at the end of their section 5.3 that it is possible (but leave details to the reader) to give a simpler description using the double nerve of the square
\[
    \xymatrix{
    PQ \ar[r] \ar[d] & Q \ar[d] \\
    PX \ar[r] & X
  }
\]
given a choice of compatible basepoints in $Q$ and $X$.
Our approach here is somewhat different, however, in that the bisimplicial space underlying the rigid Chern--Simons bundle 2-gerbe does \emph{not} arise as the double nerve of any commuting square.

Another point worth noting about our rigid model is that in the bisimplicial space arising from $Z_\bullet$, while the manifolds are infinite-dimensional, we have fewer infinite-dimensional factors than in the bigerbe model \cite{KM19}; the submersion $Q\times PG\to Q\times G$ results in spaces of the form $Q\times PG^{[k]}$, rather than using $PQ\times PG \simeq P(Q^{[2]})\to Q^{[2]} \simeq Q\times G$, which leads to spaces of the form $PQ^{[k]}\times PG^{[k]}$.
This parsimony leads to cleaner and easier calculations of the differential forms making up the rigid connective structure.

\end{remark}

\subsection{Geometric crossed module structure}

In the next section we will describe the connective structure on $CS(Q)$. 
First we would like to record some relationships between the geometry of the loop group and the structure of the crossed module $\Omegahat{G} \to PG$, which we turn to now. 
We can give another formula for the $U(1)$-component of $\widetilde{\Ad}$, which highlights its relationship to holonomy of a one-form on $PG \times \Omega G$. 
For this, we need Stokes' theorem for integration along a fibre where the fibre has a boundary, in the special case of the projection map $X \times [0,2\pi] \to X$. 
For our purposes we variously take $X$ to be one of the infinite-dimensional manifolds $P\Omega G$ or $PG\times P\Omega G$.

\begin{proposition}[Stokes' theorem {\cite[Proposition XI (2), \S 4.10]{GHV}}]\label{prop:stokes}
Given a $k$-form $\xi$ on $X\times [a,b]$ we have 
\[
  d\int_{[a,b]}\xi + \int_{[a,b]}d\xi = \xi\big|_{X\times\{b\}} - \xi\big|_{X\times\{a\}}\,.
\]
\end{proposition}

\begin{remark}
The signs here differ from the general fibrewise Stokes' theorem given as \cite[Problem VII.4(iii)]{GHV}, due to the latter using a different convention in the definition of fibre integration, namely contraction in the \textbf{last} slot(s) of the argument.
\end{remark}

\label{form:xmod form rho}
\begin{lemma}
Define the 1-form $\rho \in \Omega^1(PG\times \Omega G)$ by
      \[
        \rho_{(p,\gamma)} = 2\intSone\langle\Theta_p,\Ad_\gamma(\phi_p)-\phi_p - \phihat_\gamma\rangle + \langle\Theta_\gamma,\phi_p\rangle\,.
      \]
The $U(1)$-component of the definition of $\widetilde{\Ad}$ in \eqref{eq:lifted_Ad}, evaluated on $(p;f,z)$, is given by
\[
   z\cdot \exp\left(i \left(\int_{[0,2\pi]} (\id\times \ev)^*\rho\right)(p,f)\right),
\]
where $\ev\colon P\Omega G\times [0,2\pi]\to \Omega G$ is evaluation. 
\end{lemma}

Note that since $\rho$ is a 1-form, the expression $\int_{[0,2\pi]} (\id\times \ev)^*\rho$ describes a real-valued function on $PG\times P\Omega G$.

\begin{proof}
It suffices to prove that
\[
  \left(\int_{[0,2\pi]} (\id\times \ev)^*\rho\right)(p,f) = 2\int_0^{2\pi}  \intSone \langle f(t)^{-1}\partial_t f(t), \phi_p \rangle\, \mathrm{d}t
\]
for $(p,f)\in PG\times P\Omega G$. We apply the definition of integration over the fibre:
\begin{align*}
  \left(\int_{[0,2\pi]} (\id\times \ev)^*\rho\right)(p,f) 
  & = \int_0^{2\pi}((\id\times \ev)^*\rho)_{(p,f,t)}((0;0,T_t)) \,\mathrm{d}t\\
  & = \int_0^{2\pi}\rho_{(p,f(t))}(0,\ev_*(0,T_t))\, \mathrm{d}t\\
  & = \int_0^{2\pi} \rho_{(p,f(t))}(0,\partial_t f(t))\, \mathrm{d}t\\
  & = \int_0^{2\pi} 2 \intSone \langle 0,\Ad_\gamma(\phi_p)-\phi_p - \phihat_\gamma\rangle+\langle f(t)^{-1}\partial_t f(t),\phi_p\rangle \, \mathrm{d}t\\
  & = 2\int_0^{2\pi} \intSone\langle f(t)^{-1}\partial_t f(t),\phi_p\rangle \, \mathrm{d}t,
\end{align*}
which gives us the desired result.
\end{proof}

The final result we would like to record here concerns the connection and curvature of the central extension $\widehat{\Omega G}$ and how it interacts with the crossed module structure above. 

\begin{lemma}
The connection $\mu$ on $\Omegahat{G}$ satisfies
\begin{enumerate}
  \item $\widehat{\Ad}^*\mu - \pr_2^*\mu = \rho$.
  \item $\Ad^*R - \pr_2^*R = d\rho$,
\end{enumerate}
where $R= \intSone \langle\Theta_\gamma,\partial \Theta_\gamma \rangle$ is the curvature of $\mu$ as in Section~\ref{sec:new_example}.
\end{lemma}

One might note that since $R$ is the curvature of $\mu$, (2) follows from (1), but recall that the connection $\mu$ on $\Omegahat{G}$ is given by the descent of $\theta^{-1}d\theta + \int_{[0,2\pi]} \ev^* R$ along $q\colon P\Omega G \ltimes U(1) \twoheadrightarrow \Omegahat{G}$; so we will in fact prove them in the reverse direction. Note that (2) implies the result\footnote{The family of 1-forms $\beta_p$, $p\in PG$, is not used elsewhere in this paper, and it is not related to the 2-form denoted $\beta_A$ below.} $\Ad_p^*R - R = d\beta_p$ noted in the proof of \cite[Proposition 3.1]{BSCS}, using the notation of \emph{loc.\ cit.}, and modulo the sign correction noted in Remark \ref{rem:BCSS_error}.

\begin{proof}
First, we need the ingredients to calculate $\Ad^* R$. We have
\[
  (\Ad^*\Theta)_{(p,\gamma)} = \Ad_{p\gamma^{-1}}\Theta_p + \Ad_p(\Theta_\gamma - \Thetahat_p).
\]
The following identities are established using standard manipulations,
\begin{align*}
  \partial \Theta_\gamma & = \Ad_{\gamma^{-1}}d\phihat_\gamma, 
  \\
  \partial\Theta_p+[\phi_p,\Theta_p] &= d\phi_p, 
  \\
  \partial(\Ad_{\gamma^{-1}}\Theta_\gamma) &= \Ad_{\gamma^{-1}}[\Theta_\gamma,\phihat_\gamma] + \Ad_{\gamma^{-1}}\partial\Theta_\gamma, 
  \\
  \partial \Thetahat_p & = \Ad_p[\phi_p,\Theta_p] + \Ad_p \partial\Theta_p,
  \\
  d(\Ad_\gamma\phi_p) & = -\Ad_\gamma[\phi_p,\Theta_\gamma] + \Ad_\gamma \partial\Theta_p,
\end{align*}
and will be applied in the calculations below.
Thus we  can show that
\[
\begin{multlined}
\partial(\Ad^*\Theta)_{(p,\gamma)} = \Ad_p[\phi_p,\Ad_{\gamma^{-1}}\Theta_p] + \Ad_{p\gamma^{-1}}\partial\Theta_p - \Ad_{p\gamma^{-1}}[\phihat_\gamma,\Theta_p] \\+ \Ad_p[\phi_p,\Theta_\gamma-\Theta_p] + \Ad_p\partial(\Theta_\gamma-\Theta_p).
\end{multlined}
\]
Substituting in the definition of $R$ we get
\begin{align*}
&(\Ad^*R)_{(p,\gamma)} \\
&= \begin{multlined}[t]
  \intSone \langle \Ad_p(\Ad_{\gamma^{-1}}\Theta_p + \Theta_\gamma - \Theta_p),\Ad_p\big\{[\phi_p,\Ad_{\gamma^{-1}}\Theta_p] + \Ad_{\gamma^{-1}}\partial\Theta_p - \Ad_{\gamma^{-1}}[\phihat_\gamma,\Theta_p]\\
   + [\phi_p,\Theta_\gamma-\Theta_p] + \partial(\Theta_\gamma-\Theta_p)\big\}\rangle
\end{multlined}
\\
  &= \begin{multlined}[t]
  \intSone \langle \Ad_{\gamma^{-1}}\Theta_p + \Theta_\gamma - \Theta_p,[\phi_p,\Ad_{\gamma^{-1}}\Theta_p] + \Ad_{\gamma^{-1}}\partial\Theta_p - \Ad_{\gamma^{-1}}[\phihat_\gamma,\Theta_p] \\
    + [\phi_p,\Theta_\gamma-\Theta_p] + \partial(\Theta_\gamma-\Theta_p)\rangle\end{multlined}\\
  &=\begin{multlined}[t]
    \intSone \langle -\Ad_{\gamma^{-1}}[\Theta_p,\Theta_p],\phi_p-\Ad_{\gamma^{-1}}\phihat_\gamma\rangle + \langle\Theta_p,\partial\Theta_p\rangle 
    + \langle \Ad_{\gamma^{-1}}\Theta_p,[\phi_p,\Theta_\gamma-\Theta_p]\rangle\\
    +\langle \Ad_{\gamma^{-1}}\Theta_p,\partial(\Theta_\gamma-\Theta_p) \rangle + \langle\Theta_\gamma-\Theta_p,[\phi_p,\Ad_{\gamma^{-1}}\Theta_p]\rangle
    - \langle \Theta_\gamma-\Theta_p,\Ad_{\gamma^{-1}}[\phihat_\gamma,\Theta_p]\rangle\\
    +\langle \Theta_\gamma-\Theta_p,\Ad_{\gamma^{-1}}\partial\Theta_p\rangle - \langle [\Theta_\gamma-\Theta_p,\Theta_\gamma-\Theta_p],\phi_p\rangle + \langle \Theta_\gamma-\Theta_p,\partial(\Theta_\gamma-\Theta_p)\rangle.
  \end{multlined}
\end{align*}
Thus we can calculate the difference we want:
\begin{align*}
&(\Ad^*R - \pr_2^*R)_{(p,\gamma)} \\
& = 
  \begin{multlined}[t]
    \intSone 2\langle d\Theta_p,\Ad_\gamma\phi_p - \phihat_\gamma\rangle + 2\langle \Theta_p,\partial\Theta_p\rangle + 2\langle \Theta_p,\Ad_\gamma[\phi_p,\Theta_\gamma-\Theta_p]\rangle\\
    +2\langle \Theta_p ,\Ad_\gamma\partial(\Theta_\gamma - \Theta_p)\rangle
    - \langle [\Theta_\gamma,\Theta_\gamma],\phi_p\rangle 
    -\langle [\Theta_p,\Theta_p],\phi_p\rangle
    +2\langle [\Theta_\gamma,\Theta_p],\phi_p\rangle \\
    - 2\langle\Theta_\gamma,\partial\Theta_p\rangle
  \end{multlined}\\
  & = \intSone
  \begin{multlined}[t]
    2\langle d\Theta_p,\Ad_\gamma\phi_p - \phihat_\gamma\rangle -
    2\left(\langle-\Theta_p,\Ad_\gamma[\phi_p,\Theta_\gamma-\Theta_p]\rangle + \langle\Theta_p,\Ad_\gamma\partial\Theta_p\rangle\right)\\
    +2\langle\Theta_p,\Ad_\gamma\partial\Theta_\gamma\rangle + 2\langle d\Theta_\gamma ,\phi_p\rangle + 2\langle d\Theta_p,\phi_p\rangle 
    + 2\langle[\Theta_\gamma,\Theta_p],\phi_p\rangle \\
    -2\langle \Theta_\gamma-\Theta_p,\partial\Theta_p\rangle
  \end{multlined}\\
  & = \intSone 
  \begin{multlined}[t]
    2\langle d\Theta_p, \Ad_\gamma\phi_p - \phihat_\gamma\rangle 
    - \left(2\langle \Theta_p,d(\Ad_\gamma\phi_p)\rangle - 2\langle \Theta_p, d\phihat_\gamma\rangle \right)\\
    +2\langle d\Theta_\gamma + d\Theta_p, \phi_p \rangle - 2\langle \Theta_\gamma - \Theta_p,d\phi_p\rangle + 2\langle \Theta_\gamma - \Theta_p,[\phi_p,\Theta_p]\rangle\\
    + 2\langle [\Theta_\gamma,\Theta_p],\phi_p\rangle
  \end{multlined}\\
  & = 2\intSone \begin{multlined}[t]
      \langle d\Theta_p,\Ad_\gamma\phi_p - \phihat_\gamma\rangle - \langle \Theta_p,d(\Ad_\gamma\phi_p - \phihat_\gamma)\rangle + \langle d(\Theta_\gamma-\Theta_p),\phi_p\rangle \\
      - \langle \Theta_\gamma - \Theta_p,d\phi_p\rangle
      \end{multlined}
      \\
  & = d\left( 2\intSone \langle \Theta_p,\Ad_\gamma\phi_p - \phihat_\gamma\rangle
        + \langle \Theta_\gamma - \Theta_p,\phi_p\rangle
        \right)\\
  & = d\rho_{(p,\gamma)},
\end{align*}
as we needed to show.

We now turn to item (1), using the definition of the adjoint action given in Definition \ref{def:string_xmod}.
Since $\mu$ is defined as the descent of a 1-form along the quotient map $q\colon P\Omega G \ltimes U(1) \to \Omegahat{G}$ (hence a submersion), we calculate $\widehat{\Ad}^*\mu - \pr_2^*\mu$ by pulling this back along $\id\times q$, using the following diagram:
\[
  \xymatrix{
    &PG\times P\Omega G \ltimes U(1) \ar[rr]^{\widetilde{\Ad}}\ar[dd]_(0.35){\pr_{12}}|!{[dl];[dr]}\hole \ar[dl]_{\id\times q}&& P\Omega G\ltimes U(1) \ar[dd]^{\pr_1}  \ar[dl]^q\\
    PG\times \Omegahat{G} \ar[rr]^(0.7){\widehat{\Ad}} \ar[dd]_(0.35){\id\times \pi} && \Omegahat{G} \ar[dd]^(0.35){\pi} \\
    & PG\times P\Omega G \ar[rr]_(0.4){\Ad}|!{[ur];[dr]}\hole \ar[dl]^(0.4){\id\times \ev_{2\pi}}&& P\Omega G\ar[dl]^{\ev_{2\pi}}\\
    PG\times \Omega G \ar[rr]_{\Ad} && \Omega G
  }
\]
We have that 
\begin{align*}
  &(\id\times q)^*\left(\widehat{\Ad}^*\mu - \pr_2^*\mu\right) \\
  & = \widetilde{\Ad}^*q^*\mu - \pr_{23}^*q^*\mu\\
  & = \widetilde{\Ad}^*\tilde \mu - \pr_{23}^*\tilde \mu\\
  & = \widetilde{\Ad}^*\left(\pr_2^*\theta^{-1}d\theta + \pr_1^*\int_{[0,2\pi]}\ev^*R\right) - \left(\pr_3^*\theta^{-1}d\theta + \pr_{2}^*\int_{[0,2\pi]}\ev^*R \right)\\
  & = \pr_{12}^*d\int_{[0,2\pi]}(\id\times \ev)^*\rho + \pr_{12}^*\Ad^*\int_{[0,2\pi]}\ev^*R - \pr_{12}^*\int_{[0,2\pi]}\ev^*R\\
  & = \pr_{12}^*\left(d\int_{[0,2\pi]}(\id\times \ev)^*\rho + \int_{[0,2\pi]}(\id\times \ev)^*(\Ad^*R - R)\right)\\
  & = \pr_{12}^*\left(d\int_{[0,2\pi]}(\id\times \ev)^*\rho + \int_{[0,2\pi]}d(\id\times \ev)^*\rho\right)\\
  & = \pr_{12}^*((\id\times\ev_{2\pi})^*\rho - (\id\times\ev_0)^*\rho)\\
  & = (\id\times q)^*(\id\times \pi)^* \rho,
\end{align*}
where we have applied the fibrewise Stokes' theorem (Proposition \ref{prop:stokes}) in the second-to-last line, and then the fact $(\id\times\ev_0)^*\rho=0$. 
This latter point follows from the fact that the function $\ev_0$ is constant at the constant loop $1\in \Omega G$. 
As $\id\times q$ is a submersion, pullback along it is injective, so we have, finally, $\widehat{\Ad}^*\mu - \pr_2^*\mu = \rho$.
\end{proof}

\subsection{The Chern--Simons connective structure}

Suppose that the principal $G$-bundle $Q \to X$ has connection 1-form $A$. In the following Definition we refer to the items (1)--(6) in Definition \ref{def:conn_structure}. \label{form:CS2-gerbe curving} \label{form:CS2-gerbe 2-curving} \label{form:CS2-gerbe alpha}

\begin{definition}\label{def:CS_conn}
The data for the rigid connective structure on $CS(Q)$ is given by 
\begin{enumerate}[(i)]
\item
The bundle gerbe connection in (1) is given by the 1-form $\nabla$ in Lemma \ref{lemma:connection}.
\item
The curving 2-form  $\beta_A$ on $Q \times PG$ in (2) is given by 
\[
  \beta_A = \pr_2^*B - \pi^*\langle\pr_1^*A, \pr_2^* \Thetahat^G\rangle,
\] 
where $B$ is defined in Lemma \ref{lemma:curving}, and $\Thetahat^G$ is the right Maurer--Cartan form on $G$.
\item
The 2-curving in (3) is the (negative of the) Chern--Simons 3-form
\[
  -CS(A) = -\langle A, F_A\rangle + \frac16 \langle A, [A, A]\rangle,
\]
where $F_A$ is the curvature of $A$.
\item 
The 1-form $\a$ on $Z_2 = Q\times PG^2$ for (4)--(6) is given at $(x,p,q)\in Q\times PG^2$  by 
\[
  \alpha_{(x,p,q)}=2\intSone \langle \Theta_p ,\phihat_q \rangle.
\]
(Recall that the 1-form $\alpha$ is denoted $\alpha_{2,1}$ in the course of the proof of Proposition \ref{prop:conn_exist}.)
\end{enumerate}
That this defines a rigid connective structure is proved below.
\end{definition}

\label{form:kappa}
Note that $d\alpha = \pr_{23}^*\pi^*\kappa$, where $\kappa \in \Omega^2(G\times G)$ is given by $\kappa = \langle\pr_1^*\Theta,\pr_2^*\Thetahat\rangle$. This also satisfies $d\kappa = \delta_h(\omega)$, where $\omega$ is the curvature of the basic gerbe, from Lemma \ref{lemma:basic_gerbe_curvature}, and appears in \cite{Wal}.

\begin{proposition}\label{prop:conn_str}
The forms $\nabla$, $\beta_A$ and $CS(A)$ give a connective structure on $CS(Q)$.
\end{proposition}

\begin{proof}
The connection $\nabla$ is a bundle gerbe connection, and there is no further condition on it. 
The 2-form $\beta_A$ is a curving for $\nabla$, which verifies conditions (1) and (2) of Definition \ref{def:conn_structure}; it remains to prove condition (3): $\pi^*\delta_h(CS(A)) = d\beta_A$.

We first calculate $\delta_h(-CS(A))$, using the identity $(\Act^* A)_{(y,g)} =\Ad_{g^{-1}}(A_y + dg\,g^{-1})$, in stages.
First,
\begin{align*}
\left(\Act^* \langle A,F_A \rangle - \pr_1^*\langle A,F_A \rangle\right)_{(y,g)} &= \langle \Ad_{g^{-1}}(A_y + dg\,g^{-1}), \Ad_{g^{-1}} (F_A)_y \rangle - \langle A_y,(F_A)_y \rangle\\
& = \langle dg\,g^{-1},(F_A)_y\rangle,
\end{align*}
and also then
\begin{align*}
&\tfrac16 \left(\Act^*\langle A, [A, A]\rangle -\pr^*\langle A, [A, A]\rangle\right)_{(y,g)}\\
&= \tfrac16 \langle \Ad_{g^{-1}}(A_y + dg\,g^{-1}) , [\Ad_{g^{-1}}(A_y + dg\,g^{-1}),\Ad_{g^{-1}}(A_y + dg\,g^{-1})]\rangle - \tfrac16 \langle A_y,[A_y,A_y]\rangle\\
& = \tfrac16 \langle A_y + dg\,g^{-1} , [A_y + dg\,g^{-1},A_y + dg\,g^{-1}]\rangle - \tfrac16 \langle A_y,[A_y,A_y]\rangle\\
& = \begin{multlined}[t]
\tfrac16 \left(
\langle dg\,g^{-1},[A_y,A_y]\rangle 
+ \langle dg\,g^{-1},[dg\,g^{-1},A_y]\rangle 
+ \langle dg\,g^{-1},[A_y,dg\,g^{-1}]\rangle \right.\\
+ \langle dg\,g^{-1},[dg\,g^{-1},dg\,g^{-1}]\rangle 
+ \langle A_y,[dg\,g^{-1},A_y] \rangle
+ \langle A_y,[A_y,dg\,g^{-1}] \rangle\\
\left.+ \langle A_y,[dg\,g^{-1},dg\,g^{-1}] \rangle
\right)\end{multlined}\\
& = \tfrac12 \left(\langle dg\,g^{-1},[A_y,A_y]\rangle + \langle A_y,[dg\,g^{-1},dg\,g^{-1}]\rangle\right) + \omega_g,
\end{align*}
where $\omega\in \Omega^3(G)$ is defined in Lemma \ref{lemma:basic_gerbe_curvature}. 
So we have
\begin{align*}
\delta_h(-CS(A))_{y,g} & = -\langle dg\,g^{-1},(F_A)_y\rangle
+\tfrac12 \langle dg\,g^{-1},[A_y,A_y]\rangle +\tfrac12 \langle A_y,[dg\,g^{-1},dg\,g^{-1}]\rangle + \omega_g\\
& = \omega_g -\langle dA_y,\Thetahat^G_g\rangle + \langle A_y,d\Thetahat^G_g\rangle
\\
& = \omega_G -d\langle\pr_1^*A, \pr_2^* \Thetahat^G\rangle_{(y,g)}.
\end{align*}
If we then pull this up to $Q\times PG$, we get
\begin{align*}
&\pi^*\omega - \pi^*d\langle\pr_1^*A, \pr_2^* \Thetahat^G\rangle \\
& = dB - d\pi^*\langle\pr_1^*A, \pr_2^* \Thetahat^G\rangle \\
& = d\beta_A,
\end{align*}
as we needed to show.
\end{proof}

However, more is true: the data in Definition \ref{def:CS_conn} will form a \emph{rigid} connective structure. We will next check condititions (4)--(6) in Definition \ref{def:conn_structure}.

\begin{lemma}\label{lemma: delta_h a = delta_v alpha}
Recall from Definition \ref{def:CS_b2g} (ix) the section $M\colon (PG\ltimes \Omega G)^2 \to \delta_h(PG\times \Omegahat{G})$. Then $M$ satisfies
\[
  \delta_v\alpha = M^*(\pr_1^*\nabla - \mult^*\nabla + \pr_2^*\nabla)= M^*\delta_h \nabla,
\]
which is condition (4) in Definition \ref{def:conn_structure} (here $\mult$ is the multiplication map on the semidirect product $PG\ltimes \Omega G$).
\end{lemma}

\begin{proof}
Recall that $\nu_{(\gamma,\eta)}= 2\intSone\langle\Theta_\gamma,\phihat_\eta\rangle$ satisfies $\pi^*\nu = \delta(\nabla)$, where here $\delta$ is with respect to the three maps $\widehat{\Omega G}^2 \to \widehat{\Omega G}$, namely multiplication and the two projections.
Using the factorisation 
\begin{equation}\label{eq:semidirect_prod_mult}
  \begin{multlined}[t]
  (p,\gamma,q,\eta) \mapsto (pq,q;\gamma,\eta) \xmapsto{\id\times\inv\times\id^2} (pq,q^{-1},\gamma,\eta) \xmapsto{\id\times \Ad\times \id} (pq,\Ad_{q^{-1}}(\gamma),\eta) \\ \xmapsto{\id\times m} (pq,\Ad_{q^{-1}}(\gamma)\eta) \end{multlined}
\end{equation}
of $\mult$, we will now calculate $M^*\delta_h \nabla$ at a point $(p,\gamma,q,\eta)\in (PG\times \Omega G)^2$. 
\begin{align*}
  &(M^*\delta_h \nabla)_{(p,\gamma,q,\eta)} \\
  & = M^*\big[\mu_{\hat\gamma} - \mult^*\mu + \mu_{\hat\eta}\big] - \epsilon_{(p,\gamma)} + \epsilon_{(pq,\Ad_{q^{-1}}(\gamma)\eta)} - \epsilon_{(q,\eta)}\\
  & = \begin{multlined}[t]
    M^*\big[\mu_{\hat\gamma} + \mu_{\hat\eta} - (\id\times \inv\times \id^2)^*(\id\times \Ad \times \id)^*(\pr_1^*\mu + \pr_2^*\mu - \pi^*\nu)\big]\\
    - \epsilon_{(p,\gamma)} 
    + \epsilon_{(pq,\Ad_{q^{-1}}(\gamma)\eta)} - \epsilon_{(q,\eta)}
    \end{multlined}\\
  & = M^*\big[\mu_{\hat\gamma} + \mu_{\hat\eta} - (\inv\times \id)^*\Ad^*\pr_1^*\mu - \mu_{\hat\eta}\big] + \nu_{(\Ad_{q^{-1}}(\gamma),\eta)}- \epsilon_{(p,\gamma)} + \epsilon_{(pq,\Ad_{q^{-1}}(\gamma)\eta)} - \epsilon_{(q,\eta)}\\
  & = M^*\big[\mu_{\hat\gamma} - \mu_{\widehat{\Ad}_{q^{-1}}\hat\gamma}\big]+ \nu_{(\Ad_{q^{-1}}(\gamma),\eta)}  + \epsilon_{(pq,\Ad_{q^{-1}}(\gamma)\eta)} - \epsilon_{(q,\eta)}- \epsilon_{(p,\gamma)}\\
  & = M^*\inv_q^*\big[\mu_{\hat\gamma} - \mu_{\widehat{\Ad}_q\hat\gamma}\big]+ \nu_{(\Ad_{q^{-1}}(\gamma),\eta)}  + \epsilon_{(pq,\Ad_{q^{-1}}(\gamma)\eta)} - \epsilon_{(q,\eta)}- \epsilon_{(p,\gamma)}\\
  & = -\rho_{q^{-1},\gamma} + \nu_{(\Ad_{q^{-1}}(\gamma),\eta)}  + \epsilon_{(pq,\Ad_{q^{-1}}(\gamma)\eta)} - \epsilon_{(q,\eta)}- \epsilon_{(p,\gamma)}.
\end{align*}
Now we can calculate that
\begin{align*}
  \nu_{(\Ad_{q^{-1}}(\gamma),\eta)}
  & = 2 \intSone \langle\Theta_{\Ad_{q^{-1}}(\gamma)}, \phihat_\eta\rangle\\
  & = 2 \intSone \langle \Ad_{q^{-1}}(\Theta_\gamma +\Thetahat_q - \Ad_{\gamma^{-1}}\Thetahat_q),\phihat_\eta\rangle,
\end{align*}
and also
\begin{align*}
  &\epsilon_{(pq,\Ad_{q^{-1}}(\gamma)\eta)} \\
  & = 2 \intSone \langle \Theta_{pq},\phihat_{\Ad_{q^{-1}}(\gamma)\eta}\rangle\\
  & = 2 \intSone \langle \Ad_{q^{-1}}(\Theta_p + \Thetahat_q ),
    \phihat_{\Ad_{q^{-1}}(\gamma)} + \Ad_{\Ad_{q^{-1}}(\gamma)}\phihat_\eta
      \rangle\\
  & = 2 \intSone \langle \Ad_{q^{-1}}(\Theta_p + \Thetahat_q),
    \Ad_{q^{-1}}(\phihat_\gamma + \Ad_\gamma \phihat_q-\phihat_q)
    + \Ad_{\Ad_{q^{-1}}(\gamma)}\phihat_\eta
     \rangle\\
  & = 2 \intSone \langle \Theta_p + \Thetahat_q,
    \phihat_\gamma + \Ad_\gamma \phihat_q-\phihat_q + \Ad_{\gamma q}\phihat_\eta
     \rangle\\
  & = \begin{multlined}[t]
    \epsilon_{(p,\gamma)} + 2 \intSone \Big(\langle
         \Theta_p, \Ad_\gamma\phihat_q\rangle - \langle \Theta_p,\phihat_q\rangle + \langle \Theta_p,\Ad_{\gamma q}\phihat_\eta\rangle \\
        + \langle\Thetahat_q, \phihat_\gamma\rangle + \langle\Thetahat_q ,\Ad_\gamma\phihat_q \rangle - \langle\Thetahat_q,\phihat_q\rangle + \langle\Thetahat_q,\Ad_{\gamma q}\phihat_\eta\rangle\Big).
    \end{multlined}
\end{align*}
Therefore, we have:
\begin{align*}
  &(M^*\delta_h \nabla)_{(p,\gamma,q,\eta)} \\
  & =  \begin{multlined}[t]
       2\intSone \Big\{
        -\langle\Thetahat_q,\Ad_\gamma \phihat_q + \phihat_\gamma\rangle + \langle\Theta_\gamma + \Thetahat_q,\phihat_q\rangle
        + \langle\Theta_\gamma + \Thetahat_q - \Ad_{\gamma^{-1}}\Thetahat_q,\Ad_q \phihat_\eta  \rangle
        \\
        + \langle\Theta_p, \phihat_\gamma\rangle - \langle\Theta_p, \phihat_q\rangle + \langle\Theta_p,\Ad_\gamma \phihat_q\rangle + \langle\Theta_p, \Ad_\gamma\Ad_q \phihat_\omega\rangle -\langle\Ad_q \Theta_q, \phihat_q \rangle\\
        + \langle\Ad_q \Theta_q, \phihat_\gamma\rangle + \langle\Ad_q\Theta_q ,\Ad_\gamma\phihat_q\rangle + \langle\Ad_q\Theta_q,\Ad_{\gamma q}\phihat_\eta\rangle
        -\langle\Theta_q,\phihat_\eta\rangle  - \langle\Theta_p, \phihat_\gamma\rangle
      \Big\}
    \end{multlined}\\
  & = 2\intSone \begin{multlined}[t]
       \Big\{
         \langle\Thetahat_q ,\phihat_q \rangle
        - \langle \Theta_p ,\phihat_q \rangle
        - \langle\Thetahat_q  ,\phihat_q \rangle
        -  \langle\Thetahat_q,\Ad_\gamma \phihat_q + \phihat_\gamma\rangle
        + \langle\Theta_\gamma,\phihat_q \rangle\\
        + \langle\Thetahat_\gamma,\Ad_{\gamma q} \phihat_\eta \rangle
        + \langle \Theta_q ,\phihat_\eta \rangle
        - \langle\Thetahat_q ,\Ad_{\gamma q}\phihat_\eta \rangle
        + \langle\Theta_p ,\Ad_\gamma\phihat_q \rangle\\
        + \langle\Theta_p ,\Ad_{\gamma q}\phihat_\eta \rangle
        + \langle\Thetahat_q,\phihat_\gamma \rangle
        + \langle\Thetahat_q,\Ad_\gamma\phihat_q \rangle
        + \langle\Thetahat_q ,\Ad_{\gamma q}\phihat_\eta \rangle
        - \langle\Theta_q ,\phihat_\eta \rangle
      \Big\}
    \end{multlined}\\
  & = 2\intSone  
        \langle\Thetahat_\gamma 
          + \Theta_p ,\Ad_\gamma \phihat_q\rangle
        + \langle\Thetahat_\gamma + \Theta_p , \Ad_{\gamma q}\phihat_\eta \rangle -\langle\Theta_p ,\phihat_q \rangle
    \\
  & = 2\intSone \langle\Theta_{p\gamma} ,\phihat_{q\eta}\rangle  -\langle \Theta_p ,\phihat_q\rangle
  \\
    & = \alpha_{(p\gamma,q\eta)} - \alpha_{(p,q)},
\end{align*}
as we needed to show.
\end{proof}

\begin{lemma}\label{lemma:rigid_conn_str_2}
The condition (5) in Definition \ref{def:conn_structure} holds for $\beta_A$ and $\alpha$, namely $d\alpha = \delta_h\beta_A$.
\end{lemma}

\begin{proof}
Let us first calculate $d\alpha$, taking into account the fact that $\alpha$ is independent of the $Q$-coordinate, that $PG^2$ is a Lie group, and that $\alpha$ is left invariant in the first coordinate.
Take a pair of tangent vectors $(pX_1,qY_1),(pX_2,qY_2) \in T_{(p,q)}PG^2$, and evaluate, where by an abuse of notation we are suppressing any mention of coordinates on $Q$:
\begin{align*}
&d\alpha_{(p,q)}\left((pX_1,qY_1),(pX_2,qY_2)\right) \\
& = \tfrac12 \Big\{
(pX_1,qY_1)\cdot \alpha(X_2,qY_2) - (pX_2,qY_2)\cdot \alpha(X_1,qY_1) 
- \alpha(p[X_1,X_2],q[Y_1,Y_2])
\Big\}.
\end{align*}
To continue, we need the derivative $(pX,qY)\cdot \alpha(pX',qY')$, which is
\begin{align*}
(pX,qY)\cdot \alpha(pX',qY') &=
\left.\frac{d}{dt} \left\{ 2\intSone \langle X',\partial_\theta[q(1+tY+O(t^2))](1-tY+O(t^2))q^{-1}\rangle\right\}\right|_{t=0}\\
& = \begin{multlined}[t]
2\intSone  \frac{d}{dt}\langle X',(\partial_\theta q)(1+tY+O(t^2))(1-tY+O(t^2))q^{-1}\rangle\big|_{t=0} \\
+
 \frac{d}{dt}\langle X',qt\partial_\theta(Y)(1-tY)q^{-1}+O(t^2)\rangle\big|_{t=0}  
\end{multlined}\\
& = 2\intSone  \frac{d}{dt}\langle X',(\partial_\theta q)(1+O(t^2))q^{-1}\rangle\big|_{t=0} 
+\langle X',q\partial_\theta(Y)q^{-1}\rangle \\
& = 2\intSone  \langle X',q\partial_\theta(Y)q^{-1}\rangle.
\end{align*}
Using this, we find
\[
d\alpha_{(p,q)}\left((pX_1,qY_1),(pX_2,qY_2)\right)=
2\intSone \langle X_2,\Ad_q\partial Y_1\rangle - \langle X_1 ,\Ad_q \partial Y_2\rangle - \langle [X_1,X_2],\phihat_q\rangle.
\]
Let us check some particular evaluations of $B$, noting that it is left-invariant:
\begin{align*}
  B(Y_1,\Ad_q^{-1}X_2) 
  &= \intSone \langle Y_1 ,\Ad_{q^{-1}}[X_2,\phihat_q]\rangle + \langle Y_1 ,\Ad_{q^{-1}}\partial X_2\rangle
\end{align*}
and 
\begin{align*}
B(\Ad_{q^{-1}}X_1,\Ad_{q^{-1}}X_2) 
& = \intSone \langle q^{-1}X_1q,\partial(q^{-1}X_2 q)\rangle\\
& = \intSone \langle X_1,[X_2,\phihat_q]\rangle + \langle X_1,\partial X_2\rangle.
\end{align*}
We also need to calculate $-\delta_h \pi^*\langle\pr_1^*A, \pr_2^* \Thetahat^G\rangle = -\pi^* \delta_h\langle\pr_1^*A, \pr_2^* \Thetahat^G\rangle$, and so we need
\begin{align*}
  \delta_h\langle\pr_1^*A, \pr_2^* \Thetahat^G\rangle_{(y,g,h)}&=\langle A_{yg},\Thetahat_h\rangle - \langle A_y,\Thetahat_{gh}\rangle 
  + \langle A_y,\Thetahat_g\rangle\\
  &=\langle \Ad_{g^{-1}}A_y + \Theta_g,\Thetahat_h \rangle
  - \langle A_y,\Thetahat_g+\Ad_g\Thetahat_{h}\rangle 
  + \langle A_y,\Thetahat_g\rangle\\
  &= \langle A_y,\Ad_g\Thetahat_h\rangle + \langle\Theta_g,\Thetahat_h \rangle
  - \langle A_y,\Thetahat_g+\Ad_g\Thetahat_{h}\rangle 
  + \langle A_y,\Thetahat_g\rangle\\
  & = \langle\Theta_g,\Thetahat_h \rangle = \kappa_{(g,h)}.
\end{align*}
And now we can calculate $\delta_h(B)_{(p,q)}$:
\begin{align*}
&\delta_h(B)_{(p,q)}\left((pX_1,qY_1),(pX_2,qY_2)\right) \\
& = B(X_1,X_2) - B(\Ad_{q^{-1}}X_1+Y_1,\Ad_{q^{-1}}X_2+Y_2) + B(Y_1,Y_2)\\
 & = B(X_1,X_2)- B(\Ad_{q^{-1}}X_1,\Ad_{q^{-1}}X_2)
- B( Y_1,\Ad_{q^{-1}}X_2)
- B(\Ad_{q^{-1}}X_1, Y_2)
\\
& = \begin{multlined}[t]
\frac12 \intSone \langle X_1,\partial X_2\rangle - \langle X_2,\partial X_1\rangle  
-\langle\Ad_{q^{-1}}X_1,\partial(q^{-1}X_2q)\rangle
+\langle\Ad_{q^{-1}}X_2,\partial(q^{-1}X_1q)\rangle\\
- \langle Y_1, \partial(q^{-1}X_2q)\rangle + 
\langle \Ad_{q^{-1}}X_2 , \partial Y_1\rangle 
-\langle \Ad_{q^{-1}}X_1 , \partial Y_2\rangle
+ \langle Y_1, \partial(q^{-1}X_2q)\rangle 
\end{multlined}\\
\\
&=\begin{multlined}[t]
\frac12 \intSone \langle X_1,\partial X_2\rangle - \langle X_2,\partial X_1\rangle - \langle X_1,\partial X_2\rangle - \langle X_1,[X_2,\phihat_q] \rangle\\
-\langle Y_1,\partial(\Ad_{q^{-1}}X_2)\rangle 
+ \langle \Ad_{q^{-1}}, \partial Y_1\rangle
+ \langle X_2,\partial X_1\rangle + \langle X_2,[X_1,\phihat_q]\rangle\\
-\langle\Ad_{q^{-1}}X_1,\partial Y_2\rangle
+ \langle Y_2,\partial(\Ad_{q^{-1}}X_1)\rangle
\end{multlined}
\\
&=\begin{multlined}[t]
\frac12 \intSone \langle X_2,\Ad_q\partial Y_1\rangle
- \langle X_1,\Ad_q Y_2\rangle
-2\langle [X_1,X_2,\phihat_q\rangle\\
-\langle Y_1,\partial(\Ad_{q^{-1}}X_2)\rangle
+ \langle Y_2,\partial(\Ad_{q^{-1}}X_1)\rangle
\end{multlined}
\\
&=\begin{multlined}[t]
\frac12\intSone \Big\{\langle X_2,\Ad_q\partial Y_1\rangle 
+ \langle \partial Y_1, \Ad_{q^{-1}} X_2\rangle
- \langle X_1,\Ad_q\partial Y_2\rangle
- \langle \partial Y_2,\Ad_{q^{-1}}X_1 \rangle \\
- 2\langle [X_1,X_2,\phihat_q\rangle\Big\} + \tfrac12 \left(-\langle X_2(2\pi),\Ad_{q(2\pi)}Y_1(2\pi)\rangle + \langle X_1(2\pi),\Ad_{q(2\pi)}Y_2(2\pi)\rangle \right)
\end{multlined}\\
&=\begin{multlined}[t]
\intSone \Big\{\langle X_2,\Ad_q\partial Y_1\rangle
- \langle X_1,\Ad_q\partial Y_2\rangle
- \langle [X_1,X_2,\phihat_q\rangle\Big\} + \pi^*\kappa_{(p,q)}\left((pX_1,qY_1),(pX_2,qY_2)\right)
\end{multlined}\\
&= d\alpha_{(p,q)}\left((pX_1,qY_1),(pX_2,qY_2)\right) + \pi^*\kappa_{(p,q)}\left((pX_1,qY_1),(pX_2,qY_2)\right).
\end{align*}
Thus
\[
\delta_h\beta_A= \delta_h B - \delta_h\pi^*\langle\pr_1^*A, \pr_2^* \Thetahat^G\rangle = d\alpha+ \pi^*\kappa -\pi^*\kappa =d\alpha,
\]
as we needed to show.
\end{proof}

\begin{lemma}\label{lemma:rigid_conn_str_3}
We have $\delta_h\alpha = 0$, which is Definition \ref{def:conn_structure} (6).
\end{lemma}
\begin{proof}
Since integration is linear, it is sufficient to show that
\[
  \delta_h\Big( \langle \Theta_p ,\phihat_q\rangle\Big)  = 0.
\]
We shall first evaluate this form at $(p_1,p_2,p_3)\in PG^3$. 
\begin{align*}
   \delta_h \langle \Theta_p,\phihat_q\rangle_{(p_1,p_2,p_3)} 
   &=  
      \langle \Theta_{p_2} - \Theta_{p_1p_2},\phihat_{p_3}\rangle + \langle \Theta_{p_1},\phihat_{p_2p_3}-\phihat_{p_2}\rangle 
    \\
   &=  \langle \Theta_{p_2} - \Ad_{p_2^{-1}}(\Theta_{p_1} + \Thetahat_{p_2} 
      ) ,\phihat_{p_3}\rangle
      + \langle \Theta_{p_1},\Ad_{p_2}\phihat_{p_3}\rangle
    \\
    &=  \langle\Ad_{p_2^{-1}}( {\Thetahat_{p_2}}  
          - { \Theta_{p_1}} - {\Thetahat_{p_2}} 
          ),\phihat_{p_3}\rangle
           + {\langle\Ad_{p_2^{-1}}\Theta_{p_1},\phihat_{p_3}\rangle}
    \\
    &=  0.
\end{align*}
Hence $\delta_h\alpha = 0$, as we needed to show.
\end{proof}

We now arrive at last at the main result of this paper, the raison d'etre of the theory developed so far. Denote by $p_1(Q,A)$ the standard Chern--Weil representative for the first Pontryagin class associated to the $G$-bundle $Q$ equipped with connection $A$. The following theorem is our version of \cite[Theorem~2.7]{Wal}, modulo choices of signs of generators of cohomology.

\begin{theorem}\label{thm:CSrb2g_conn_str}
The data in Definition \ref{def:CS_conn} gives a rigid connective structure on $CS(Q)$, with curvature 4-form
\[
  \curv(\nabla,\beta_A,-CS(A)) = -\langle F_A,F_A\rangle,
\]
 which in the case of $G=\Spin(n)$ ($n\geq5$), satisfies $\curv(\nabla,\beta_A,-CS(A))/2\pi = -\frac12 p_1(Q,A)$.
\end{theorem}

\begin{proof}
The first claim follows from adding Lemmata \ref{lemma:rigid_conn_str_2}, \ref{lemma:rigid_conn_str_3} and \ref{lemma: delta_h a = delta_v alpha} to Proposition \ref{prop:conn_str}.
That the 4-form curvature, the descent along $Q\to X$ of $-d\,CS(A)$, is $-\langle F_A,F_A\rangle$ is a standard result (see e.g.\ \cite[Proposition 1.27]{Freed_95}). When $G=\Spin(n)$, the first Pontryagin form is defined to be (a certain multiple of) the \emph{trace} of $F_A^2$, hence a \emph{negative} multiple of the inner product---we just have to check the coefficient is correct. From Example \ref{eg:normalised_Killing_forms} we have that $\langle X,Y\rangle = -\tr(XY)/8\pi$ for $\fg = \mathfrak{spin(n)} \simeq \mathfrak{so}(n)$, $n\geq 5$, so:
\begin{align*}
\curv(\nabla,\beta_A,-CS(A))/2\pi &= -\langle F_A,F_A\rangle/2\pi \\
&= \big(\tr(F_A^2)/8\pi\big)/2\pi \\
& = \tfrac{1}{2}\tr(F_A^2)/8\pi^2\\
& = -\tfrac{1}{2}p_1(Q,A). \qedhere
\end{align*}
\end{proof}

\begin{remark}
If one takes the connection and curving on the basic gerbe on $G$ from \cite{MS03}, then using the 1-form $\varepsilon_{MS}$ from Remark \ref{rem:MS_epsilon_correction}, the only other necessary change to the rigid connective structure in Theorem \ref{thm:CSrb2g_conn_str} is that one replaces $\alpha$ by $\alpha + \delta_h(\varepsilon_{MS})$.
The 2-curving (the negative of the Chern--Simons form) and the 4-form curvature (the first fractional Pontryagin form) are unchanged.
\end{remark}

\begin{corollary}\label{cor:CSrb2g_rigidification_with_geom}
Modulo the previous Remark, the geometric content of the rigid bundle 2-gerbe given here---the \emph{actual differential forms} $\nabla$, $\beta_A$ and $-CS(A)$ making up the connective structure---is the same as that in \cite{CJMSW}\footnote{The paper's \S6, which deals with the Chern--Simons 2-gerbe, works with a general invariant symmetric bilinear form $\Phi$, but on p.607 give the formula for the Chern--Simons action using the trace, with a positive sign, as the calculation in the proof of \ref{thm:CSrb2g_conn_str}.}. 
Thus we have described a rigidification of the Chern--Simons 2-gerbe of \emph{loc.\ cit.}, including with connection, curving and 2-curving. 
Moreover, as we have a \emph{rigid} connective structure, this gives bundle 2-gerbe connection in the stronger sense of \cite[Definition 4.1]{Wal} (using the background definitions in \S 6.2 of the same).
\end{corollary}

\begin{proof}
This follows after one has the construction in Appendix \ref{app:weak_2-gerbes} of the associated bundle 2-gerbe with connective structure (Propositions \ref{prop:associated_connective_structure} and \ref{prop:associated_Waldorf_connective_structure}) and that the underlying rigid 2-gerbe here rigidifies the Chern--Simons bundle 2-gerbe.
\end{proof}

The point of Theorem \ref{thm:CSrb2g_conn_str}, then, is that it allows easier calculations with the geometry of the Chern--Simons gerbe, in particular in the presence of trivialisations, which we will treat in a forthcoming paper \cite{rb2g_III}.

\appendix


\section{Associated bundle 2-gerbes and trivialisations}
\label{app:weak_2-gerbes}

Here we give the details of what is only outlined above in Subsection~\ref{subsec:associated_b2gs}.
The economy of our definitions will be seen in comparison to the amount of work that needs to be done to verify the weaker general bundle 2-gerbe definitions from our rigid data.

\subsection{Additional technical requirements}

Before we start, we need to build up some technical material related to constructions with bundle gerbes, as well as the bigroupoid of bundle gerbes.

Recall the definition of a \emph{morphism} of bundle gerbes $(\varphi,f)\colon (E,Y)\to (F,Z)$ over a manifold $X$: it is a map $f\colon Y\to Z$ over the manifold $X$, together with an isomorphism $\varphi\colon E\simeq (f^{[2]})^* F$ respecting the bundle gerbe product. 
Viewing bundle gerbes as certain sort of Lie groupoids over $X$, a morphism is equivalent data to a functor over $X$. 
However, recall the definition of a \emph{stable isomorphism} of bundle gerbes from $(E,Y)$ to $(F,Z)$, both on $X$ (of which there are several equivalent versions), which is a trivialisation of the bundle gerbe $(E^*,Y)\otimes (F,Z) := \left(\pr_{12}^*E^*\otimes \pr_{34}^*F,Y\times_X Z\right)$.
Such a trivialisation consists of a principal $U(1)$-bundle $L\to Y\times_XZ$ plus a bundle isomorphism $\tau\colon \delta(L)\xrightarrow{\simeq} \pr_{12}^*E^*\otimes \pr_{34}^* F$ on $(Y\times_X Z)^{[2]}\simeq Y^{[2]}\times_X Z^{[2]}$ preserving bundle gerbe multiplication. 
We shall denote a stable isomorphism with this data by $(L,\tau)\colon (E,Y) \rightsquigarrow (F,Z)$.
If there is a morphism between a pair of bundle gerbes, then the bundle gerbes are also stably isomorphic, and we can construct a canonical such.

\begin{construction}\label{constr:morphism_to_stable_iso}
Given a morphism of bundle gerbes $(\varphi,f)\colon (E,Y)\to (F,Z)$ on $X$, consider the pullback $Y\times_X Z$, and recall that there is a map $Y\to Y\times_X Z$, $y\mapsto (y,f(y))$. 
Then the projection $Y\times_X Z\to Z$ is equal to the composite $Y\times_X Z \to Y \to Z$, so that the bundle on $Y^{[2]}\times_X Z^{[2]}$ gotten by pulling back $F$ along $\pr_{34}$ is isomorphic to the pullback of $E$ along $\pr_{12}$, and this isomorphism is the pullback along $\pr_{12}$ of $\varphi$. 
Thus, taking the bundle $(f\times \id)^*F$ on $Y\times_X Z$, we get a stable isomorphism $\Sigma(\varphi,f) = ((f\times \id)^*F,\tau_f)\colon (E,Y) \rightsquigarrow (F,Z)$ by forming the multiplication-preserving isomorphism $\tau_f \colon \delta((f\times \id)^*F) \xrightarrow{\simeq} \pr_{12}^*E^*\otimes \pr_{34}^*F$, given fibrewise by
\begin{align*}
	F^*_{(f(y_1),x_1)}\otimes F_{(f(y_2),x_2)} & \simeq F^*_{(f(y_1),x_1)}\otimes F^*_{(x_1,f(y_2))}\otimes F_{(x_1,f(y_2))}\otimes F_{(f(y_2),x_2)} \\
	& \simeq F^*_{(f(y_1),f(y_2))}\otimes F_{(x_1,x_2)}\\
	& \simeq  E^*_{(y_1,y_2)}\otimes F_{(x_1,x_2)}.
\end{align*}

Further, this construction is appropriately pseudofunctorial for the composition of bundle gerbe morphisms. In particular given morphisms $(\varphi,f)\colon (E,Y) \to (F,Z)$ and $(\psi,g)\colon (F,Z)\to (G,W)$ over $X$, there is an isomorphism 
\[
	\Sigma^2_{(\psi,g),(\varphi,f)}\colon \Sigma(\psi,g) \circ\Sigma(\varphi,f) \Rightarrow \Sigma((\psi,g)\circ (\varphi,f)
\]
between the associated stable isomorphisms.\footnote{The other coherence map for $\Sigma$, comparing the identity morphism of bundle gerbes and the identity stable isomorphism, is an equality, essentially by definition}

If we compute the composite $\Sigma(\varphi,f) \circ\Sigma(\psi,g)$ according to \cite[pp 31--32]{StevPhD}, we can use the fact there is a section of $Y\times_X Z\times_X W \to Y\times_X W$, so that instead of constructing a bundle on $Y\times_X W$ via descent, we can just pull back the constructed bundle on $Y\times_X Z\times_X W$ by the section.
The result is that the trivialising $U(1)$-bundle on $Y\times_X W$ for the composite stable isomorphism is $\mathbf{1}\otimes\mathbf{1} \otimes ((g\circ f)\times\id)^*G$ (where, recall, $\mathbf{1}$ denotes the trivial bundle). 
The canonical isomorphism to $((g\circ f)\times\id)^*G$ is then the data of a 2-arrow between stable isomorphisms, where the latter bundle is the trivialising bundle for the stable isomorphism $\Sigma((\psi,g)\circ (\varphi,f))$. 
In particular, the data of this 2-arrow is independent of $(\psi,g)$ and $(\varphi,f)$.
\end{construction}

Recall that bundle gerbes on a given manifold $X$ form a bicategory  with 1-arrows stable isomorphisms \cite[Proposition~3.9]{StevPhD}.
In a bicategory, we can talk about what it means for a 1-arrow to be an \emph{adjoint} inverse to another \cite[Proposition~6.2.1]{Johnson-Yau}.
This is a stronger and better-behaved notion than just being an inverse up to isomorphism, and it is the case that any pseudo-inverse can be strengthened to be an adjoint inverse (e.g.\ \cite[Proposition~6.2.4]{Johnson-Yau}).

\begin{lemma}\label{lemma:bundlegerbemorphisms}
Let $(E,Y)$ and $(F,Z)$ be bundle gerbes on $X$. The adjoint inverse of a stable isomorphism $(L,\tau)\colon (E,Y) \rightsquigarrow (F,Z)$ is $(\sigma^*L^*,\sigma^*(\tau)^*)$, where $\sigma\colon Z\times_X Y\to Y\times_X Z$ is the swap map and $(\tau)^*$ denotes the isomorphism between the dual bundles arising from $\tau$.
Thus if we start with a \emph{morphism} $(\varphi,f)\colon (E,Y)\to (F,Z)$ of bundle gerbes, the adjoint inverse of $\Sigma(\varphi,f)$ is $((\id\times f)^*F^*,\sigma^* (\tau_f)^*)\colon (F,Z) \rightsquigarrow (E,Y)$.
If the morphism $(\varphi,f)$ is an \emph{isomorphism} of bundle gerbes, so that $f$ is invertible, the adjoint inverse of $\Sigma(\varphi,f)$ is $\Sigma((f^{[2]})^*\varphi^{-1},f^{-1})$ of $(\varphi,f)$ (here $((f^{[2]})^*\varphi^{-1},f^{-1}) = (\varphi,f)^{-1}$).
\end{lemma}

\begin{proof}
The proof is a routine calculation using the definition.
\end{proof}

We need to use the concept of a \emph{mate} of a 2-commutative square in a bicategory, where two parallel edges are equipped with adjoints (on the same side), see for example \cite[Definition~6.1.12]{Johnson-Yau}. 
We also need the fact this is functorial with respect to pasting of squares, cf \cite[Proposition 2.2]{Kelly-Steet_74} (this follows in more general bicategories by strictification/coherence results, and the result in 2-categories from \emph{loc.\ cit.}).

Since every stable isomorphism has an adjoint inverse, we can take the mate in the bicategory of bundle gerbes of any (2-)commuting square of stable isomorphisms in either or both directions. 
In particular given a commuting square of \emph{morphisms}, there is a commuting square of stable isomorphisms, and hence we can take mates. 
The following lemma deals with a special case that we need to consider later.

\begin{lemma}\label{lemma:mate_using_isos}
Given a commutative square of bundle gerbe morphisms
\[
	\xymatrix{
		(E_1,Y_1) \ar[r]^{(\alpha,g)} \ar[d]_{(\phi_1,f_1)}^\simeq & (E_2,Y_2) \ar[d]^{(\phi_2,f_2)}_\simeq \\
		(F_1,Z_1) \ar[r]_{(\beta,h)}  & (F_2,Z_2) 
	}
\]
over $X$, where the vertical morphisms are bundle gerbe \emph{isomorphisms}, consider its mate
\[
	\xymatrix{
		(E_1,Y_1) \ar@{~>}[r]  & (E_2,Y_2) \\
		(F_1,Z_1) \ar@{~>}[r]^{\ }="t" \ar@{~>}[u] & (F_2,Z_2) \ar@{~>}[u]^{\ }="s"
		\ar@{=>}"s";"t"
	}
\]
in the bicategory of bundle gerbes and stable isomorphisms. Here the horizontal stable isomorphisms are those associated to the given horizontal morphisms, and the vertical stable isomorphisms are those adjoint inverses from Lemma~\ref{lemma:bundlegerbemorphisms}. 
Then the 2-arrow filling this square is an \emph{identity} 2-arrow and the square of stable isomorphisms is that associated to the commutative square 
\[
	\xymatrix{
		(E_1,Y_1) \ar[r]^{(\alpha,g)}  & (E_2,Y_2) \\
		(F_1,Z_1) \ar[r]_{(\beta,h)} \ar[u]_\simeq^{((f_1^{[2]})^*\phi_1^{-1},f_1^{-1})} & (F_2,Z_2) \ar[u]^\simeq_{((f_1^{[2]})^*\phi_2^{-1},f_2^{-1})}
	}
\]
of morphisms, using the inverses of the original vertical morphisms.
\end{lemma}

\begin{proof}
We use the last point in Lemma~\ref{lemma:bundlegerbemorphisms}, so that the upward-pointing vertical stable isomorphisms in the mate are those associated to the upward-pointing isomorphisms in the last diagram. Thus we only need to know that the 2-arrow in the mate diagram is the identity 2-arrow.
This is a routine calculation with pasting diagrams for bicategories, together with the following fact: the data of the coherence isomorphism $\Sigma^2_{(\varphi,f),(\psi,g)}$ only depends on the composite $(\varphi,f)\circ(\psi,g)$, and not on the individual factors, and hence $\Sigma^2_{(\phi_2,f_2),(\alpha,g)} (\Sigma^2_{(\beta,h),(\phi_1,f_1)})^{-1}$ is an identity 2-arrow.
\end{proof}

We also need to consider an analogue of Construction \ref{constr:morphism_to_stable_iso} when one has bundle gerbes with connection and curving. Following \cite[Definition 6.6]{Wal}, a stable isomorphism $(L,\tau)$ between bundle gerbes $(E,Y;\nabla_E,b_E)$ and $(F,Z;\nabla_F,b_F)$ (on $X$) with connections and curvings has a \emph{compatible connection} if there is a connection $\nabla_L$ on $L\to Y\times_X Z$, the isomorphism $\tau$ is connection-preserving for the induced connection on $\pr_{12}^*E^*\otimes\pr_{34}^*F$, and the curvature of $\nabla_L$ is \emph{equal} to $\pr_2^*b_F - \pr_1^*b_E \in \Omega^2(Y\times_X Z)$.
Here we will just say that the stable isomorphism \emph{has a connection}, since we have no need to discuss non-compatible connections.

\begin{construction}\label{constr:morphism_to_stable_iso_w_conn_curv}
We start with the basic details as in Construction \ref{constr:morphism_to_stable_iso}, in particular the stable isomorphism $\Sigma(\varphi,f)$.
Suppose given bundle gerbe connections $\nabla_E$ and $\nabla_F$ on $(E,Y)$ and $(F,Z)$ respectively, and a 1-form $a$ on $Y$ such that $\pi^*\delta(a) = \nabla_E - \varphi^*(f^{[2]})^*\nabla_F$.
We need to give a connection on $(f\times\id)^*F$ and show that the isomorphism $\tau_f$ is connection-preserving.

However, defining the connection on $(f\times\id)^*F$ to be $(f\times\id)^*\nabla_F - \pi^* \pr_1^* a$, the composite isomorphism in Construction \ref{constr:morphism_to_stable_iso} is connection-preserving, by the assumption on $a$ and $\varphi$, and because bundle gerbe multiplication preserves connections.

Now assume that there are also curvings $b_E$ and $b_F$ on $(E,Y)$ and $(F,Z)$ respectively, satisfying $b_E - f^*b_F = da$.
Then the curvature of $(f\times\id)^*\nabla_F - \pi^* \pr_1^* a$ can be easily checked to be equal to $\pr_2^*b_F - \pr_1^*b_E$, so we have a compatible connection on $\Sigma(L,\tau)$.
\end{construction}

Further, we will need the analogue of Lemma~\ref{lemma:bundlegerbemorphisms} with connections and curvings.
Following \cite[Proposition 1.13]{Waldorf_07}, the adjoint inverse of what amounts to a stable isomorphism with connection is constructed---for Waldorf's definition of bundle gerbe morphism. 
That this should work for the definition of stable morphism here won't be an issue, we just need to record the construction of a special case.
Attentive readers might wonder what the definition of the bicategory of bundle gerbes with connections and curvings, stable isomorphisms with (compatible) connections and 2-isomorphisms is. We take it to be the analogue of Stevenson's bicategory of bundle gerbes, only now equipped with connections and curvings as outlined here. 
That this should form a bicategory (with a forgetful 2-functor to Stevenson's bicategory) will follow by considering the descent problems in \cite[Chapter 3]{StevPhD} that arise when forming composites as consisting of $U(1)$-bundles equipped with connections.

We can now give the version of Lemma~\ref{lemma:bundlegerbemorphisms} with connections and curvings, where a morphism of bundle gerbes with connections and curvings means a morphism of the underlying bundle gerbes such that that the connection and curving on the domain are given by pullback from the codomain. We will only use the part of Lemma~\ref{lemma:bundlegerbemorphisms} for stable morphisms arising from plain morphisms.

\begin{lemma}\label{lemma:bundlegerbemorphisms_w_conn_curv}
Let $(E,Y)$ and $(F,Z)$ be bundle gerbes on $X$, with connections $\nabla_E$ and $\nabla_F$ and curvings $b_E$ and $b_F$. 
Suppose we have a morphism of bundle gerbes $(\varphi,f)\colon (E,Y)\to (F,Z)$, such that $\nabla_E=\varphi^*(f^{[2]})^*\nabla_F$ (this means that the connection on the stable isomorphism is just $(f\times \id)^*\nabla_F$), and that $f^*b_F = b_E$.
Then the adjoint inverse of the stable isomorphism $\Sigma(\varphi,f)$ with its connection is given by the adjoint inverse of the underlying stable isomorphism (from Lemma \ref{lemma:bundlegerbemorphisms}, together with the connection\footnote{recall that the underlying space of the dual bundle $F^*$ is the same as that of $F$, only equipped with the action via inversion of $U(1)$. The `dual connection' on $F^*$ is therefore just $-\nabla_F$.} $-(\id\times f)^*\nabla_F$.
\end{lemma}

\begin{proof}
Note that the compatibility condition of 2-arrows does not involve the curvings at all.
This proof only needs to check that the unit (and hence counit) of the adjunction consists of a connection-preserving isomorphism of bundles. 
But then the triangle identities hold automatically as this is just a property, not structure.
However, the isomorphism of $U(1)$-bundles underlying the 2-arrow, given on \cite[page 33]{StevPhD}, is built using bundle gerbe multiplication on $(F,Z)$ and so preserves connections.
\end{proof}

As a corollary, the analogue of Lemma~\ref{lemma:mate_using_isos} will hold using isomorphisms of bundle gerbes \emph{with connection and curving}.
In particular, it will hold in the case where the square of bundle gerbes is induced from a square of maps between surjective submersions $Y_i\to X$,
\[
	\xymatrix{
		&Y_3 \ar[r] \ar[dl] \ar[drr]|!{[r];[d]}{\hole} & Y_1 \ar[dl] \ar[dr]\\
		Y_2 \ar[r] & Y_0 \ar[rr]&& X\,,\\
	}
\]
and a bundle gerbe $\cE = (E,Y_0)$, and the connections and curvings on them all are given by pullback from a connection and curving on $\cE$.
This means that although we start from a strictly commuting square, the $U(1)$-bundle isomorphism underlying the nontrivial 2-arrow that is the mate of the identity 2-arrow will be a connection-preserving isomorphism.

\subsection{Associated bundle 2-gerbes}

The following definition is adapted from Definitions~7.1 and 7.3 in \cite{StevPhD}. 
Note that here we say `bundle 2-gerbe' for what \emph{loc.\ cit.} calls a `stable' bundle 2-gerbe.

\begin{definition}\label{def:bundle2gerbe}
A \emph{simplicial bundle gerbe} on a semisimplicial manifold $B_\bullet$ consists of the following:
\begin{enumerate}
\item A bundle gerbe $\mathcal{E}$ on $B_1$;
\item A bundle gerbe stable isomorphism $\overline{m}\colon d_0^*\cE \otimes d_2^*\cE \to d_1^*\cE$ over $B_2$;
\item A 2-arrow $\tilde \alpha$, the \emph{associator}, between stable isomorphisms over $B_3$ as in the following diagram: 
\[
	\begin{gathered}
	\xymatrix{
	\cE_{23}\otimes \cE_{12}\otimes \cE_{01} \ar[d]_{1\otimes d_3^*\overline{m}} \ar[r]^-{d_0^*\overline{m}\otimes 1} & 
	\cE_{13}\otimes \cE_{01} \ar[d]_{\ }="s"^-{d_2^*\overline{m}}
	\\
	\cE_{23}\otimes \cE_{02} \ar[r]^{\ }="t"_-{d_1^*\overline{m}}
	 & \cE_{03}
	\ar@{=>}"s";"t"_{\tilde \alpha}
	}
	\end{gathered}
	\quad :=
	\begin{gathered}
\xymatrix@C=1em{
			d_1^*d_0^*\cE\otimes d_3^*(d_0^*\cE\otimes d_2^*\cE) 
			\ar[d]^{1\otimes d_3^*\overline{m}}
			\ar@{=}[r] 
			& 
			d_0^*(d_0^*\cE\otimes d_2^*\cE)\otimes d_2^*d_2^*\cE 
			\ar[d]_{d_0^*\overline{m}\otimes 1} 
			\\
			d_1^*d_0^*\cE\otimes d_3^* d_1^*\cE 
			\ar@{=}[d]
			& 
			d_0^*d_1^*\cE\otimes d_2^*d_2^*\cE
			\ar@{=}[d]
			\ar@{=>}[]!DL;[dl]!UR_{\tilde \alpha}
			\\
			d_1^*(d_0^*\cE\otimes d_2^*\cE)
			\ar[d]_{d_1^*\overline{m}}
			&
			d_2^*(d_0^*\cE \otimes d_2^*\cE)
			\ar[d]_{d_2^*\overline{m}}
			\\
	d_1^*d_1^*\cE \ar@{=}[r] &  d_2^*d_1^*\cE
	}
		\end{gathered}
\]
where $\cE_{ab}$ denotes the pullback of $\cE$ by the composite $B_3\to B_1$ of face maps induced by the map $(a,b)\to (0,1,2,3)$ in the simplex category (thinking of a semisimplicial manifold as a presheaf $\Delta^{op}_{\mathrm{inj}} \to \Mfld$).

\item Such that the following pair of 2-arrows over $B_4$ are equal,  where by $\cE_{ab}$ we mean the pullback of $\cE$ by the composite $B_4\to B_1$ of face maps induced by the inclusion $(a,b) \to (0,1,2,3,4)$ in the simplex category $\Delta_{\mathrm{inj}}$: 
\begin{align*}
	\xymatrix@-1em{
		& \cE_{34}\otimes \cE_{23}\otimes \cE_{12}\otimes \cE_{01} \ar[r] \ar[ddl]^(0.6){\ }="t1" & \cE_{24}\otimes \cE_{12}\otimes \cE_{01}  \ar[ddl]_{\ }="s1" \ar[ddr] \\
		\\
		\cE_{34}\otimes \cE_{23}\otimes \cE_{02}  \ar[r] \ar[ddr]^(0.6){\ }="t3" & \cE_{24}\otimes \cE_{02} \ar[ddr]^(0.55){\ }="t2"_(0.4){\ }="s3" && \cE_{14}\otimes \cE_{01} \ar[ddl]_(0.35){\ }="s2"
		\\
		\\
		& \cE_{34}\otimes \cE_{03} \ar[r] & \cE_{04}
		\ar@{=>}"s1";"t1"_{!\simeq}
		\ar@{=>}"s2";"t2"_{d_3^*\tilde \alpha}
		\ar@{=>}"s3";"t3"_{d_1^*\tilde \alpha}
	}
	\\
	\\
	\xymatrix@-1em{
		& \cE_{34}\otimes \cE_{23}\otimes \cE_{12}\otimes \cE_{01}  \ar[r] \ar[ddl] \ar[ddr]^(0.6){\ }="t2" & \cE_{24}\otimes \cE_{12}\otimes \cE_{01}   \ar[ddr]_(0.4){\ }="s2" \\
		\\
		\cE_{34}\otimes \cE_{23}\otimes \cE_{02}  \ar[ddr]^(0.4){\ }="t1" && 
		\cE_{34}\otimes \cE_{13}\otimes \cE_{01}  \ar[r] \ar[ddl]_(0.35){\ }="s1"^{\ }="t3" & \cE_{14}\otimes \cE_{01} \ar[ddl]_{\ }="s3"\\
		\\
		& \cE_{34}\otimes \cE_{03} \ar[r] & \cE_{04}
		\ar@{=>}"s1";"t1"_{\id\otimes d_4^*\tilde \alpha}
		\ar@{=>}"s2";"t2"_{d_0^*\tilde \alpha\otimes\id}
		\ar@{=>}"s3";"t3"_{d_2^*\tilde \alpha}
	}
\end{align*}
\end{enumerate}
A \emph{bundle 2-gerbe} is the special case where the semisimplicial manifold is that associated to a surjective submersion $Y\to X$, namely $Y^{[\bullet+1]}$.
A \emph{strict} bundle 2-gerbe is a bundle 2-gerbe where the stable isomorphism $\overline{m}$ arises from a plain morphism of bundle gerbes, and $\tilde \alpha$ is the identity 2-arrow.
\end{definition}

Our framework is defined specifically to avoid having to deal with the complexity of this definition, which, even as written, is hiding a lot of detail.

We note that given a map of semisimplicial manifolds $B'_\bullet \to B_\bullet$, and a simplicial bundle gerbe on $B_\bullet$, one can pull back to get a simplicial bundle gerbe on $B'_\bullet$.
It is by this process that the Chern--Simons bundle 2-gerbe is constructed (\cite[Lemma 5.4]{CJMSW}), whereby one considers the (strictly) multiplicative bundle gerbe arising from the \emph{basic gerbe} on $G$, and then pulls it back along the simplicial map $Q^{[\bullet+1]} \to NG_\bullet$.

The following construction is given at a reasonably high level, since it is intended for those who wish to see how a rigid bundle 2-gerbe gives rise to a bundle 2-gerbe, which we presume might be a minority of readers, namely those already familiar with bundle 2-gerbes.

\begin{construction}\label{constr:associated_b2g}
Consider a rigid bundle 2-gerbe on the manifold $X$, with surjective submersion $Y\to X$, semisimplicial submersion $Z_\bullet\to Y^{[\bullet+1]}$, bundle gerbe $(E,Z_1)$ on $Y^{[2]}$, and 2-gerbe multiplication $M$. 
The \emph{associated bundle 2-gerbe} is given by the following data

\begin{enumerate}

\item The bundle gerbe $\cE$ on $Y^{[2]}$ is just $(E,Z_1)$ as given;

\item The strong trivialisation of $(\delta_h(E),Z_2)$ gives an isomorphism of bundle gerbes
\[
	(d_0^*E\otimes d_2^*E,Z_2) \simeq (d_1^*E,Z_2)
\] 
which, if we denote the left bundle  gerbe by $\cE_{012}$, gives a span of morphisms of bundle gerbes 
\[
	\cE_{12}\otimes \cE_{01} \leftarrow \cE_{012} \to \cE_{02}
\] 
where $\cE_{12} = d_0^*\cE$, and so on, a notation which is consistent with the notation in Definition~\ref{def:bundle2gerbe}. 
The left-pointing morphism $\cE_{012}\to \cE_{12}\otimes \cE_{01}$ has adjoint inverse stable isomorphism $\cE_{12}\otimes \cE_{01} \rightsquigarrow \cE_{012} $ as in Lemma~\ref{lemma:bundlegerbemorphisms}. 
The bundle 2-gerbe multiplication $\overline{m}\colon \cE_{12}\otimes \cE_{01} \rightsquigarrow\cE_{02}$ is the composition of this adjoint inverse with the stable isomorphism associated to $\cE_{012} \to \cE_{02}$.

\item We need to define the associator, and here is where we will use mates. 
First, we consider the induced maps filling the diagram\footnote{You should think of this diagram as being indexed by the face lattice of a square.} of (domains of) submersions to $Y^{[4]}$:
\[
	\xymatrix{
		d_0^*d_0^*Z_1 \times_{Y^{[4]}} d_3^*d_0^*Z_1\times_{Y^{[4]}} d_3^*d_2^*Z_1 & \ar[l]d_0^*Z_2\times_{Y^{[4]}} d_3^*d_2^*Z_1 \ar[r]& d_0^*d_1^* Z_1\times_{Y^{[4]}}d_3^*d_2^*Z_1 \\
		\ar[u] d_0^* d_0^* Z_1 \times_{Y^{[4]}} d_3^*Z_2 \ar[d] & \ar[l] \ar[u] Z_3 \ar[r] \ar[d] & \ar[u] d_2^*Z_2 \ar[d] \\
		d_0^* d_0^* Z_1 \times_{Y^{[4]}} d_3^*d_1^*Z_1& \ar[l] d_1^*Z_2 \ar[r] & d_1^*d_1^* Z_1
		}
\]
where we have used some simplicial identities to make sure the domains and codomains match, and some squares commute.
The associativity condition $\delta_h(M)=1$ then allows us to define four commuting squares of morphisms of bundle gerbes on $Y^{[4]}$:
\begin{equation}\label{eq:pre_associator}
	\begin{gathered}
		\xymatrix{
				\cE_{23}\otimes \cE_{12}\otimes \cE_{01} & \ar@{~>}[l] \cE_{123}\otimes \cE_{01} \ar[r] & \cE_{13}\otimes \cE_{01}\\
				\ar@{~>}[u] \cE_{23}\otimes \cE_{012} \ar[d] & \ar@{~>}[l] \ar@{~>}[u]  \cE_{0123} \ar[r] \ar[d]& \ar@{~>}[u] \cE_{013}\ar[d] \\
				\cE_{23}\otimes \cE_{02} & \ar@{~>}[l] \cE_{023} \ar[r] & \cE_{03}
		}
		\end{gathered}
\end{equation}
where the bundle gerbe $\cE_{0123}$ has $U(1)$-bundle induced by pulling back the $U(1)$-bundle of $\cE_{23}\otimes \cE_{12}\otimes \cE_{01} $ along the composite $Z_3 \to d_0^*d_0^*Z_1 \times_{Y^{[4]}} d_3^*d_0^*Z_1\times_{Y^{[4]}} d_3^*d_2^*Z_1$ from the previous diagram.
In fact we can define everything in this diagram without using associativity of $M'$, except the commutativity of the lower right square, which follows from condition (5)(b') in Proposition~\ref{prop:rb2g_mult_associativity_square} on $M'$ and the universal property of pullback of $U(1)$-bundles.
Now to get the associator, we consider the mates of the top right, bottom left, and top left squares using the adjoint inverses for the arrows denoted $\rightsquigarrow$, and then paste the resulting diagram of 2-arrows between stable isomorphisms.

\item Now we need to prove that this associator is coherent. The strategy is to build up a commuting diagram of bundle gerbes and morphisms on $Y^{[5]}$, of shape the face lattice of a cube, as displayed in Figure~\ref{fig:big_cube}. 
The arrows denoted $\rightsquigarrow$ are again considered as being equipped with their adjoint inverses as stable isomorphisms. 
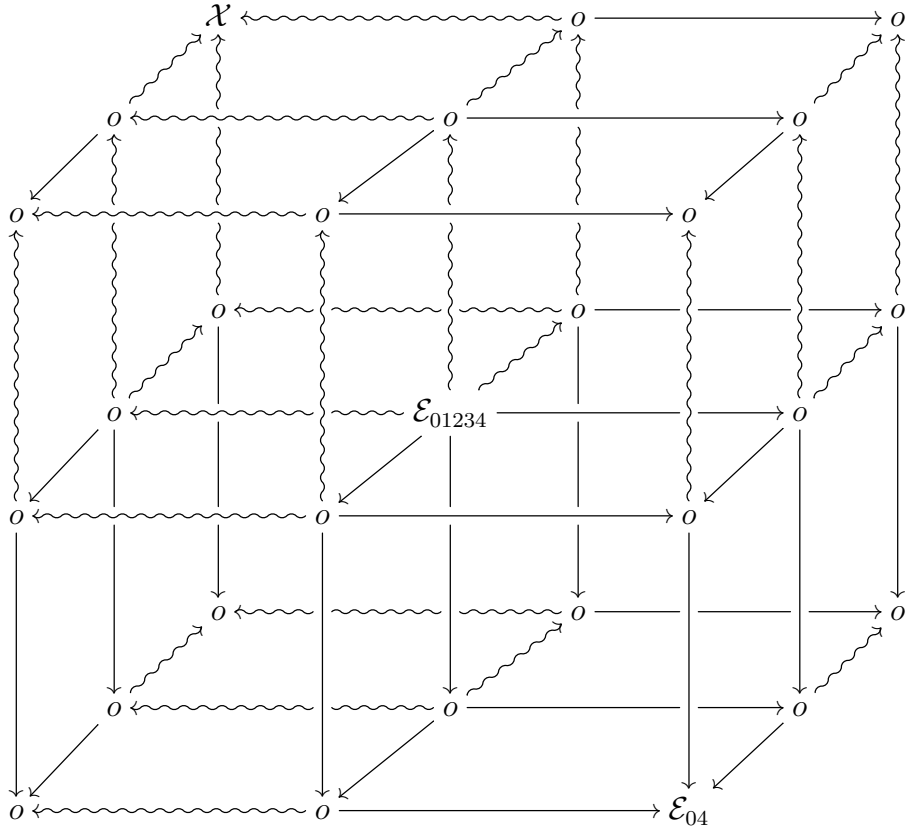
\begin{figure}
\[
	\xymatrix{   
		&& \mathcal{X} &&& \ar@{~>}[lll]_{} o\ar[rrr] &&& o\\
		& o \ar@{~>}[ur]^{} \ar[dl] &&& \ar@{~>}[lll]_{} o\ar@{~>}[ur]^{} \ar[dl] \ar[rrr] &&& 			o \ar@{~>}[ur]^{} \ar[dl]\\
		o &&& \ar@{~>}[lll]^{} o \ar[rrr] &&& o
		\\
		&& o \ar@{~>}[uuu]^{}|!{[ur];[ull]}{\hole}|!{[uurr];[uul]}{\hole} \ar[ddd]|!{[drr];[dl]}{\hole}|!{[ddr];[ddll]}{\hole}&&& 
			\ar@{~>}[lll]^{}|!{[ull];[ul]}{\hole}|!{[ddll];[ull]}{\hole} o \ar@{~>}^{}|!{[ull];[ul]}{\hole}|!{[uurr];[uul]}{\hole}[uuu]   \ar[rrr]|!{[ddr];[ur]}{\hole}|!{[drr];[uurr]}{\hole} \ar[ddd]|!{[dl];[drr]}{\hole}|!{[ddll];[ddr]}{\hole}	&&& 
			o \ar@{~>}[uuu]^{} \ar[ddd]\\
		& o \ar@{~>}[uuu]^{}|!{[uurr];[uul]}{\hole} \ar@{~>}[ur]^{} \ar[dl] \ar[ddd]|!{[drr];[dl]}{\hole}&&& 
			\ar@{~>}[lll]^{}|!{[dl];[uul]}{\hole} \cE_{01234}				\ar@{~>}[ur]^{} \ar@{~>}[uuu]^{}|!{[uul];[uurr]}{\hole} \ar[dl] \ar[rrr]|!{[drr];[uurr]}{\hole} \ar[ddd]|!{[dl];[drr]}{\hole} &&& 
			o \ar@{~>}[ur]^{}\ar@{~>}[uuu]^{} \ar[dl] \ar[ddd]\\
		o \ar@{~>}[uuu]^{} \ar[ddd] &&& \ar@{~>}[lll]^{}o \ar@{~>}[uuu]^{}\ar[ddd] \ar[rrr] &&			& o \ar@{~>}[uuu]^{} \ar[ddd]
		\\
		&& o &&& \ar@{~>}[lll]^{}|!{[uul];[dl]}{\hole}|!{[ull];[ddll]}{\hole} o\ar[rrr]|!{[ur];[ddr]}{\hole}|!{[uurr];[drr]}{\hole} &&& o\\
		& o \ar@{~>}[ur]^{} \ar[dl] &&& \ar@{~>}[lll]^{}|!{[uul];[dl]}{\hole} o\ar@{~>}[ur]^{} \ar[dl] \ar[rrr]|!{[uurr];[drr]}{\hole} &&& o \ar@{~>}[ur]_{} \ar[dl]\\
		o &&& \ar@{~>}[lll]_{} o \ar[rrr] &&& \cE_{04}\\
	}
\]
\caption{Diagram of bundle gerbes and morphisms over $Y^{[5]}$. The objects with placeholders marked $o$ are specified in Figure~\ref{fig:cut open associator cube}.}
\label{fig:big_cube}
\end{figure}

Then taking mates involving all such adjoint inverses, pasting, and using the result that pasting preserves mates, the resulting boundary cube will 2-commute as needed.
Two specific objects have been named, to help orient the reader, and the object labelled $\mathcal{X}$ is in fact $\cE_{34}\otimes\cE_{23}\otimes\cE_{12}\otimes \cE_{01}$. The cube should be considered as encoding associator coherence for the various ways of performing iterated bundle 2-gerbe multiplication to get a stable isomorphism $\mathcal{X} = \cE_{34}\otimes\cE_{23}\otimes\cE_{12}\otimes \cE_{01} \rightsquigarrow \cE_{04}$.
The bundle gerbe $\cE_{1234}$ uses the submersion $Z_4\to Y^{[5]}$, and its $U(1)$-bundle is defined to make the top-left-back cube commute.

In Figure~\ref{fig:cut open associator cube} we display the outer faces in two pieces, as in the definition of coherence for the associator in Definition~\ref{def:bundle2gerbe}, and describe the inner structure seperately. 

Since the cubical diagram in the 1-category of bundle gerbes and morphisms commutes on the nose, we have that the image of this diagram, up to insertion of associativity isomorphisms, commutes in the bigroupoid of bundle gerbes and stable isomorphisms. Coherence for bicategories ensures that the squares 2-commute when associator isomorphisms are added back in.
Then by the compatibility of pasting and mates, the eight 2-commuting small cubes still 2-commute when passing to the corresponding diagram with the mates.
Then the outer cube 2-commutes, which means that the associator satisfies the required coherence condition.

\begin{figure}
\noindent Back-right-bottom faces:
{\scriptsize
\[
	\centerline{
	\xymatrix@C=1em{
		&& \cE_{34}\otimes\cE_{23}\otimes\cE_{12}\otimes \cE_{01}  & \ar@{~>}[l] \cE_{234}\otimes\cE_{12}\otimes \cE_{01} \ar[r] & \cE_{24}\otimes\cE_{12}\otimes \cE_{01} \\
		& \ar[dl] \cE_{34}\otimes \cE_{23}\otimes \cE_{012} \ar@{~>}[ur] & \ar@{~>}[l] \ar[dl] \ar@{~>}[ur] \ar[r] \cE_{234}\otimes \cE_{012} & \ar[dl] \ar@{~>}[ur] \cE_{24}\otimes \cE_{012}  && \ar@{~>}[ul] \cE_{124}\otimes \cE_{01} \ar[dr]\\
		\cE_{34}\otimes\cE_{23}\otimes \cE_{02} & \ar@{~>}[l]  \cE_{234}\otimes\cE_{02} \ar[r]  & \cE_{24}\otimes\cE_{02} && \ar@{~>}[ul] \ar[dl] \cE_{0124} \ar@{~>}[ur] \ar[dr] && \cE_{14}\otimes \cE_{01}\\
		& \ar@{~>}[ul] \cE_{34}\otimes\cE_{023} \ar[dr] &\ar@{~>}[l] \ar@{~>}[ul] \cE_{0234} \ar[r] \ar[dr] & \ar@{~>}[ul] \cE_{024} \ar[dr] && \ar[dl] \cE_{014} \ar@{~>}[ur]\\
		&& \cE_{34}\otimes\cE_{03} & \ar@{~>}[l] \cE_{034} \ar[r] & \cE_{04} 
	}
	}
\]}

\noindent Front-left-top faces:
{\scriptsize
\[
	\centerline{
	\xymatrix@C=1em{
		&& \cE_{34}\otimes\cE_{23}\otimes\cE_{12}\otimes \cE_{01} & \ar@{~>}[l] \cE_{234}\otimes\cE_{12}\otimes \cE_{01} \ar[r] & \cE_{24}\otimes\cE_{12}\otimes \cE_{01} \\
		& \ar[dl] \cE_{34}\otimes \cE_{23}\otimes \cE_{012} \ar@{~>}[ur] &&  \ar@{~>}[ul] \cE_{34}\otimes \cE_{123}\otimes \cE_{01} \ar[dr] & \ar@{~>}[l] \ar@{~>}[ul] \cE_{1234}\otimes \cE_{01} \ar[r] \ar[dr] & \ar@{~>}[ul] \cE_{124}\otimes \cE_{01} \ar[dr]\\
		\cE_{34}\otimes\cE_{23}\otimes \cE_{02} && \ar@{~>}[ul] \ar[dl] \cE_{34}\otimes \cE_{0123} \ar@{~>}[ur] \ar[dr] && \cE_{34}\otimes \cE_{13}\otimes \cE_{01} & \ar@{~>}[l] \cE_{134}\otimes \cE_{01} \ar[r] & \cE_{14}\otimes \cE_{01}\\
		& \ar@{~>}[ul] \cE_{34}\otimes\cE_{023} \ar[dr] && \ar[dl] \cE_{34}\otimes \cE_{013} \ar@{~>}[ur] & \ar@{~>}[l] \ar[dl] \cE_{0134} \ar@{~>}[ur] \ar[r] & \ar[dl] \cE_{014} \ar@{~>}[ur]\\
		&& \cE_{34}\otimes\cE_{03} & \ar@{~>}[l] \cE_{034} \ar[r] & \cE_{04} 
	}
	}
\]}\medskip

\caption{Encoding of associator coherence from rigid 2-gerbe data, giving the outer faces of Figure \ref{fig:big_cube}}\label{fig:cut open associator cube}
\end{figure}

\end{enumerate}

\noindent
This completes the construction of the associated bundle 2-gerbe.

\end{construction}

Now on top of this, we can also consider what happens to the connective structure on a rigid bundle 2-gerbe, once we form the associated bundle 2-gerbe. Here, in fact, things are relatively much easier, if we only work with \cite[Definition 8.1]{Ste}.
This is because the connection data (connection, curving, 2-curving) on a bundle 2-gerbe is actually insensitive to the specifics of the bundle 2-gerbe multiplication and its coherence; all that is needed is that the 3-form curvature $H$ of the bundle gerbe on $Y^{[2]}$ satisfies\footnote{The existence of such 3-form curvature such that $\delta(H)$ is \emph{exact} is immediate after we know that \emph{any} bundle 2-gerbe multiplication exists, coherent or not.
That one can improve the curving on the bundle gerbe to get $\delta(H)=0$ on the nose is the hard part of \cite[Chapter 10]{StevPhD}.} $\delta(H)=0$.
The proof of the following proposition is immediate.

\begin{proposition}\label{prop:associated_connective_structure}
Consider a rigid bundle 2-gerbe on the manifold $X$, with surjective submersion $Y\to X$, semisimplicial submersion $Z_\bullet\to Y^{[\bullet+1]}$, and bundle gerbe $(E,Z_1)$ on $Y^{[2]}$. 
Now assume also that the rigid bundle 2-gerbe has a connective structure $(\nabla_E,\beta,\lambda)$. 
The associated bundle 2-gerbe then has a connection $\nabla_E$, curving $\beta$, and 2-curving $\lambda$ in the sense of \cite[Definition 8.1]{Ste}, all of which are exactly the same pieces of data as for the rigid bundle 2-gerbe.
\end{proposition}

Note that here we are not saying what happens to the data $\alpha\in \Omega^1(Z_2)$ in a \emph{rigid} connective structure.
One could in fact include this information, if presenting a bundle 2-gerbe with connective structure as descent data valued in the bigroupoid of bundle gerbes equipped with connections and curvings (cf \cite{Nikolaus-Schweigert_11}), as is done in \cite[Definition 4.1]{Wal}.
In fact, it seems\footnote{Konrad Waldorf, personal communication} that the definition of connection and curving on a bundle 2-gerbe in \cite{Ste} is not rigid enough, and that the proof that bundle 2-gerbes with connection, curving and 2-curving are classified by an appropriate Deligne cohomology group may fail. 
So we shall here also prove that the connective structure on the associated bundle 2-gerbe satisfies Waldorf's stronger definition, which we now give. 

\begin{definition}[\cite{Wal}]\label{def:Waldorf_connective_structure}
Assume given a bundle 2-gerbe, with connection $\nabla$, curving $\beta$ and 2-curving $\lambda$ as in \cite[Definition 8.1]{Ste}. Then we require the following conditions on $\nabla$ and $\beta$
\begin{enumerate}
\item There is a connection $\nabla_{\overline{m}}$ on the $U(1)$-bundle encoding the stable isomorphism $\overline{m}$ on $Y^{[3]}$
\item The associator isomorphism $\alpha$ in Definition \ref{def:bundle2gerbe}, which is an isomorphism between certain $U(1)$-bundles, preserves the induced connections.
\end{enumerate}
\end{definition}

\begin{proposition}\label{prop:associated_Waldorf_connective_structure}
If we assume the connection $\nabla_E$ and curving $\beta$ in \ref{prop:associated_connective_structure} are part of a \emph{rigid} connective structure, with 1-form $\alpha$, then they satisfy the additional conditions in \ref{def:Waldorf_connective_structure}.
\end{proposition}

\begin{proof}
Lemma~\ref{lemma:bundlegerbemorphisms_w_conn_curv} guarantees that the stable isomorphism $\overline{m}$ has a compatible connection, built using $\alpha$.
The identity $\delta_h(\alpha)=0$ ensures that the diagram (\ref{eq:pre_associator}) becomes a diagram of bundle gerbes with connection and curving, and morphisms compatible with these.
Further, since the associator is built from mates, and these mates are 2-arrows with data connection-preserving bundle isomorphisms, as per the discussion after Lemma~\ref{lemma:bundlegerbemorphisms_w_conn_curv}, then the associator preserves connections.
The coherence equation holds for free, since it is an equality of 2-arrows, which neither places or required extra constraints on the differential form data.
\end{proof}

\subsection{The proof that the Chern--Simons bundle 2-gerbe is rigidified}\label{subsec:CS2-gerbe_rigidified}

The following Proposition captures a feature that holds for the rigid Chern--Simons 2-gerbe, in a more abstract way.

\begin{proposition}\label{prop:general_strictification_setup}
Suppose we are given a rigid bundle 2-gerbe on $X$ with submersion $Y\to X$ and semisimplicial submersion $Z_\bullet\to Y^{[\bullet+1]}$.
\begin{enumerate}
 \item If the map $Z_2\to d_0^* Z_1\times_{Y^{[3]}} d_2^*Z_1$ over $Y^{[3]}$ induced by the universal properties of the pullbacks is an isomorphism, then in the span of morphisms of bundle gerbes giving rise to the associated bundle 2-gerbe multiplication, the left-pointing morphism is an \emph{isomorphism} of bundle gerbes.

 \item Suppose additionally that either the indued map $Z_3 \to d_1^*d_0^*Z_1 \times_{Y^{[4]}}d_3^*Z_2$, or the induced map $Z_3 \to d_0^*Z_2 \times_{Y^{[4]}} d_2^*d_2^* Z_1$, is an isomorphisms. Then the associator 2-arrow is an identity.
 \end{enumerate} 

\end{proposition}

\begin{proof}
\begin{enumerate}
\item This follows from Lemma~\ref{lemma:bundlegerbemorphisms}, so that the inverse of the bundle gerbe isomorphism involving $Z_2\to d_0^* Z_1\times_{Y^{[3]}} d_2^*Z_1$ gives the adjoint inverse as required.

\item Here we apply Lemma~\ref{lemma:mate_using_isos} twice, since the relevant square of bundle gerbe morphisms involves the commutative square
\[
	\xymatrix@C=1em{
		 d_1^*d_0^*Z_1\times_{Y^{[4]}} d_3^*(d_0^*Z_1\times_{Y^{[3]}}d_2^*Z_1)
		 \ar@{=}[r] & d_0^*(d_0^*Z_1 \times_{Y^{[3]}} d_2^*Z_1)\times_{Y^{[4]}} d_2^*d_2^*Z_1
		  & \ar[l] d_0^*Z_2\times_{Y^{[4]}} d_2^*d_2^*Z_1 \\
		\ar[u] d_1^* d_0^*Z_1\times_{Y^{[4]}} d_3^*Z_2 && \ar[ll] \ar[u] Z_3
	}
\]
The hypothesis from part (1) implies that the left and top arrows are invertible (hence giving rise to isomorphisms of bundle gerbes), and the extra hypothesis from part (2) implies that either the bottom or the right arrow is invertible, implying that the whole square consists of isomorphisms, and hence the top left square in \eqref{eq:pre_associator} consists of isomorphisms of bundle gerbes. Thus the mate of this square (using both adjoint inverse pairs) is the on-the-nose commutative square corresponding to the inverses of the bundle gerbe isomorphisms.
A similar argument works for the bottom left and top right squares of \eqref{eq:pre_associator}, except in this case we only need to check that the morphisms of bundle gerbes induced by 
\[
	d_1^*Z_2 \to d_1^*(d_0^*Z_1 \times_{Y^{[3]}}d_2^* Z_1) = d_0^* d_0^* Z_1 \times_{Y^{[4]}} d_3^*d_1^*Z_1  
\] 
and
\[
	d_2^* Z_2 \to d_2^*(d_0^*Z_1 \times_{Y^{[3]}}d_2^* Z_1) = d_0^*d_1^*Z_1\times_{Y^{[4]}}d_3^*d_2^*Z_1 
\]
are isomorphisms, but this holds by assumption from (1).
Thus the associator, which is the pasting of the four squares, all of which commute on the nose, is an identity 2-arrow. \qedhere
\end{enumerate}
\end{proof}

\begin{remark}

 \begin{enumerate}
 \item In Proposition~\ref{prop:general_strictification_setup} (1), the bundle 2-gerbe multiplication is a stable isomorphism arising from a \emph{morphism} of bundle gerbes, namely the composite of the inverse isomorphism of bundle gerbes with the forward-pointing morphism of bundle gerbes.
 
 Thus in this instance the associated bundle 2-gerbe is a bundle 2-gerbe (in the sense of Definition~\ref{def:bundle2gerbe}) arising from the slightly stricter notion of bundle 2-gerbe given as \cite[Definition~7.2]{StevPhD}, via \cite[Note 7.1.1]{StevPhD} (recall that the objects given by Definition~\ref{def:bundle2gerbe} are called `stable bundle 2-gerbes' in \cite{StevPhD}).

 \item In the situation of Proposition~\ref{prop:general_strictification_setup} (2) the associated bundle 2-gerbe is the bundle 2-gerbe arising from a \emph{strict} bundle 2-gerbe.
 \end{enumerate}

 Thus our notion of rigid bundle 2-gerbe is a different way of encoding a 2-gerbe-like object in a strict fashion, only sometimes corresponding to the more obvious strict notion coming from general higher categorical considerations (e.g.\ honest morphisms, trivial associators etc).
 \end{remark}

Here we give the postponed proof of Theorem~\ref{prop:rigidification}, that our rigid Chern--Simons 2-gerbe really is a rigidification of the Chern--Simons bundle 2-gerbe appearing in the literature. It is now just a corollary of the previous Proposition.

\begin{proof}{(Of Theorem \ref{prop:rigidification})}
Recall that the definition in the literature of the Chern--Simons bundle 2-gerbe associated to the $G$-bundle $Q\to X$ \cite[\S 9.3]{StevPhD}, \cite[Definition 6.5]{CJMSW} is as the pullback along the simplicial map $Q^{[\bullet+1]} \to NG_\bullet$ of the multiplicative bundle gerbe on $G$, namely the basic gerbe. 
The multiplicative structure on the latter, defined using the crossed module structure on $\widehat{\Omega G} \to PG$, is a strict 2-group, and hence the Chern--Simons bundle 2-gerbe is a strict bundle 2-gerbe. 
The only caveat is that in \emph{loc.\ cit.} a different, but isomorphic, model for the basic gerbe is given. 
One can in fact easily check that the conditions of Proposition~\ref{prop:general_strictification_setup} hold, and that the resulting bundle 2-gerbe multplication \emph{morphism} (not stable isomorphism) is exactly analogous to that from the literature, modulo the different model of the basic gerbe. 
Thus the associated bundle 2-gerbe of our rigid Chern--Simons 2-gerbe is isomorphic to the existing bundle 2-gerbe model, and hence we have a rigidification.
\end{proof}

That the data of the connective structure works correctly for the Chern--Simons 2-gerbe has been mentioned in Corollary \ref{cor:CSrb2g_rigidification_with_geom} discussed above. 
Thus the only other piece of data to look at is the 1-form $\alpha$ in the \emph{rigid} connective structure, which is not even considered in \cite{CJMSW}.
But then Proposition \ref{prop:associated_Waldorf_connective_structure} takes care of this last piece of data, which has never before been given explicitly.

\subsection{Trivialisations}

Now, of course the other construction we need to wrap up is that of a bundle 2-gerbe trivialisation from a rigid trivialisation, and also the version with connections. 
We note that the definition we are using here is adapted from \cite[Definition~12.2]{StevPhD} (and \cite[Definition~4.6]{CJMSW}), which is different to \cite[Definition~12.1]{StevPhD}, \cite[Definition~11.1]{Ste}, and \cite[Definition~3.5]{Wal} in that the latter essentially use the dual bundle gerbe $\cF^*$ in place of $\cF$ in the below definition (which leads to what can be considered a \emph{left} action of $\cE$ on $\cF$, rather than a right action, as in \cite[Definition~12.1]{StevPhD}). 
The choice we are making as to what gets dualised is made so as to be compatible with the notion of coboundary, and to be consistent with how duals are taken when pulling back by face maps in a semisimplicial space.
However, the definition we adopt is \emph{in the style of} \cite[Definition~3.5]{Wal}, and does not lead to any substantial difference in any results here (or elsewhere).

\begin{definition}\label{def:bundle2gerbetriv}
A \emph{trivialisation} of a bundle 2-gerbe on $X$, with surjective submersion $Y\to X$ and data $(\cE,\overline{m},\alpha)$ is given by the following data:
\begin{enumerate}
\item A bundle gerbe $\cF$ on $Y$;
\item A stable isomorphism $\eta\colon \cE\otimes d_1^*\cF\to d_0^*\cF$ of bundle gerbes on $Y^{[2]}$;
\item A 2-arrow $\theta$ (the `actionator') between stable isomorphisms over $Y^{[3]}$ as in the following diagram:
\[
	\begin{gathered}
	\xymatrix{
		\cE_{12}\otimes \cE_{01}\otimes \cF_0 \ar[dd]_{1\otimes d_2^*\eta} \ar[r]^-{\overline{m}\otimes 1} & 
		\cE_{02}\otimes \cF_0 \ar[dd]_{\ }="s"^{d_1^*\eta}
		\\
		&\ \\
		\cE_{12}\otimes \cF_1 \ar[r]^{\ }="t"_-{d_0^*\eta} & \cF_2
		\ar@{=>}"s";"t"_\theta
		}
		\end{gathered}
		:=
	\begin{gathered}
	\xymatrix@C=0.3em{
			d_0^*\cE\otimes d_2^*\cE \otimes d_2^*d_1^*\cF  
			\ar[d]_{1\otimes d_2^*\eta} \ar[rr]^-{\overline{m}\otimes 1} &&
			d_1^* \cE \otimes d_2^*d_1^*\cF
			 \ar@{=}[d]
			 \\
			 d_0^*\cE \otimes d_2^*d_0^*\cF 
			  \ar@{=}[d] &&  d_1^* \cE \otimes d_1^*d_1^*\cF 
			 \ar[d]^{d_1^*\eta} \ar@{=>}[]!DL;[ld]!U_{\theta}
			\\
d_0^*\cE  \otimes d_0^*d_1^*\cF
			 \ar[r]_-{d_0^*\eta} 
			 & d_0^*d_0^*\cF
			 \ar@{=}[r] 
			 & 
			 d_1^* d_0^*\cF 
		}
		\end{gathered}
\]
where here $\cE_{ab}$ denotes the pullback along $Y^{[3]}\to Y^{[2]}$ induced by $(a,b)\into (0,1,2)$ in the simplex category $\Delta$, and $\cF_c$ is the pullback along the map $Y^{[3]}\to Y$ induced by $(c)\into (0,1,2)$.
\item Such that the following pair of 2-arrows over $Y^{[4]}$ are equal:
\begin{align*}
	\xymatrix@-1em{
		&  \cE_{23}\otimes \cE_{12}\otimes \cE_{01} \otimes\cF_0 \ar[r] \ar[ddl]^(0.6){\ }="t1" & \cE_{13}\otimes \cE_{01}\otimes \cF_0   \ar[ddl]_{\ }="s1" \ar[ddr] \\
		\\
		\cE_{23}\otimes \cE_{12}  \otimes \cF_1 
		 \ar[r] \ar[ddr]^(0.6){\ }="t3" & \cE_{13}\otimes \cF_1  \ar[ddr]^(0.55){\ }="t2"_(0.4){\ }="s3" && \cE_{03}\otimes \cF_0
		  \ar[ddl]_(0.35){\ }="s2"
		\\
		\\
		& \cE_{23} \otimes \cF_2  \ar[r] & \cF_3
		\ar@{=>}"s1";"t1"_{!\simeq}
		\ar@{=>}"s2";"t2"_{d_2^*\theta}
		\ar@{=>}"s3";"t3"_{d_0^*\theta}
	}
	\\
	\\
	\xymatrix@-1em{
		& \cE_{23}\otimes \cE_{12}\otimes \cE_{01} \otimes\cF_0  \ar[r] \ar[ddl] \ar[ddr]^(0.6){\ }="t2" &  \cE_{13}\otimes \cE_{01}\otimes \cF_0   \ar[ddr]_(0.4){\ }="s2" \\
		\\
		 \cE_{23}\otimes \cE_{12}  \otimes \cF_1 \ar[ddr]^(0.4){\ }="t1" && 
		\cE_{23}\otimes \cE_{02} \otimes \cF_0  \ar[r] \ar[ddl]_(0.35){\ }="s1"^{\ }="t3" &  \cE_{03}\otimes \cF_0 \ar[ddl]_{\ }="s3"\\
		\\
		& \cE_{23} \otimes \cF_2 \ar[r] & \cF_3
		\ar@{=>}"s1";"t1"_{1\otimes d_3^*\theta}
		\ar@{=>}"s2";"t2"_{\tilde \alpha\otimes 1}
		\ar@{=>}"s3";"t3"_{d_1^*\theta}
	}
\end{align*}
where again we have used the notation $\cE_{ab}$ from Definition~\ref{def:bundle2gerbe}, and now also use the notation $\cF_a$ to denote the pullback along the map $Y^{[4]}\to Y$ induced by $(a)\into (0,1,2,3)$ in the simplex category $\Delta$.
\end{enumerate}
\end{definition}

We will use the equivalent definition of a rigid trivialisation from Proposition~\ref{prop:module_version_of_rigid_triv} in the following construction, since it matches the definition of bundle 2-gerbe trivialisation more closely.

\begin{construction}\label{constr:associated_trivialisation}
Assume we are given a rigid bundle 2-gerbe with data $Y\to X$, $Z_\bullet\to Y^{[\bullet+1]}$, $(E,Z_1)$ and $M$, and a rigid trivialisation with data $(F,Z_0)$ and $t'$ (as in Proposition~\ref{prop:module_version_of_rigid_triv}), let us construct a trivialisation in the sense of Definition~\ref{def:bundle2gerbetriv} of the associated bundle 2-gerbe. 
\begin{enumerate}
\item We already have the bundle gerbe on $Y$ as needed. 
\item We construct the stable isomorphism $\eta$ via the span of morphisms of bundle gerbes involving the span of covers
\[
	 Z_1 \times_{Y^{[2]}} d_1^*Z_0 \xleftarrow{(\id,d_1)} Z_1 \xrightarrow{d_0} d_0^*Z_0
\]
of $Y^{[2]}$, and the isomorphism $(t',\id)\colon (E\otimes d_1^*F,Z_1)\to (d_0^*F,Z_1)$ of bundle gerbes.
As in Construction~\ref{constr:associated_b2g}, we get a span
\[
	\cE\otimes d_1^*\cF \leftsquigarrow (\cE\cF_1) \to d_0^*\cF
\]
where $(\cE\cF_1) := ( E\otimes d_1^*F,Z_1)$. 
Passing to the associated stable isomorphisms, then taking the adjoint inverse to the arrow denoted $\rightsquigarrow$ and composing, this will then be the stable isomorphism $\eta$.

\item We then need to build the 2-arrow $\theta$. Consider the diagram of (domains of) submersions to $Y^{[3]}$
\[
	\xymatrix{
		 d_0^*Z_1\times_{Y^{[3]}} d_2^* Z_1 \times_{Y^{[3]}} d_2^*d_1^* Z_0 
		 & 
		\ar[l] Z_2 \times_{Y^{[3]}} d_2^*d_1^* Z_0 \ar[r] 
		 & 
		 d_1^*Z_1 \times_{Y^{[3]}} d_2^*d_1^*Z_0
		 \\
		 \ar[u] d_0^* Z_1 \times_{Y^{[3]}} d_2^* Z_1 \ar[d] 
		& 
		\ar[l] \ar[u] Z_2 \ar[r] \ar[d] 
		&
		 \ar[u] d_1^*Z_1  \ar[d]
		\\
		d_0^* Z_1 \times_{Y^{[3]}}  d_2^* d_0^*Z_0 
		&
		\ar[l]d_0^* Z_1 \ar[r] 
		& 
		d_0^*d_0^*Z_0
	}	
\]
where we have used simplicial identities implicitly to define maps using the universal property of the pullback. This gives rise to a diagram of bundle gerbe morphisms
\begin{equation}\label{eq:pre-actionator}
	\xymatrix{
		\cE_{12}\otimes \cE_{01}\otimes \cF_0 
		& 
		\ar@{~>}[l]\cE_{012}\otimes \cF_0 \ar[r]
		&  
		\cE_{02} \otimes \cF_0 
		\\
		\ar@{~>}[u] \cE_{12} \otimes (\cE_{01}\cF_0) \ar[d] 
		& 
		\ar@{~>}[l] \ar@{~>}[u] (\cE_{012}\cF_0) \ar[r] \ar[d] 
		& 
		\ar@{~>}[u] (\cE_{02}\cF_0) \ar[d]
		\\
		\cE_{12}\otimes \cF_1 
		& 
		\ar@{~>}[l](\cE_{12}\cF_1)  \ar[r] 
		& 
		\cF_2
	}
\end{equation}
where the $U(1)$-bundle part of the bundle gerbe $(\cE_{012}\cF_0)$ is given by pullback along $Z_2^{[2]}\to (d_0^*Z_1\times_{Y^{[3]}} d_2^* Z_1 \times_{Y^{[3]}} d_2^*d_1^* Z_0)^{[2]}$ of the $U(1)$-bundle from $\cE_{12}\otimes \cE_{01}\otimes \cF_0$.
The commutativity of the bottom right square follows from the condition (3') on $t'$ in Proposition~\ref{prop:module_version_of_rigid_triv}.
The coherence morphism $\theta$ is defined to be the pasting of the mates (in the bicategory of bundle gerbes and stable isomorphisms) of the top left, top right and bottom left squares, taken with respect to the adjoint equivalences denoted $\rightsquigarrow$.

\item We now need to establish the coherence of $\theta$. We have a cube of the same form as before, shown in Figure~\ref{fig:big_cube_triv}, which is a diagram of bundle gerbes and morphisms on $Y^{[4]}$, and where the arrows of the form $\rightsquigarrow$  are intended to be considered as being equipped with their adjoint inverses as stable isomorphisms.
The object denoted here $\mathcal{X}'$ is the bundle gerbe $\cE_{23}\otimes\cE_{12}\otimes \cE_{01}\otimes \cF_0$.

\begin{figure}
\[
	\xymatrix{   
		&& \mathcal{X}' &&& \ar@{~>}[lll]_{} o\ar[rrr] &&& o\\
		& o \ar@{~>}[ur]^{} \ar[dl] &&& \ar@{~>}[lll]_{} o\ar@{~>}[ur]^{} \ar[dl] \ar[rrr] &&& 			o \ar@{~>}[ur]^{} \ar[dl]\\
		o &&& \ar@{~>}[lll]^{} o \ar[rrr] &&& o
		\\
		&& o \ar@{~>}[uuu]^{}|!{[ur];[ull]}{\hole}|!{[uurr];[uul]}{\hole} \ar[ddd]|!{[drr];[dl]}{\hole}|!{[ddr];[ddll]}{\hole}&&& 
			\ar@{~>}[lll]^{}|!{[ull];[ul]}{\hole}|!{[ddll];[ull]}{\hole} o \ar@{~>}^{}|!{[ull];[ul]}{\hole}|!{[uurr];[uul]}{\hole}[uuu]   \ar[rrr]|!{[ddr];[ur]}{\hole}|!{[drr];[uurr]}{\hole} \ar[ddd]|!{[dl];[drr]}{\hole}|!{[ddll];[ddr]}{\hole}	&&& 
			o \ar@{~>}[uuu]^{} \ar[ddd]\\
		& o \ar@{~>}[uuu]^{}|!{[uurr];[uul]}{\hole} \ar@{~>}[ur]^{} \ar[dl] \ar[ddd]|!{[drr];[dl]}{\hole}&&& 
			\ar@{~>}[lll]^{}|!{[dl];[uul]}{\hole} (\cE_{0123}\cF_0)				\ar@{~>}[ur]^{} \ar@{~>}[uuu]^{}|!{[uul];[uurr]}{\hole} \ar[dl] \ar[rrr]|!{[drr];[uurr]}{\hole} \ar[ddd]|!{[dl];[drr]}{\hole} &&& 
			o \ar@{~>}[ur]^{}\ar@{~>}[uuu]^{} \ar[dl] \ar[ddd]\\
		o \ar@{~>}[uuu]^{} \ar[ddd] &&& \ar@{~>}[lll]^{}o \ar@{~>}[uuu]^{}\ar[ddd] \ar[rrr] &&			& o \ar@{~>}[uuu]^{} \ar[ddd]
		\\
		&& o &&& \ar@{~>}[lll]^{}|!{[uul];[dl]}{\hole}|!{[ull];[ddll]}{\hole} o\ar[rrr]|!{[ur];[ddr]}{\hole}|!{[uurr];[drr]}{\hole} &&& o\\
		& o \ar@{~>}[ur]^{} \ar[dl] &&& \ar@{~>}[lll]^{}|!{[uul];[dl]}{\hole} o\ar@{~>}[ur]^{} \ar[dl] \ar[rrr]|!{[uurr];[drr]}{\hole} &&& o \ar@{~>}[ur]_{} \ar[dl]\\
		o &&& \ar@{~>}[lll]_{} o \ar[rrr] &&& \cF_3\\
	}
\]
\caption{Diagram of bundle gerbes and morphisms over $Y^{[4]}$ to derive the coherence for $\theta$. The objects with placeholders marked $o$ are specified in Figure~\ref{fig:cut open actionator cube}.}
\label{fig:big_cube_triv}
\end{figure}
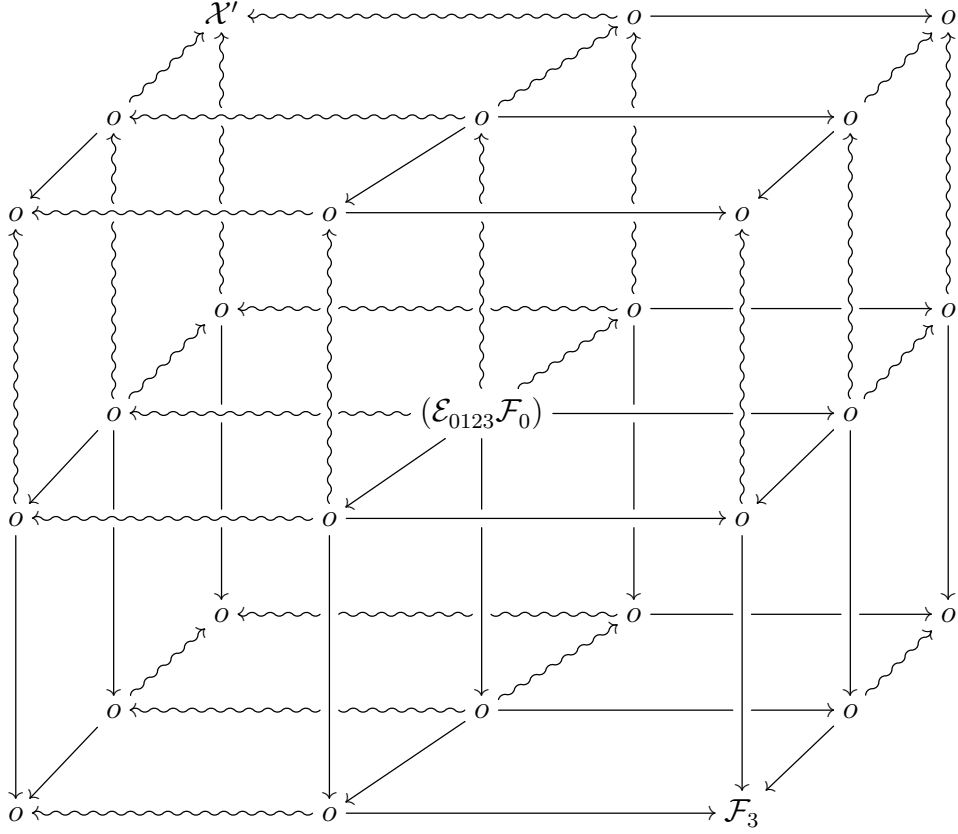

Cutting open this cube to reveal the faces and the specific bundle gerbes at the nodes, we get Figure~\ref{fig:cut open actionator cube}. Note that the bundle gerbe $(\cE_{0123}\cF_0)$ in the cubical diagram has submersion $Z_3\to Y^{[4]}$, and has $U(1)$-bundle defined to be the pullback along $Z_3\to Z_3\times_{Y^{[4]}}d_1^*d_1^*d_1^*Z_0$ of the $U(1)$-bundle of the bundle gerbe $\cE_{0123}\otimes d_1^*d_1^*d_1^*\cF$.

\begin{figure}
\noindent Back-right-bottom faces:
{\scriptsize
\[
	\centerline{
	\xymatrix@C=1em{
		&& \cE_{23}\otimes\cE_{12}\otimes\cE_{01}\otimes \cF_0  & \ar@{~>}[l] 
		\cE_{123}\otimes \cE_{01}\otimes \cF_0 \ar[r] & \cE_{13}\otimes\cE_{01}\otimes \cF_0 \\
		& \ar[dl] \cE_{23}\otimes\cE_{12}\otimes(\cE_{01}\cF_0) \ar@{~>}[ur] & \ar@{~>}[l] \ar[dl] \ar@{~>}[ur] \ar[r] \cE_{123}\otimes (\cE_{01}\cF_0) & \ar[dl] \ar@{~>}[ur] \cE_{13}\otimes(\cE_{01}\cF_0)  && \ar@{~>}[ul] \cE_{013}\otimes\cF_0  \ar[dr]
		\\
		\cE_{23}\otimes\cE_{12}\otimes \cF_1 & \ar@{~>}[l]   \cE_{123}\otimes \cF_1 \ar[r]  & \cE_{13}\otimes\cF_1 && \ar@{~>}[ul] \ar[dl] (\cE_{013}\cF_0)  \ar@{~>}[ur] \ar[dr] && \cE_{03}\otimes\cF_0 
		\\
		& \ar@{~>}[ul] \cE_{23}\otimes (\cE_{12}\cF_1) \ar[dr] &\ar@{~>}[l] \ar@{~>}[ul] (\cE_{123}\cF_1)\ar[r] \ar[dr] & \ar@{~>}[ul] (\cE_{13}\cF_1) \ar[dr] && \ar[dl] (\cE_{03}\cF_0) \ar@{~>}[ur]\\
		&& \cE_{23}\otimes \cF_2 & \ar@{~>}[l] (\cE_{23}\cF_2) \ar[r] & \cF_3
	}
	}
\]}
\medskip

\noindent Front-left-top faces:
{\scriptsize
\[
	\centerline{
	\xymatrix@C=1em{
		&& \cE_{23}\otimes\cE_{12}\otimes\cE_{01}\otimes \cF_0 & \ar@{~>}[l] \cE_{123}\otimes \cE_{01}\otimes \cF_0  \ar[r] &  \cE_{13}\otimes\cE_{01}\otimes \cF_0\\
		& \ar[dl]  \cE_{23}\otimes\cE_{12}\otimes(\cE_{01}\cF_0) \ar@{~>}[ur] &&  \ar@{~>}[ul] \cE_{23}\otimes \cE_{012}\otimes \cF_0 \ar[dr] & \ar@{~>}[l] \ar@{~>}[ul] \cE_{0123}\otimes \cF_0 \ar[r] \ar[dr] & \ar@{~>}[ul]  \cE_{013}\otimes\cF_0  \ar[dr]
		\\
		\cE_{23}\otimes\cE_{12}\otimes \cF_1 && \ar@{~>}[ul] \ar[dl] \cE_{1234}\otimes \cE_{01} \ar@{~>}[ur] \ar[dr] && \cE_{23}\otimes \cE_{02}\otimes \cF_0 & \ar@{~>}[l]  \cE_{023}\otimes \cF_0\ar[r] & \cE_{03}\otimes\cF_0
		\\
		& \ar@{~>}[ul] \cE_{23}\otimes (\cE_{12}\cF_1) \ar[dr] && \ar[dl] \cE_{23}\otimes (\cE_{02}\cF_0) \ar@{~>}[ur] & \ar@{~>}[l] \ar[dl] (\cE_{023}\cF_0) \ar@{~>}[ur] \ar[r] & \ar[dl]  (\cE_{03}\cF_0)\ar@{~>}[ur]
		\\
		&& \cE_{23}\otimes \cF_2 & \ar@{~>}[l](\cE_{23}\cF_2)  \ar[r] & \cF_3
	}}
\]}
\caption{Encoding of action coherence from rigid trivialisation data, giving the outer faces of Figure \ref{fig:big_cube_triv}.}\label{fig:cut open actionator cube}
\end{figure}

The same pasting argument with mates for the arrows indicated $\rightsquigarrow$ as in Construction~\ref{constr:associated_b2g} establishes the required coherence axiom for $\theta$.
\end{enumerate}

This completes the construction of the associated bundle 2-gerbe trivialisation.
\end{construction}

Now consider the case where the rigid 2-gerbe is equipped with a rigid connective structure, and the rigid trivialisation has a trivialisation connection. 
Then, as before in Proposition~\ref{prop:associated_connective_structure}, we get the data of a connection, curving and 2-curving on the associated bundle 2-gerbe, and a trivialisation in the sense of Definition~\ref{def:bundle2gerbetriv}.
We need to prove that the trivialisation connection gives rise to a compatible connection in the sense of \cite[Definition~4.3]{Wal}.
We outline Waldorf's definition here in more elementary terms.

\begin{definition}[Waldorf]
\label{def:weak_trivialisation_connection}
Consider a bundle 2-gerbe with submersion $Y\to X$, bundle gerbe $\cE=(E,Z_1)$ on $Y^{[2]}$, multiplication stable morphism $\overline{m}$ and associator $\tilde \alpha$, and also connection, curving and 2-curving.
The data of a compatible connection on a trivialisation with data $\cF=(F,Z_0)$, $\eta$ and $\theta$, is a connection and curving on $\cF$, together with a connection $\nabla_L$ on the $U(1)$-bundle $L\to (Z_1\times_{Y^{[2]}}d_1^*Z_0) \times_{Y^{[2]}}d_0^*Z_0$ underlying the stable isomorphism $\eta$ such that the curvature $F_{\nabla_L}$ is the difference between the (pullback of the) curving $\beta$ of the bundle 2-gerbe and the (pullback of) the curving of $\delta(\cF)$. 
Further, the 2-arrow $\theta$ in Definition~\ref{def:bundle2gerbetriv}, the data of which is an isomorphism of trivialising $U(1)$-bundles (satisfying a compatibility condition), needs to be an isomorphism of bundles \emph{with connection}. 
\end{definition}

The following construction is in some sense the keystone of this paper, linking our notion of rigid trivialisation to the construction of geometric string structures \`a la Waldorf \cite{Wal}.
Here we only give it in outline, as it repeats a lot of steps shown previously.

\begin{construction}
Start with a rigid bundle 2-gerbe with rigid trivialisation, rigid connective structure, and trivialisation connection. 
Build the associated bundle 2-gerbe from Construction~\ref{constr:associated_b2g}, the associated trivialisation from Construction~\ref{constr:associated_trivialisation}, and the associated connection, curving and 2-curving from Proposition~\ref{prop:associated_connective_structure}. 
The 1-form $a$ in the trivialisation connection allows us to define a connection on the stable isomorphism $\eta$, as in Lemma~\ref{lemma:bundlegerbemorphisms_w_conn_curv}.
That this is compatible follows from the assumptions on $a$.
Then following the same argument as in the proof of Proposition~\ref{prop:associated_Waldorf_connective_structure}, we get that the diagram (\ref{eq:pre-actionator}) is a diagram of bundle gerbes with connection and curving, and morphisms between these, and that the 2-arrow $\theta$ preserves the required connection.
Again, the coherence equation for $\theta$ holds for free.
\end{construction}

We will rely on this construction in a future paper to give a large family of examples of geometric string structures.


\section{From bigerbes to rigid bundle 2-gerbes}\label{app:bigerbes}

In this appendix we will work in the topological setting, as in \cite{KM19}\footnote{More specifically, the category of compactly-generated Hausdorff spaces}. 
We thus require rigid bundle 2-gerbes that use purely topological notions, so that surjective submersions are replaced by continuous maps that admit local sections around every point in the codomain, and all bundles and isomorphisms of them are now merely topological and continuous, rather than smooth.
We will not need anything to do with connective structures. 

Conversely, one could upgrade the definition in \cite{KM19} to a manifold category (which may include infinite-dimensional manifolds), where now instead of locally split squares in a topological sense, one replaces all the locally split maps by surjective submersions.
We will just refer to maps being `covers', when they are either a locally split continuous surjection, or a surjective submersion, in the respective categories.
Further, we will write everything in terms of $U(1)$-bundles, but one can write everything in terms of line bundles if so desired, as in \cite{KM19}. 
In the definition of bigerbe and bigerbe trivialisation, the line bundles could be replaced by the corresponding $U(1)$-bundles with minimal adjustments, in a completely reversible way---this we will do, referring simply to `bundles' in a mood of deliberate ambiguity. 

Any terms used here that are not defined can be found in \cite{KM19}. 
We will also use the following construction from \cite[Example 2.7]{MRSV} of a degreewise cover $\mu^{-1}(U)_\bullet \to Y^{[\bullet + 1]}$ from a given cover $U\to Y^{[2]}$.

\begin{definition}
Consider the projection maps $\mu_{ij} \colon Y^{[k+1]} \to Y^{[2]}$ for $0\leq i<j \leq k$ giving the $i$ and $j$ entries.
We can assemble these into maps $Y^{[k+1]} \to \left(Y^{[2]}\right)^{k(k+1)/2}$, which we also call $\mu$. 
If $U \to Y^{[2]}$ is a cover then we can pull back $U^{k(k+1)/2}$ to give a cover of $Y^{[k+1]}$, for $k \geq 1$. 
So we have a collection of spaces (or manifolds)
\[
\mu^{-1}(U)_k = 
            \begin{cases} 
                Y, & k = 0\\ 
                U, & k=1\\ 
                \mu^{-1}(U^{k(k+1)/2}), & k >1 
            \end{cases}
\]
and maps $\mu^{-1}(U)_k \to \mu^{-1}(U)_{k-1}$ satisfying the simplicial identities for face maps. 
Note that in general $\mu^{-1}(U)_\bullet$ is only a {\em semi}-simplicial space (or manifold), since we may not have degeneracy maps $\mu^{-1}(U)_0 = Y \to \mu^{-1}(Y)_1 = U$.  
The canonical map $\mu^{-1}(U)_\bullet \to Y^{[\bullet+1]}$ is then a degreewise semi-simplicial cover.  
\end{definition}

\subsection{Bigerbes}

Recall the definition of bigerbe from \cite[Definition 4.5]{KM19}, using the most of the notation introduced there, with the exception of using our $\delta_h$ and $\delta_v$ in place of Kottke and Melrose's $d_1$ and $d_2$ respectively (from \cite[start of \S 4.2]{KM19}).

A bigerbe consists of a \emph{locally split square}
\[
    \xymatrix{
        W \ar[r] \ar[d] & Y_2 \ar[d] \\
        Y_1 \ar[r] & X
    }
\]
---that is, all the maps are covers (and so is $W\to Y_1\times_X Y_2$)---together with a bundle $E\to W^{[2,2]} = (W\times_{Y_2}W)^{[2]}$ (the outer fibre product taken over $Y_1\times_X Y_1$) and trivialisations $M$ of $\delta_h E$ and $s$ of $\delta_v E$, satisfying $\delta_h M=0$ and $\delta_v s = 0$, and also that $\delta_v M = \delta_h s$ as trivialisations of (the canonically isomorphic) $\delta_h \delta_v E \stackrel{!}{\simeq} \delta_v \delta_h E$.

Note that  we are aiming to consider the degreewise cover $Z_\bullet = \mu^{-1}_\bullet(W) \to Y_1^{[\bullet + 1]}$ sitting over the simplicial space arising from the cover $Y_1\to X$. 
There is an inherent asymmetry here, and choosing instead to start from the cover $Y_2\to X$ would give a different rigid bundle 2-gerbe, which could be denoted $r_2(B)$.

\begin{proposition}\label{prop:bigerbe_to_rgb2g}
    Given a bigerbe $B = (E,W,Y_1,Y_2)$ on $X$, there is a rigid bundle 2-gerbe $r_1(B)$ using the cover $Y:= Y_1\to X$, and the semisimplicial space $Z_\bullet :=W^{[\bullet+1,1]}$ with $Z_k = W\times_{Y_2} \cdots \times_{Y_2} W$ ($k+1$ factors).
    The bundle $E\to Z_1 \times_{Y_1^{[2]}} Z_1 = W^{[2,2]}$ gives a bundle gerbe on $Y_1^{[2]}$ as needed, using the trivialisation $s$ as bundle gerbe multiplication. 
    Then the bundle 2-gerbe multiplication is given by the trivialisation $M$.
\end{proposition}

Proposition \ref{prop:bigerbe_to_rgb2g} follows from simply writing down the relevant bigerbe data and noticing that it supplies exactly the relevant data for a rigid bundle 2-gerbe, and the axioms for a bigerbe directly imply the axioms for a rigid bundle 2-gerbe.

\subsection{Trivisalisations}

In the case that the bigerbe $B$ is equipped with a trivialisation (which we will denote $t$), it is \emph{not} automatically the case that the associated rigid bundle 2-gerbe has a rigid trivialisation.
However, from the data of the pair $(B,t)$, we can make a different rigid bundle 2-gerbe, here called $r_1(B,t)$, which \emph{does} have a canonical rigid trivialisation.
These two rigid bundle 2-gerbes are related in a very nice way: there is a morphism  $r_1(B,t)\to r_1(B)$, in the sense arising from Example~\ref{example:rb2g_constructions}.3.
Note that as a result, the associated bundle 2-gerbes (as per Appendix~\ref{app:weak_2-gerbes}) of $r_1(B,t)$ and $r_1(B)$ are stably isomorphic, and hence the bundle 2-gerbe associated to $r_1(B)$ has a stable trivialisation, though we do not give details here.

We note that while $r_1(B)$ is a rigid bundle 2-gerbe where the semisimplicial space $Z_\bullet$ is the \v{C}ech nerve of a cover (the map $W\to Y_2$ in the locally split square), the rigid bundle 2-gerbe $r_1(B,t)$ is \emph{not} of this form, and so is a more general object.
In fact, the corresponding semisimplicial space will not, in general, be a simplicial space at all.

\begin{lemma}\label{lemma:upgradeStableTriv}
let $(E_1,Y)$ and $(E_2,Y)$ be bundle gerbes on $X$, and suppose they are stably isomorphic in the sense that there is a bundle $T\to Y$ such that $E_1\otimes \delta(T) \simeq E_2$ is a bundle gerbe isomorphism. 
Then there is an isomorphism of bundle gerbes $(E_1\times_{Y^{[2]}}T^{[2]},T) \simeq (E_2\times_{Y^{[2]}}T^{[2]},T)$ on $X$.
\end{lemma}
\begin{proof}[Proof (sketch)]
    Consider the composite $T\to Y \to X$ and the map of surjective submersions $\pi\colon T \to Y$ over $X$.
    Pull $T$ back along this map, whereby it trivialises canonically, and also create the two bundle gerbes by pulling back $E_1$ and $E_2$ along $\pi^{[2]}\colon T^{[2]} \to Y^{[2]}$.
    The isomorphism $E_1\otimes \delta(T) \simeq E_2$ then pulls back to $E_1\times_{Y^{[2]}}T^{[2]} \simeq E_2\times_{Y^{[2]}}T^{[2]}$ over $T^{[2]}$.
    It is easily checked this is a bundle gerbe isomorphism.
\end{proof}

We first remind the reader of the definition of a trivialisation of a bigerbe, to fix notation.
Given a bigerbe $B = (E,W,Y_1,Y_2)$, a trivialisation consists of a pair of bundles $Q_1\to W\times_{Y_1}W$ and $Q_2\to W\times_{Y_2}W$, and an isomorphism $E \otimes \delta_v(Q_2)\simeq \delta_h(Q_1)$, that is compatible with the bigerbe structure in a way we will not specify here (see \cite[Definition~4.6]{KM19} for details).

Lemma~\ref{lemma:upgradeStableTriv} is needed because the notion of rigid trivialisation involves an isomorphism that looks like $E\simeq \delta_h(Q_1)$, namely an \emph{isomorphism} of bundle gerbes, whereas $E \otimes \delta_v(Q_2)\simeq \delta_h(Q_1)$ is closer to a \emph{stable} isomorphism between the bundle gerbes involving $E$ and $\delta_h(Q_1)$ (in the restricted sense stated in Lemma~\ref{lemma:upgradeStableTriv}).

Thus a trivialisation $t$ of a bigerbe $B$ can be thought of as giving rise to a `semi-rigid' trivialisation of $r_1(B)$; instead of an isomorphism of the appropriate bundle gerbes on $Y^{[2]}$, it is a stable isomorphism in the sense of Lemma~\ref{lemma:upgradeStableTriv} (compatible with all data, as needed).
Moreover this bundle $Q_2$ in the stable isomorphism is equipped with the data of a simplicial bundle over $Z_\bullet$.
Even if $Q_2$ were trivial, this data can still involve a non-trivial trivialisation of $\delta_h(Q_2)\to Z_2$ (in our notation).
Thus we need to build not just a `new $Z_1$' (which will be via the submersion $Q_2 \to W\times_{Y_2}W = Z_1$), but a whole new simplicial space $Z_\bullet$, in such a way that we do not just canonically trivialise $Q_2$, but also ensure that the induced trivialisation of the trivial bundle $\delta_h(Q_2)$ is the \emph{canonical} trivialisation.

\begin{lemma}\label{lemma:trivialiseSimplicialBundle}
    Let $Z_\bullet$ be a semisimplicial space, and $L\to Z_1$  a \emph{simplicial $U(1)$-bundle}, namely there is a trivialisation $s$ of $\delta L$, such that $\delta s$ is equal to the canonical trivialisation of $\delta^2 L$.
    Then there is a semisimplicial space $N(L)_\bullet$ that in lowest degrees looks like 
    \[
       \cdots d_0^*L\times_{Z_2}d_2^*L \rightthreearrow L \rightrightarrows Z_0
    \]
    with a level-wise cover $\pi_\bullet\colon N(L)_\bullet \to Z_\bullet$.
    This semisimplicial space has the following properties: (a) the simplicial bundle $\pi_1^*L\to L$ has a canonical trivialisation; (b) the trivialisation $s$ pulls up to a trivialisation $\pi_2^*s$ of $\pi_2^*(\delta L)$; and (c) $\pi_2^*s = 1$, the canonical trivialisation of $\delta \pi_1^*L$ coming from the trivialisation of $\pi_1^* L$.
\end{lemma}

The proof of this lemma is a careful but routine calculation.

This result is a strengthening of the fact that the pullback of a principal bundle along its own projection map canonically trivialises. 
In this construction we \emph{also} get that the simplicial bundle structure is trivial, because in principle one could have $L$ already trivial but $s$ a different trivialisation to the given one on $\delta L$ (i.e.\ a trivialisation that is not identically $1$).

The rigid bundle 2-gerbe $r_1(B,t)$ uses the cover $Y_1\to X$, the semisimplicial space $Z_\bullet = N(Q_2)_\bullet \to W^{[\bullet+1,1]} \to Y_1^{[\bullet+1]}$ from Lemma~\ref{lemma:trivialiseSimplicialBundle}, using the simplicial bundle structure of $Q_2$ over $W^{[\bullet+1,1]}$. 
We have to pull back $E$ along $Z_1^{[2]} = Q_2\times_{Y_1^{[2]}} Q_2\to W^{[2,2]}$ to get the bundle $E'$ on $Z_1^{[2]}$ for the rigid bundle 2-gerbe, but then the bundle $Q_1\to Z_0^{[2]} = W^{[1,2]}$ serves as the data for a rigid trivialisation. 
The required isomorphism of bundles $E'\simeq \delta_h(Q_1)$ for the rigid trivialisation uses the fact that the pullback of $\delta_v(Q_2)$ to $Q_2\times_{Y_1^{[2]}} Q_2$ trivialises canonically in a compatible way so that the pullback of $E\otimes \delta_v(Q_2) \simeq \delta_h(Q_1)$ to $Z_1^{[2]}$ is exactly what we need.
This data then satisfies condition (3) in Definition~\ref{def:vsb2g_triv}, because of the property (c) in Lemma~\ref{lemma:trivialiseSimplicialBundle}.

Thus we have:

\begin{proposition}
    Given a bigerbe $B$ on $X$ equipped with a trivialisation $t$ there is a rigidly trivialised rigid bundle 2-gerbe $r_1(B,t)$ as just described, and a morphism $r_1(B,t)\to r_1(B)$ of rigid bundle 2-gerbes over $X$ that is the identity on $Y_1^{[\bullet+1]}$.
\end{proposition}

Note that if we had a trivialisation of a bigerbe as in example (i) in the discussion after \cite[Defintion~4.6]{KM19}, where \emph{only} the data of $Q_1$ is given, then this is the same as $Q_2$ being trivial, and also the trivialisation of $\delta_h(Q_2)$ being the canonical one.
In this case, the rigid bundle 2-gerbe $r_1(B)$ would already admit a rigid trivialisation using the data of $Q_1$ (and the attendant isomorphisms) only.

\begin{remark}
The above process works to make a rigidly trivialised rigid bundle 2-gerbe analogous to $r_2(B)$, say $r_2(B,t)$, keeping the data of the bundle $Q_2$ and instead passing to a cover that trivialises $Q_1$ etc.
\end{remark}


\section{Table of specific differential forms}\label{app:data}

\renewcommand{\arraystretch}{1.4}
\begin{tabular}{p{0.3\linewidth} p{0.53\linewidth} p{0.1\linewidth}}
      Form  & Description & Page \\ 
      \hline
      $\Theta_g = \lMC{g} \in \Omega^1(G,\fg)$ & left-invariant Maurer--Cartan form &\\
      $\Thetahat_g = \rMC{g} \in \Omega^1(G,\fg)$ & right-invariant Maurer--Cartan form & \\
      $\mu\in \Omega^1(\widehat{\Omega G})$ & connection on the central extension of $\Omega G$ & \pageref{form:central extension connection}\\
      $R\in \Omega^2(\Omega G)$ & curvature of $\mu$, left-invariant 2-form  & \pageref{form:central extension curvature}\\
      $\nu \in \Omega^1(\Omega G\times \Omega G)$ & measures the failure of the connection $\mu$ to be multiplicative on the nose &\pageref{form:nu}\\
      $\phi\colon PG\to \Omega \fg$&the function $p\mapsto \phifun{p}$ (a `Higgs field')&\\
      $\phihat\colon PG\to \Omega \fg$ & the function $p\mapsto \phihatfun{p} = \Ad_p(\phifun{p})$ &\\
      $B \in \Omega^2(PG)$ & curving on the basic gerbe on $G$ & \pageref{form: basic gerbe curving}\\
      $\epsilon \in \Omega^1(PG\times \Omega G)$& 1-form used to fix the connection $\mu$ into a \emph{bundle gerbe} connection,& \pageref{form:basic gerbe connection fix-up form}\\
      $\nabla = \mu - \pi^*\epsilon$ \newline\phantom{xxxx} $\in \Omega^1(PG\times \Omegahat{G})$ &bundle gerbe connection on the basic gerbe &  \pageref{form:basic gerbe connection}\\
      $\kappa \in \Omega^2(G\times G)$ & measures the failure of the curvature of the basic gerbe to be a multiplicative 3-form& \pageref{form:kappa}\\
      $\rho \in \Omega^1(PG\times \Omega G)$ & measures the failure of the crossed module action to preserve the connection $\mu$ & \pageref{form:xmod form rho} \\
      $A\in \Omega^1(Q,\g)$ & connection on the principal $G$-bundle $Q$ &\\
      $\beta_A \in \Omega^2(Q\times PG)$ &  curving on the rigid Chern--Simons bundle 2-gerbe of $Q$& \pageref{form:CS2-gerbe curving}\\
      $\alpha \in \Omega^1(Q\times PG\times PG)$ &strong trivialisation connection witnessing Chern--Simons bundle 2-gerbe product&\pageref{form:CS2-gerbe alpha}\\
      $-CS(A)\in \Omega^3(Q)$ &2-curving on the rigid Chern--Simons bundle 2-gerbe of $Q$ & \pageref{form:CS2-gerbe 2-curving}\\
      \hline
    \end{tabular}
    \label{tab:my_label}

\bibliographystyle{amsalpha}
\bibliography{rb2g}

\newcommand{\etalchar}[1]{$^{#1}$}
\providecommand{\bysame}{\leavevmode\hbox to3em{\hrulefill}\thinspace}
\providecommand{\MR}{\relax\ifhmode\unskip\space\fi MR }
\providecommand{\MRhref}[2]{%
  \href{http://www.ams.org/mathscinet-getitem?mr=#1}{#2}
}
\providecommand{\href}[2]{#2}
\begin{thebibliography}{MRSV17}

\bibitem[ACJ05]{ACJ}
Paolo Aschieri, Luigi Cantini, and Branislav Jur\v{c}o, \emph{Nonabelian bundle gerbes, their differential geometry and gauge theory}, Comm. Math. Phys. \textbf{254} (2005), no.~2, 367--400, \href{https://arxiv.org/abs/hep-th/0312154}{arXiv:hep-th/0312154}.

\bibitem[Bac81]{Back_81}
Allen Back, \emph{Rational {P}ontryagin classes and {K}illing forms}, J. Differential Geometry \textbf{16} (1981), 191--193.

\bibitem[Bar06]{Bartels}
T.~Bartels, \emph{{Higher gauge theory I: 2-Bundles}}, Ph.D. thesis, University of California Riverside, 2006, \href{https://arxiv.org/abs/math/0410328}{arXiv:math.CT/0410328}.

\bibitem[BC04]{HDA6}
John~C. Baez and Alissa~S. Crans, \emph{Higher-{D}imensional {A}lgebra {VI}: {L}ie 2-algebras}, Theory Appl. Categ. \textbf{12} (2004), 492--538, \href{https://arxiv.org/abs/math/0307263}{arXiv:math/0307263}.

\bibitem[BM94]{BM}
J.-L. Brylinski and D.~A. McLaughlin, \emph{The geometry of degree-four characteristic classes and of line bundles on loop spaces. {I}}, Duke Math. J. \textbf{75} (1994), no.~3, 603--638.

\bibitem[BM05]{B_M}
Lawrence Breen and William Messing, \emph{Differential geometry of gerbes}, Adv. Math. \textbf{198} (2005), no.~2, 732--846, \href{https://arxiv.org/abs/math/0106083}{arXiv:math/0106083}.

\bibitem[BPST75]{BPST}
Alexander Belavin, Alexander Polyakov, Albert Schwarz, and Yu.~S. Typkin, \emph{Pseudoparticle solutions of the yang--mills equations}, Physics Letters B \textbf{59} (1975), no.~1, 85--87.

\bibitem[Bre94]{Breen94}
Lawrence Breen, \emph{On the classification of 2-gerbes and 2-stacks}, Ast\'erisque, no. 225, Soci\'et\'e math\'ematique de France, 1994, \url{https://www.numdam.org/issues/AST_1994__225__1_0/}.

\bibitem[BS07]{BS04}
John~C. Baez and Urs Schreiber, \emph{Higher gauge theory}, Categories in algebra, geometry and mathematical physics ({S}treetfest) (A.~Davydov, ed.), Contemp. Math., vol. 431, Amer. Math. Soc., 2007, \href{https://arxiv.org/abs/hep-th/0511710}{arXiv:hep-th/0511710}.

\bibitem[BSCS07]{BSCS}
John~C. Baez, D.~Stevenson, Alissa~S. Crans, and U.~Schreiber, \emph{From loop groups to 2-groups}, Homology, Homotopy and Applications \textbf{9} (2007), no.~2, 101--135, \href{https://arxiv.org/abs/math/0504123}{arXiv:math/0504123}.

\bibitem[CJM{\etalchar{+}}05]{CJMSW}
Alan Carey, Stuart Johnson, Michael~K. Murray, Danny Stevenson, and Bai-Ling Wang, \emph{Bundle gerbes for {Chern}--{Simons} and {Wess}--{Zumino}--{Witten} theories}, Comm. Math. Phys. \textbf{259} (2005), 577--613, \href{https://arxiv.org/abs/math/0410013}{arXiv:math.DG/0410013}.

\bibitem[CM86]{Carey-Murray_86}
A.~L. Carey and M.~K. Murray, \emph{Holonomy and the {W}ess--{Z}umino term}, Letters in Mathematical Physics \textbf{12} (1986), no.~4, 323--327.

\bibitem[CMW97]{Carey-Murray-Wang_97}
Alan Carey, Michael~K. Murray, and Bai-Ling Wang, \emph{Higher bundle gerbes and cohomology classes in gauge theories}, J. Geom. Phys. \textbf{21} (1997), no.~2, 183--197, \href{https://arxiv.org/abs/hep-th/9511169}{arXiv:hep-th/9511169}.

\bibitem[CW08]{CareyWang_08}
Alan~L. Carey and Bai-Ling Wang, \emph{Fusion of symmetric {D}-branes and {V}erlinde rings}, Comm. Math. Phys. \textbf{277} (2008), no.~3, 577--625, \href{https://arxiv.org/abs/math-ph/0505040}{arXiv:math-ph/0505040}.

\bibitem[Fre95]{Freed_95}
Daniel~S. Freed, \emph{Classical {C}hern--{S}imons theory. {I}}, Adv. Math. \textbf{113} (1995), no.~2, 237--303, \href{https://arxiv.org/abs/hep-th/9206021}{arXiv:hep-th/9206021}.

\bibitem[FSS12]{FiorenzaSchreiberStasheff_12}
Domenico Fiorenza, Urs Schreiber, and Jim Stasheff, \emph{{\v{C}}ech cocycles for differential characteristic classes: an {$\infty$}-{L}ie theoretic construction}, Adv. Theor. Math. Phys. \textbf{16} (2012), no.~1, 149--250, \href{https://arxiv.org/abs/1011.4735}{arXiv:1011.4735}.

\bibitem[FSS14]{FiorenzaSatiSchreiber_14}
Domenico Fiorenza, Hisham Sati, and Urs Schreiber, \emph{Multiple {M}5-branes, string 2-connections, and 7d nonabelian {C}hern-{S}imons theory}, Adv. Theor. Math. Phys. \textbf{18} (2014), no.~2, 229--321, \href{https://arxiv.org/abs/1201.5277}{arXiv:1201.5277}.

\bibitem[FSS15a]{Fiorenza-Sati-Schreiber_15}
\bysame, \emph{A higher stacky perspective on {C}hern--{S}imons theory}, Mathematical Aspects of Quantum Field Theories (Damien Calaque and Thomas Strobl, eds.), Springer International Publishing, Cham, 2015, \href{https://arxiv.org/abs/1301.2580}{arXiv:1301.2580}, pp.~153--211.

\bibitem[FSS15b]{Fiorenza-Sati-Schreiber_15b}
\bysame, \emph{Super-{L}ie {$n$}-algebra extensions, higher {WZW} models and super-{$p$}-branes with tensor multiplet fields}, International Journal of Geometric Methods in Modern Physics \textbf{12} (2015), no.~02, 1550018, \href{https://arxiv.org/abs/1308.5264}{arXiv:1308.5264}.

\bibitem[FSS20]{Fiorenza-Sati-Schreiber_20}
\bysame, \emph{Twisted cohomotopy implies {M}-theory anomaly cancellation on 8-manifolds}, Communications in Mathematical Physics \textbf{377} (2020), no.~3, 1961--2025, \href{https://arxiv.org/abs/1904.10207}{arXiv:1904.10207}.

\bibitem[FSS21a]{Fiorenza-Sati-Schreiber_21b}
\bysame, \emph{Twisted cohomotopy implies level quantization of the full 6d {W}ess--{Z}umino term of the {M5}-brane}, Communications in Mathematical Physics \textbf{384} (2021), no.~1, 403--432, \href{https://arxiv.org/abs/1906.07417}{arXiv:1906.07417}.

\bibitem[FSS21b]{Fiorenza-Sati-Schreiber_21a}
Domenico Fiorenza, Hisham Sati, and Urs Schreiber, \emph{Twisted cohomotopy implies twisted string structure on {M5}-branes}, Journal of Mathematical Physics \textbf{62} (2021), no.~4, 042301, \href{https://arxiv.org/abs/2002.11093}{arXiv:2002.11093}.

\bibitem[GHV72]{GHV}
Werner Greub, Stephen Halperin, and Ray Vanstone, \emph{Connections, curvature, and cohomology. {V}ol. {I}: {D}e {R}ham cohomology of manifolds and vector bundles}, Pure and Applied Mathematics, Vol. 47, Academic Press, New York-London, 1972.

\bibitem[Joh03]{JohnsonPhD}
Stuart Johnson, \emph{Constructions with bundle gerbes}, Ph.D. thesis, Adelaide University, Department of Pure Mathematics, 2003, \href{https://arxiv.org/abs/math/0312175}{arXiv:math.DG/0312175}.

\bibitem[JY21]{Johnson-Yau}
Niles Johnson and Donald Yau, \emph{2-{D}imensional {C}ategories}, Oxford University Press, 2021, \href{https://arxiv.org/abs/2002.06055}{arXiv:2002.06055}.

\bibitem[Kil87]{Killingback_87}
T.~P. Killingback, \emph{World-sheet anomalies and loop geometry}, Nuclear Phys. B \textbf{288} (1987), 578--588.

\bibitem[KM21]{KM19}
Chris Kottke and Richard~B. Melrose, \emph{Bigerbes}, Algebr. Geom. Topol. \textbf{21} (2021), no.~7, 3335--3399, \href{https://arxiv.org/abs/1905.03081}{arXiv:1905.03081}.

\bibitem[KN63]{Kobayashi-Nomizu}
Shoshichi Kobayashi and Katsumi Nomizu, \emph{Foundations of differential geometry. {V}ol {I}}, Interscience Publishers, New York-London, 1963.

\bibitem[KS74]{Kelly-Steet_74}
G.~M. Kelly and Ross Street, \emph{Review of the elements of {$2$}-categories}, Category {S}eminar ({P}roc. {S}em., {S}ydney, 1972/1973), Lecture Notes in Math, vol. 420, 1974, pp.~75--103.

\bibitem[KS20]{Kim-Saemann_20}
Hyungrok Kim and Christian Saemann, \emph{Adjusted parallel transport for higher gauge theories}, J. Phys. A \textbf{53} (2020), no.~44, 445206, 52, \href{https://arxiv.org/abs/1911.06390}{arXiv:1911.06390}.

\bibitem[LM95]{Lada-Markl_95}
Tom Lada and Martin Markl, \emph{Strongly homotopy {L}ie algebras}, Comm. Algebra \textbf{23} (1995), no.~6, 2147--2161, \href{https://arxiv.org/abs/hep-th/9406095}{arXiv:hep-th/9406095}.

\bibitem[McL92]{McL_92}
Dennis~A. McLaughlin, \emph{Orientation and string structures on loop space}, Pacific J. Math. \textbf{155} (1992), no.~1, 143--156.

\bibitem[Mic18]{Mickler_18}
Ryan Mickler, \emph{Localization for {C}hern--{S}imons on circle bundles via loop groups}, J. Geom. Phys. \textbf{132} (2018), 257--281, \href{https://arxiv.org/abs/1507.01626}{arXiv:1507.01626}.

\bibitem[MRSV17]{MRSV}
Michael~K. Murray, David~Michael Roberts, Danny Stevenson, and Raymond~F. Vozzo, \emph{Equivariant bundle gerbes}, Adv. Theor. Math. Phys. \textbf{21} (2017), no.~4, 921--975, \href{https://arxiv.org/abs/1506.07931}{arXiv:1506.07931}.

\bibitem[MS00]{MurSte}
Michael~K. Murray and Daniel Stevenson, \emph{Bundle gerbes: stable isomorphism and local theory}, J. London Math. Soc. (2) \textbf{62} (2000), no.~3, 925--937, \href{https://arxiv.org/abs/math/9908135}{arXiv:math.DG/9908135}.

\bibitem[MS03]{MS03}
\bysame, \emph{Higgs fields, bundle gerbes and string structures}, Comm. Math. Phys. \textbf{243} (2003), no.~3, 541--555, \href{https://arxiv.org/abs/math/0106179}{arXiv:math/0106179}.

\bibitem[Mur96]{Mur}
Michael~K. Murray, \emph{Bundle gerbes}, J. London Math. Soc. \textbf{54} (1996), no.~2, 403--416, \href{https://arxiv.org/abs/math/9407015}{arXiv:dg-ga/9407015}.

\bibitem[Mur09]{Murray_09}
M.~K. Murray, \emph{An introduction to bundle gerbes}, The many facets of geometry, a tribute to {N}igel {H}itchin (O.~Garcia-Prada, J.-P. Bourguignon, and S.~Salamon, eds.), Oxford University Press, 2009, \href{https://arxiv.org/abs/math/0712.1651}{arXiv:math.DG/0712.1651}.

\bibitem[NS11]{Nikolaus-Schweigert_11}
Thomas Nikolaus and Christoph Schweigert, \emph{Equivariance in higher geometry}, Adv. Math. \textbf{226} (2011), no.~4, 3367--3408, \href{https://arxiv.org/abs/1004.4558}{arXiv:1004.4558}.

\bibitem[NW13]{Nikolaus-Waldorf_13}
Thomas Nikolaus and Konrad Waldorf, \emph{Four equivalent versions of nonabelian gerbes}, Pacific J. Math. \textbf{264} (2013), no.~2, 355--419, \href{https://arxiv.org/abs/1103.4815}{arXiv:1103.4815}.

\bibitem[PS86]{PS}
Andrew Pressley and Graeme Segal, \emph{Loop groups}, Oxford Mathematical Monographs, The Clarendon Press, Oxford University Press, New York, 1986, Oxford Science Publications.

\bibitem[Red11]{Redden}
Corbett Redden, \emph{String structures and canonical 3-forms}, Pacific J. Math. \textbf{249} (2011), no.~2, 447--484, \href{https://arxiv.org/abs/0912.2086}{arXiv:0912.2086}.

\bibitem[Rob14]{Roberts_14}
David~Michael Roberts, \emph{Explicit string bundles}, Notes from a talk delivered at Heriot-Watt University on 26 June 2014 at the Workshop on Higher Gauge Theory and Higher Quantization, 2014, \href{https://arxiv.org/abs/2203.04544}{arXiv:2203.04544}.

\bibitem[RSW22]{RSW_22}
Dominik Rist, Christian Saemann, and Martin Wolf, \emph{Explicit non-abelian gerbes with connections}, \href{https://arxiv.org/abs/2203.00092}{arXiv:2203.00092}, 2022.

\bibitem[RVa]{rb2g_II}
David~Michael Roberts and Raymond~F. Vozzo, \emph{Rigid models for 2-gerbes {II}: Multiplicative classifying maps and the universal smooth rigid bundle 2-gerbe}, In preparation.

\bibitem[RVb]{rb2g_III}
\bysame, \emph{Rigid models for 2-gerbes {III}: {E}xplicit geometric string structures on homogeneous spaces}, In preparation.

\bibitem[Sat10]{Sati_10}
Hisham Sati, \emph{Geometric and topological structures related to {M}-branes}, Superstrings, geometry, topology, and {$C^\ast$}-algebras, Proc. Sympos. Pure Math., vol.~81, Amer. Math. Soc., Providence, RI, 2010, \href{https://arxiv.org/abs/1001.5020}{arXiv:1001.5020}, pp.~181--236.

\bibitem[Sch23]{Schmeding_book}
Alexander Schmeding, \emph{An introduction to infinite-dimensional differential geometry}, Cambridge studies in advanced mathematics, no. 202, Cambridge University Press, 2023, \href{https://arxiv.org/abs/2112.08114}{arXiv:2112.08114}.

\bibitem[SSS12]{Sati-Schreiber-Stasheff_12}
Hisham Sati, Urs Schreiber, and Jim Stasheff, \emph{Twisted differential string and fivebrane structures}, Communications in Mathematical Physics \textbf{315} (2012), no.~1, 169--213, \href{https://arxiv.org/abs/0910.4001}{arXiv:0910.4001}.

\bibitem[ST04]{Stolz-Teichner}
Stephan Stolz and Peter Teichner, \emph{What is an elliptic object?}, Topology, geometry and quantum field theory, London Math. Soc. Lecture Note Ser., vol. 308, Cambridge Univ. Press, Cambridge, 2004, pp.~247--343.

\bibitem[Ste00]{StevPhD}
Danny Stevenson, \emph{The geometry of bundle gerbes}, Ph.D. thesis, Adelaide University, Department of Pure Mathematics, 2000, \href{http://arxiv.org/abs/math/0004117}{arXiv:math.DG/0004117}.

\bibitem[Ste04]{Ste}
Daniel Stevenson, \emph{Bundle 2-gerbes}, Proc. Lond. Math. Soc. \textbf{88} (2004), no.~3, 405--435, \href{https://arxiv.org/abs/math/0106018}{arXiv:math/0106018}.

\bibitem[SW13]{Schreiber-Waldorf_13}
Urs Schreiber and Konrad Waldorf, \emph{Connections on non-abelian gerbes and their holonomy}, Theory Appl. Categ. \textbf{28} (2013), 476--540, \href{https://arxiv.org/abs/0808.1923}{arXiv:0808.1923}.

\bibitem[vH24]{vanHelden_21}
Kevin~S. van Helden, \emph{Classification of 2-term {$L_\infty$-algebras}}, J. Homotopy Relat. Struct \textbf{19} (2024), 541--560, \href{https://arxiv.org/abs/2109.10202}{arXiv:2109.10202}.

\bibitem[Wal07]{Waldorf_07}
Konrad Waldorf, \emph{More morphisms between bundle gerbes}, Theory and Appl. Categ. \textbf{18} (2007), no.~9, 240--273, \href{https://arxiv.org/abs/math/0702652}{arXiv:math/0702652}.

\bibitem[Wal13]{Wal}
Konrad Waldorf, \emph{String connections and {C}hern--{S}imons theory}, Trans. Amer. Math. Soc. \textbf{365} (2013), no.~8, 4393--4432, \href{https://arxiv.org/abs/0906.0117}{arXiv:0906.0117}.

\bibitem[Wal18]{Waldorf_18}
\bysame, \emph{A global perspective to connections on principal 2-bundles}, Forum Math. \textbf{30} (2018), no.~4, 809--843, \href{https://arxiv.org/abs/1608.00401}{arXiv:1608.00401}.

\bibitem[Wit85]{Witten_85}
Edward Witten, \emph{Global anomalies in string theory}, Symposium on anomalies, geometry, topology (Chicago, Ill., 1985), World Sci. Publishing, Singapore, 1985, pp.~61--99.

\bibitem[WZ85]{Wang-Ziller_85}
McKenzie~Y. Wang and Wolfgang Ziller, \emph{On normal homogeneous {Einstein} manifolds}, Annales scientifiques de l'{\'E}cole Normale Sup{\'e}rieure, S{\'e}rie 4 \textbf{18} (1985), no.~4, 563--633.

\end{thebibliography}

\end{document}